\documentclass[aap]{imsart}
\makeatletter\def\journal@name{}\makeatother

\RequirePackage{amsthm,amsmath,amsfonts,amssymb,mathtools}
\RequirePackage[numbers]{natbib}
\RequirePackage[colorlinks=true,citecolor=blue,urlcolor=blue,linkcolor=blue,hypertexnames=false,pagebackref=true]{hyperref}
\renewcommand*{\backref}[1]{}
\renewcommand*{\backrefalt}[4]{%
  \ifcase #1\or (Cited on page~#2.)\else (Cited on pages~#2.)\fi}
\RequirePackage{microtype}
\RequirePackage{booktabs}
\RequirePackage{float}
\RequirePackage{needspace}
\RequirePackage{tikz}
\RequirePackage{pgfplots}
\usepgfplotslibrary{groupplots}
\usepgfplotslibrary{fillbetween}
\pgfplotsset{compat=1.18}

\startlocaldefs
\numberwithin{equation}{section}
\theoremstyle{plain}
\newtheorem{theorem}{Theorem}[section]
\newtheorem{proposition}[theorem]{Proposition}
\newtheorem{lemma}[theorem]{Lemma}
\newtheorem{corollary}[theorem]{Corollary}
\newtheorem{question}{Question}[section]
\newtheorem{assumption}{Assumption}
\theoremstyle{definition}

\newtheorem{remark}[theorem]{Remark}

\newcommand{\R}{\mathbb R}
\newcommand{\E}{\mathbb E}
\newcommand{\Pp}{\mathbb P}
\newcommand{\cM}{\mathcal M}
\newcommand{\cN}{\mathcal N}
\newcommand{\cU}{\mathcal U}
\newcommand{\cV}{\mathcal V}
\newcommand{\TV}{\mathrm{TV}}
\newcommand{\proj}{\operatorname{proj}}
\newcommand{\Id}{\mathrm{Id}}
\newcommand{\Law}{\operatorname{Law}}
\newcommand{\norm}[1]{\lVert #1\rVert}
\newcommand{\abs}[1]{\lvert #1\rvert}
\newcommand{\ip}[2]{\langle #1,#2\rangle}
\newcommand{\eps}{\varepsilon}
\newcommand{\rhoq}{\rho_{\mathrm q}}
\newcommand{\supp}{\operatorname{supp}}
\endlocaldefs

\begin{document}

\begin{frontmatter}
\title{Provable Non-Acceleration of Standard Strang\\ Splittings of Kinetic Langevin Dynamics}
\runtitle{Non-acceleration of Strang Langevin splittings}

\begin{aug}
\author[A]{\fnms{Nawaf}~\snm{Bou-Rabee}\ead[label=e1]{nawaf.bourabee@rutgers.edu}}
\address[A]{Department of Mathematical Sciences,
Rutgers University, Camden, NJ 08102, USA;
and Visiting Scholar, Center for Computational Mathematics,
Flatiron Institute, 162 5th Ave, New York, NY 10010, USA.\printead[presep={; }]{e1}}
\end{aug}

\begin{abstract}
The OBABO and BAOAB schemes and the other standard Strang splittings
of kinetic (underdamped) Langevin dynamics are widely used Markov
chain Monte Carlo algorithms. Under a suitable friction scaling, the
underlying diffusion relaxes on a ballistic time scale, suggesting
that these discretizations, suitably tuned, sample targets with
condition number $\kappa$ in $O(\sqrt{\kappa})$ iterations.
We prove that no fixed choice of step size and friction, based only
on the curvature bounds and the dimension, achieves this
acceleration: total variation mixing time lower bounds for OBABO
show that ballistic cold-start mixing fails uniformly over the
smooth strongly convex class, and the lower bounds extend, with the
same orders, to BAOAB and the other four Strang splittings.
The proof transfers non-acceleration from optimization to sampling.
Eliminating velocity gives an exact noisy heavy-ball recursion, and
by an extension of the non-acceleration theorem of Goujaud, Taylor
and Dieuleveut, for every tuning either some Gaussian target has a
mode with relaxation time at least of order $\kappa$, or an
attracting cycle exists on a smooth potential; dilating such a
potential as $U_R(x)=R^2U(x/R)$ preserves its curvature bounds and
produces metastability for a number of steps exponential in $R^2$,
from an initial state at Wasserstein distance $O(R)$ from
equilibrium. Using contraction estimates of Leimkuhler, Paulin and
Whalley and a Wasserstein-to-total-variation regularization
estimate, we prove a complementary upper bound of $O(\kappa)$ steps,
up to logarithmic factors, for a fixed-parameter OBABO tuning.
Hence, among fixed-parameter OBABO tunings, the optimal
condition-number dependence of cold-start total variation mixing
over this class is linear, up to logarithmic factors. A direct
Gaussian calculation also rules out fixed-parameter acceleration for
the left-endpoint exponential integrator.
\end{abstract}


\end{frontmatter}

\section{Introduction}

\subsection{Motivation}

Let $U:\R^d\to\R$ be differentiable, let $(W_t)_{t\ge0}$ be a
standard Brownian motion on $\R^d$, and let $\gamma>0$ be a friction
parameter. Consider the kinetic, or underdamped, Langevin dynamics
\begin{equation}\label{eq:KLD}
 dX_t=V_t\,dt,\qquad
 dV_t=-\nabla U(X_t)\,dt-\gamma V_t\,dt+\sqrt{2\gamma}\,dW_t.
\end{equation}
For the strongly convex potentials with Lipschitz gradient considered in this paper, \eqref{eq:KLD} admits a unique non-explosive strong solution and has invariant probability measure
\begin{equation}\label{eq:gibbs}
 \mu(dx,dv)=Z^{-1}e^{-U(x)-\abs{v}^2/2}\,dx\,dv,
\end{equation}
where
$Z=\int_{\R^d\times\R^d}e^{-U(x)-\abs v^2/2}\,dx\,dv$.

The convergence rate of \eqref{eq:KLD} to $\mu$ provides the
continuous-time benchmark for the splitting schemes studied below.
To quantify this benchmark, suppose that $U$ satisfies the
convexity, regularity, and growth conditions in
\citep[Assumptions~2.1 and~2.2]{Lu2026}, and that the logarithmic
Sobolev constant appearing in Assumption~2.1 is $K$. At the friction
scale $\gamma\asymp\sqrt K$, the law of $(X_t,V_t)$ converges to
$\mu$ in relative entropy at a rate of order $\sqrt K$, for every
initial law of finite relative entropy with respect to $\mu$
\citep[Theorem~2.3]{Lu2026}; equivalently, the relaxation time is of
order $K^{-1/2}$. By comparison, under the same logarithmic Sobolev inequality,
overdamped Langevin dynamics converges in relative entropy at a rate
of order $K$, corresponding to a relaxation time of order $K^{-1}$
\citep[Section~5.2]{BakryGentilLedoux2014}.

The improvement
from $K^{-1}$ to $K^{-1/2}$ is the square-root, or ballistic,
acceleration. Such square-root gains are fundamental in convex optimization and
numerical linear algebra. Nesterov's accelerated gradient method
improves the worst-case dependence of the iteration complexity of
gradient descent on the condition number from linear to square root
\citep[Theorem~2.2.3]{Nesterov2004}, and this square-root dependence
is optimal for first-order methods whose iterates lie in the initial point plus
the linear span of the previously computed gradients
\citep[Assumption~2.1.4 and Theorem~2.1.13]{Nesterov2004}. The
conjugate gradient method of Hestenes and Stiefel
\citep{HestenesStiefel1952} achieves the
analogous improvement over steepest descent for symmetric positive
definite linear systems \citep[Theorem~38.5]{TrefethenBau1997}.

The OBABO splitting is a standard splitting integrator of weak
order two for \eqref{eq:KLD}, composed of exact
Ornstein-Uhlenbeck ($\mathrm O$), velocity-kick ($\mathrm B$), and
position-drift ($\mathrm A$) substeps in the order indicated by its
name. It goes back to the Langevin thermostat of Bussi
and Parrinello \citep{BussiParrinello2007} and belongs to the family
of Strang splittings analyzed systematically by Leimkuhler and
Matthews \citep{LeimkuhlerMatthews2013}. Throughout the paper, we
assume that $U$ is $K$-strongly convex with $L$-Lipschitz gradient.
We write $\kappa=L/K$ for the condition number. The formal class used in
our results appears in \eqref{eq:classUkappa} below, normalized to
$K=1$ and $L=\kappa$.

By the Bakry-\'Emery criterion
\citep[Section~5.7]{BakryGentilLedoux2014}, the lower curvature bound
$K$ gives a logarithmic Sobolev inequality with constant $K$. For
potentials that also satisfy the remaining hypotheses of
\citep{Lu2026}, this yields the benchmark described above. The
upper curvature bound $L$ limits the step size: stability of OBABO
on every quadratic potential with Hessian eigenvalues in $[K,L]$
requires $h<2/\sqrt L$; see \eqref{eq:stabletuningintro} below.

At the largest stable step-size scale $L^{-1/2}$, simulating the
kinetic Langevin diffusion \eqref{eq:KLD} with OBABO over the
ballistic time interval $K^{-1/2}$ requires order
$K^{-1/2}/L^{-1/2}=\sqrt\kappa$ steps. This suggests a ballistic
mixing bound of order $\sqrt\kappa$. Available estimates for OBABO
over the full smooth strongly convex class are slower. The synchronous-coupling estimates of Leimkuhler, Paulin and
Whalley \citep{LeimkuhlerPaulinWhalley2024} give a Wasserstein
contraction rate of order $\kappa^{-1}$ per step. Combined with the
Wasserstein-to-total-variation regularization of
\citep{BouRabeeCoxSchieven2026}, they give a total variation mixing
time of order $\kappa$, up to logarithmic factors
(Theorem~\ref{thm:upper} below). By contrast, on Gaussian targets the
optimally tuned OBABO chain has
cold-start asymptotic total variation contraction factor
$(\sqrt\kappa-1)/(\sqrt\kappa+1)$, corresponding to a relaxation scale
of order $\sqrt\kappa$ (Corollary~\ref{cor:optimalGaussian}).

The gap between the order-$\kappa$ mixing time guarantee over the
full smooth strongly convex class and the order-$\sqrt\kappa$
relaxation scale on Gaussian targets leads to the following
question.

\begin{question}\label{q:acceleration}
Can the OBABO step size and friction be tuned using only $(K,L,d)$
such that, for every potential in the smooth strongly
convex class and every fixed deterministic initial state, the resulting chain has
a total variation mixing time of order $\sqrt\kappa$, up to
logarithmic factors?
\end{question}

This is one precise form of a question raised by Ma, Chatterji,
Cheng, Flammarion, Bartlett and Jordan: whether
gradient-based sampling admits a provable analogue of Nesterov
acceleration \citep{MaChatterjiChengFlammarionBartlettJordan2021}.
They interpret \eqref{eq:KLD} as accelerated gradient
descent in relative entropy, and they note that mixing time lower
bounds by which acceleration could be characterized have been
lacking. A negative answer requires lower bounds: an upper bound of any
order leaves open the possibility that the chain mixes faster than
its analysis shows. This paper establishes such lower bounds for
OBABO and the other standard Strang splittings considered below.

\subsection{Contributions}

Corollary~\ref{cor:sharp} below rules out, for every numerically stable fixed-parameter OBABO tuning, every cold-start total variation bound over the full smooth strongly convex class whose convergence rate is of larger order than $\kappa^{-1}$, with polynomial dependence on $\kappa$ and the initial Wasserstein distance; in particular, no such bound holds with the ballistic rate of order $\kappa^{-1/2}$. Here ``fixed-parameter'' means that the step size and friction may depend on the curvature bounds $(K,L)$ and the dimension, but not on the particular target, the state, the iteration number, or the realized noise. Together with the order-$\kappa$ upper bound in Theorem~\ref{thm:upper}, this shows that order $\kappa$, up to a logarithmic factor, is optimal among cold-start exponential bounds of this form.

The proofs rest on a dichotomy between a quadratic case and a cycling case (Theorem~\ref{thm:dichotomy}). These cases mirror the deterministic heavy-ball dichotomy of Goujaud, Taylor and Dieuleveut \citep{GoujaudTaylorDieuleveut2025}. The dichotomy states that, for every heavy-ball parameter pair, either the worst-case asymptotic convergence factor of heavy ball over quadratic objectives with curvature in $[K,L]$ is nonaccelerated, or heavy ball fails to converge on some strongly convex $C^1$ objective with Lipschitz gradient. Above an explicit condition-number threshold, Appendix~B of that paper realizes the nonconvergence case through a roots-of-unity cycle. Eliminating velocity turns the OBABO position process into an exact noisy heavy-ball method and allows Theorem~\ref{thm:dichotomy} to transfer this dichotomy to the OBABO chain.

In the quadratic case, the OBABO chain on a Gaussian target is a Gaussian autoregressive process. Lemma~\ref{lem:gaussianTVrate} gives a sharp asymptotic total variation decay rate for such processes: the spectral radius is an upper bound from every initial state and is attained for suitable initial states. Choose a curvature at which the heavy-ball contraction factor is maximal. On the corresponding Gaussian target, the lemma produces two initial states whose asymptotic total variation contraction factor equals this maximum and is therefore nonaccelerated. Lemma~\ref{lem:gaussianTVrate} also yields the Gaussian non-acceleration result for the left-endpoint exponential integrator (Appendix~\ref{app:exponential-integrator}).

In the cycling case, Lemma~\ref{lem:metastability} gives exponentially long confinement and separation between two chains initialized one step apart on the cycle after dilation by a scale factor $R\ge1$. Theorem~\ref{thm:transfer} converts this separation into metastability relative to the invariant law and a mixing time lower bound. The dilation preserves the curvature bounds but reduces the normalized noise amplitude as $R^{-1}$, producing mixing times exponential in $R^2$ at fixed condition number.

Two features of the analysis deserve emphasis. First, the cycling case of the dichotomy of Goujaud, Taylor and Dieuleveut asserts that a cycle exists, while the metastability argument here requires an attracting cycle. We therefore prove a strict form of their cycling conclusion in the regime of the main results: if the worst-case quadratic contraction factor lies below the explicit threshold in Theorem~\ref{thm:dichotomy}, then the parameters lie strictly inside a single roots-of-unity cycling region, and this strict inclusion provides a positive attraction radius (Lemma~\ref{lem:interiorcycle}).

Second, eliminating the velocity from each of the six standard Strang splittings (OBABO, BAOAB, ABOBA, OABAO, AOBOA, and BOAOB) produces a noisy heavy-ball recursion of the same form. The stationary Gaussian noise terms are independent for some splittings and correlated only across adjacent steps for the others. Since Lemma~\ref{lem:metastability} uses the noise terms only through a Gaussian tail bound, the lower bounds transfer to all six splittings (Section~\ref{sec:BAOAB}).

Finally, Corollary~\ref{cor:LRPexplicit} shows that the cycling mechanism already occurs in one dimension at condition number $25$. Building on an example of Lessard, Recht and Packard \citep[Section~4.6 and Appendix~B]{LessardRechtPackard2016}, Section~\ref{sec:LRP} constructs an explicit $C^\infty$ potential on which the noise-free OBABO position recursion, at parameters optimal over quadratic potentials with curvature in $[1,25]$, has an attracting three-cycle. After dilation, the resulting chains have cold starts whose total variation distance from stationarity remains at least $3/8$ for a number of steps exponential in $R^2$.

\subsection{Setting and notation}\label{sec:notation}
Fix a step size $h>0$ and a friction parameter $\gamma>0$, and write
\begin{equation}\label{eq:r_eta_sigma}
 r=e^{-\gamma h/2},\qquad \beta=r^2=e^{-\gamma h},\qquad
 \sigma=(1-\beta)^{1/2}.
\end{equation}
\needspace{10\baselineskip}%
One OBABO step maps the state $(x,v)$ to the updated state $(x^+,v^+)$ through the intermediate velocities $v^{(a)},v^{(b)},v^{(c)}$:
\begin{align}
 v^{(a)}&=rv+\sigma\xi^{(1)},\label{eq:obabo1}\\
 v^{(b)}&=v^{(a)}-\frac h2\nabla U(x),\nonumber\\
 x^+&=x+hv^{(b)},\label{eq:obabo3}\\
 v^{(c)}&=v^{(b)}-\frac h2\nabla U(x^+),\nonumber\\
 v^+&=rv^{(c)}+\sigma\xi^{(2)},\label{eq:obabo5}
\end{align}
where $\xi^{(1)},\xi^{(2)}$ are independent $\cN(0,\Id_d)$ random
variables. The first and last lines are the two $\mathrm O$ half
steps, the second and fourth lines the two $\mathrm B$ half steps,
and the middle line the full $\mathrm A$ step. Iterating this
update with i.i.d.\ copies of the pair $(\xi^{(1)},\xi^{(2)})$ defines a Markov chain on the phase space $\R^d\times\R^d$, and we denote its transition kernel by $Q_{h,\gamma}^U$.

Throughout, $\mathbb N=\{0,1,2,\ldots\}$, $\abs{\cdot}$ denotes the Euclidean norm, $\Id_d$ denotes the $d\times d$ identity matrix, $\norm{M}$ denotes the operator norm of a matrix $M$ induced by the Euclidean norm, $B(x,r)$ denotes the closed Euclidean ball of radius $r\ge0$ centered at $x$, and $\preceq$ denotes the Loewner order on symmetric matrices. We write $\delta_z$ for the Dirac measure at $z$ and $\rho(A)$ for the spectral radius of a square matrix $A$. For probability measures $\nu,\widetilde\nu$ on the same Euclidean space, write
\begin{equation*}
 \norm{\nu-\widetilde\nu}_{\TV}
 =\sup_A\abs{\nu(A)-\widetilde\nu(A)},
 \qquad
 W_p(\nu,\widetilde\nu)^p
 =\inf_{\Gamma\in\Pi(\nu,\widetilde\nu)}
   \int\abs{z-\widetilde z}^p\,\Gamma(dz,d\widetilde z),
\end{equation*}
for $p\in\{1,2\}$, where the supremum runs over Borel sets $A$ and $\Pi(\nu,\widetilde\nu)$ denotes the set of couplings of $\nu$ and $\widetilde\nu$, that is, probability measures on the product space with first marginal $\nu$ and second marginal $\widetilde\nu$. For a Markov kernel $Q$ with invariant law $\pi$ and an initial state $z$, define the mixing time of $Q$ from the state $z$ up to accuracy $\eps>0$ by
\begin{equation*}
 t_{\mathrm{mix}}(z,\eps;Q,\pi)
 =\inf\left\{n\ge0:\norm{\delta_zQ^n-\pi}_{\TV}\le\eps\right\},
\end{equation*}
with the convention $\inf\varnothing=\infty$. Invariant laws of chains are denoted by $\pi$, possibly with a subscript identifying the target; for the splitting schemes studied here they differ in general from the measure $\mu$ in \eqref{eq:gibbs}. We refer to bounds starting from a deterministic state $z$ as cold-start bounds. For every fixed $\eps\in(0,1)$, the supremum of $t_{\mathrm{mix}}(z,\eps;Q,\pi)$ over $z$ is infinite even for an Ornstein-Uhlenbeck chain, so a cold-start bound must retain quantitative dependence on the initial state.

Because the bounds ruled out below are asserted for every pair of
curvature bounds $(K,L)$, it suffices to disprove them on the
normalized subfamily $K=1$, $L=\kappa$. For $\kappa>1$ and $d\ge1$,
define the class of $C^\infty$ potentials on $\R^d$ whose Hessian
eigenvalues lie in $[1,\kappa]$ at every point,
\begin{equation}\label{eq:classUkappa}
 \cU_\kappa^d
 =\left\{U\in C^\infty(\R^d):\ \Id_d\preceq\nabla^2U(x)\preceq\kappa\Id_d
 \text{ for every }x\right\}.
\end{equation}
The $C^\infty$ requirement rules out limited regularity as a
source of the lower bounds; Proposition~\ref{prop:smoothcycle}
shows that the piecewise-quadratic cycling construction of
Goujaud, Taylor and Dieuleveut can be realized within this class.

A \emph{fixed-parameter tuning rule} assigns to each triple $(K,L,d)$ with $0<K<L$ and $d\ge1$ a step size and friction
\begin{equation*}
 \mathfrak t(K,L,d)=(h_{K,L,d},\gamma_{K,L,d})\in(0,\infty)^2,
\end{equation*}
independently of the particular potential, the current state, the iteration number, and the realized noise. It is \emph{numerically stable} if
\begin{equation}\label{eq:stabletuningintro}
 0<h_{K,L,d}<\frac{2}{\sqrt L}
 \qquad\text{for every }0<K<L\text{ and }d\ge1.
\end{equation}
For a given tuning rule, on the normalized dimension-two subfamily
we write
\begin{equation*}
 h_\kappa:=h_{1,\kappa,2},\qquad \gamma_\kappa:=\gamma_{1,\kappa,2}.
\end{equation*}

This normalization is consistent with the natural scaling of
the OBABO chain. Indeed, if
\begin{equation*}
K\,\Id_d\preceq\nabla^2U(x)\preceq L\,\Id_d
\qquad\text{for every }x,
\end{equation*}
then $V:=U(\cdot/\sqrt K)$ belongs to $\cU_\kappa^d$, where $\kappa=L/K$. Under the change of variables
\begin{equation*}
(x,v)\longmapsto(\sqrt K\,x,v),
\end{equation*}
the OBABO chain for $U$ with parameters $(h,\gamma)$ becomes the OBABO chain for $V$ with parameters $(\sqrt K\,h,\gamma/\sqrt K)$. Total variation distances are unchanged under this bijection. The Euclidean Wasserstein distance is generally altered by this anisotropic change of variables, but no Wasserstein invariance is needed for the reduction above.

Section~\ref{sec:HB} shows that \eqref{eq:stabletuningintro} is exactly the condition under which the OBABO iteration matrix is Schur stable for every quadratic potential with Hessian eigenvalues in $[K,L]$. If \eqref{eq:stabletuningintro} fails, then for the quadratic potential $U(x)=\tfrac L2\abs x^2$ the OBABO chain admits no invariant probability law (Lemma~\ref{lem:quadratic-stability}). We therefore restrict attention to numerically stable tuning rules.

\subsection{Main results}\label{sec:intro-main}
Our lower bounds follow from a dichotomy that holds for each numerically stable OBABO tuning. To state it, let
\begin{equation*}
 C_{\mathrm{GTD}}=(3+\sqrt5)^2
\end{equation*}
denote the numerical threshold required by the cycling results of
Goujaud, Taylor and Dieuleveut
\citep[Appendix~B]{GoujaudTaylorDieuleveut2025}. Choose
$C_\star>C_{\mathrm{GTD}}$; it determines the threshold $q_\kappa$
used to distinguish the quadratic and cycling cases. For
$\kappa\ge2C_\star$ and an OBABO parameter pair $(h,\gamma)$, define
\begin{equation}\label{eq:hbparamsintro}
 \beta=e^{-\gamma h},
 \qquad
 s=\frac{h^2}{2}(1+\beta),
 \qquad
 q_\kappa=\frac{1-C_\star/\kappa}{1+C_\star/\kappa}.
\end{equation}
Here $\beta$ and $s$ are the corresponding heavy-ball momentum and
step size.
For $\lambda\in[1,\kappa]$, define
\begin{equation}\label{eq:A_lambda_intro}
 A_\lambda(s,\beta)
 =\begin{pmatrix}1+\beta-s\lambda&\;-\beta\\1&\;0\end{pmatrix},
 \qquad
 \rhoq(s,\beta;\kappa)=\max_{\lambda\in[1,\kappa]}\rho\bigl(A_\lambda(s,\beta)\bigr).
\end{equation}
Here $A_\lambda(s,\beta)$ is the companion matrix for the heavy-ball
position recursion on the quadratic potential $\lambda\abs{x}^2/2$, and
$\rhoq(s,\beta;\kappa)$ is its worst-case spectral radius over
$\lambda\in[1,\kappa]$.

\needspace{23\baselineskip}%
\begin{theorem}[Non-acceleration dichotomy]\label{thm:dichotomy}
Fix $C_\star>C_{\mathrm{GTD}}$ and $\kappa\ge2C_\star$. Let $h,\gamma>0$ satisfy $h<2/\sqrt\kappa$, and let $s$, $\beta$, and $q_\kappa$ be defined as above. Then at least one of the following two cases holds.

\begin{enumerate}
\item[\rm(Q)] There is $\lambda\in[1,\kappa]$ such that, for the
two-dimensional quadratic potential
$U_\lambda(x)=\lambda\abs{x}^2/2$, two initial states
$z,\widetilde z\in\R^4$ satisfy
\begin{equation}\label{eq:quadrootrateintro}
 \lim_{n\to\infty}
 \norm{\delta_z(Q_{h,\gamma}^{U_\lambda})^n-
       \delta_{\widetilde z}(Q_{h,\gamma}^{U_\lambda})^n}_{\TV}^{1/n}
 =\rhoq(s,\beta;\kappa)\ge q_\kappa.
\end{equation}

\item[\rm(C)] There exist $U\in\cU_\kappa^2$ and constants $c_0,C_0>0$ such that, for every $R\ge1$, the following holds. The dilated potential $U_R(x)=R^2U(x/R)$ also belongs to $\cU_\kappa^2$, the kernel $Q_{h,\gamma}^{U_R}$ has a unique invariant law $\pi_R$, and there is an initial state $z_R\in\R^4$ satisfying
\begin{equation}\label{eq:cyclemomentintro}
 W_2(\delta_{z_R},\pi_R)\le C_0(1+R),
\end{equation}
while
\begin{equation}\label{eq:cyclemetaintro}
 \norm{\delta_{z_R}(Q_{h,\gamma}^{U_R})^n-\pi_R}_{\TV}\ge\frac38,
 \qquad
 0\le n\le\left\lfloor\frac1{16}e^{c_0R^2}\right\rfloor.
\end{equation}
\end{enumerate}
\end{theorem}

The proof of Theorem~\ref{thm:dichotomy} is given in Section~\ref{sec:proofmain}.

By \eqref{eq:quadrootrateintro}, case \emph{(Q)} produces two initial states whose asymptotic total variation contraction factor satisfies $\rhoq(s,\beta;\kappa)\ge q_\kappa$. Consequently,
\begin{equation*}
1-\rhoq(s,\beta;\kappa)\;\le\;1-q_\kappa\;=\;\frac{2C_\star}{\kappa+C_\star}\;\le\;\frac{2C_\star}{\kappa}.
\end{equation*}
For comparison, under the optimal quadratic tuning the worst-case cold-start asymptotic total variation contraction factor over Gaussian targets with curvature in $[1,\kappa]$ is
\begin{equation*}
\rho_\ast(\kappa)
=\min_{(s,\beta)}\max_{\lambda\in[1,\kappa]}
\rho\bigl(A_\lambda(s,\beta)\bigr)
=\frac{\sqrt\kappa-1}{\sqrt\kappa+1},
\qquad
1-\rho_\ast(\kappa)=\frac{2}{\sqrt\kappa+1}.
\end{equation*}
The minimization over $(s,\beta)$ is the classical
quadratic-optimality result of Polyak \citep{Polyak1964};
Corollary~\ref{cor:optimalGaussian} gives its total variation
interpretation. Since $t\mapsto(1-t)/(1+t)$ is decreasing, $\rho_\ast(\kappa)<q_\kappa$ precisely when $\kappa>C_\star^2$. Thus, in the regime $\kappa>C_\star^2$ used in the main lower-bound arguments, the optimal quadratic tuning cannot satisfy case \emph{(Q)}, and the dichotomy yields case \emph{(C)}. When $2C_\star\le\kappa\le C_\star^2$, the optimal tuning itself satisfies case \emph{(Q)}, so case \emph{(Q)} does not by itself imply suboptimality on Gaussian targets in that regime. This finite regime does not enter the proof of the lower-bound part of Corollary~\ref{cor:sharp}.

Each case of Theorem~\ref{thm:dichotomy} yields a mixing time lower bound: one for a Gaussian target in case \emph{(Q)}, and one for the family of dilated potentials in case \emph{(C)}.

\needspace{16\baselineskip}%
\begin{corollary}[Mixing time lower bounds]\label{cor:mixing}
Assume the setting of Theorem~\ref{thm:dichotomy}.
\begin{enumerate}
\item[\rm(i)] In case \emph{(Q)}, choose $\lambda\in[1,\kappa]$ for which \eqref{eq:quadrootrateintro} holds, and let $\pi_\lambda$ be the unique invariant law of $Q_{h,\gamma}^{U_\lambda}$. Then there are initial states $z,\widetilde z\in\R^4$ and a constant $c>0$ such that, for every $\eps\in(0,c/4]$, there exists $z_\eps\in\{z,\widetilde z\}$ satisfying
\begin{equation*}
 t_{\mathrm{mix}}(z_\eps,\eps;Q_{h,\gamma}^{U_\lambda},\pi_\lambda)
 \ \ge\ \frac{\kappa}{3C_\star}\,\log\frac{c}{4\eps}-1.
\end{equation*}
\item[\rm(ii)] In case \emph{(C)}, for every $R\ge1$ and every $\eps\in(0,3/8)$,
\begin{equation*}
 t_{\mathrm{mix}}\bigl(z_R,\eps;Q_{h,\gamma}^{U_R},\pi_R\bigr)
 \ >\ \left\lfloor\tfrac1{16}e^{c_0R^2}\right\rfloor,
 \qquad
 W_2(\delta_{z_R},\pi_R)\le C_0(1+R).
\end{equation*}
\end{enumerate}
\end{corollary}

The proof of Corollary~\ref{cor:mixing} is given in Section~\ref{sec:proofmain}.

In particular, these results rule out a fixed-parameter cold-start bound with convergence rate of order $\kappa^{-1/2}$. Corollary~\ref{cor:sharp}(ii) below gives the stronger conclusion that every rate of larger order than $\kappa^{-1}$ is impossible within this class of bounds. These lower bounds are complemented by a cold-start upper bound at the diffusive scale (Theorem~\ref{thm:upper}).

\needspace{18\baselineskip}%
\begin{corollary}[Sharpness of the diffusive scale]\label{cor:sharp}
\begin{enumerate}
\item[\rm(i)] For the tuning
\begin{equation*}
 h=\frac1{4\sqrt\kappa},
 \qquad
 \gamma=4\sqrt\kappa,
 \qquad
 Q=Q_{1/(4\sqrt\kappa),\,4\sqrt\kappa}^{U},
\end{equation*}
there is a universal
constant $C_1<\infty$ such that, for every $\kappa>1$, $d\ge1$,
$U\in\cU_\kappa^d$, every initial state $z$, and every
$\eps\in(0,1)$,
\begin{equation*}
 t_{\mathrm{mix}}(z,\eps;Q,\pi)
 \le C_1\kappa\log\!\left(\frac{C_1\kappa(1+W_2(\delta_z,\pi))}{\eps}\right),
\end{equation*}
where $\pi$ is the unique invariant law of $Q$.
\item[\rm(ii)] Let $\mathfrak t$
be any numerically
stable fixed-parameter OBABO tuning rule, let $C,c>0$ and $a,b\ge0$
be constants, let $\tau:(1,\infty)\to(0,\infty)$ satisfy
$\tau(\kappa)=o(\kappa)$ as $\kappa\to\infty$, and let
$\kappa_0>1$. Then there exist $\kappa\ge\kappa_0$ and
$U\in\cU_\kappa^2$ such that $Q_{h_\kappa,\gamma_\kappa}^U$ has a
unique invariant law $\pi$, and there are $z\in\R^4$ and
$n\in\mathbb N$ for which
\begin{equation}\label{eq:falseuniform}
 \norm{\delta_z(Q_{h_\kappa,\gamma_\kappa}^U)^n-\pi}_{\TV}
 >C\kappa^a\bigl(1+W_2(\delta_z,\pi)\bigr)^b
   \exp\!\left(-\frac{cn}{\tau(\kappa)}\right).
\end{equation}
In particular, no such bound holds with $\tau(\kappa)=\kappa^\alpha$
for any $\alpha<1$.
\end{enumerate}
\end{corollary}

The proof of Corollary~\ref{cor:sharp} is given in Section~\ref{sec:upper}.

For every $\kappa\ge2C_\star$, the same dichotomy applies to each of the six standard Strang splittings: any parameter choice satisfying the corresponding stability condition yields either a Gaussian target on which the chain mixes no faster than at a rate of order $\kappa^{-1}$, or a family of dilated targets with exponentially long metastability (Theorems~\ref{thm:dichotomy-BAOAB} and~\ref{thm:palindromic}).

The lower bounds above rule out cold-start bounds of the form \eqref{eq:falseuniform} for fixed-parameter tunings. An accelerated bound would therefore have to fall outside this class. Possible routes include warm starts, target-dependent or adaptive tuning, randomized time integrators, and modified dynamics. Section~\ref{sec:discussion} discusses these possibilities.

\needspace{6\baselineskip}%
\begin{remark}
As in Goujaud, Taylor and Dieuleveut
\citep{GoujaudTaylorDieuleveut2025}, the counterexamples are
two-dimensional, and the explicit instance of
Corollary~\ref{cor:LRPexplicit} below is one-dimensional, so
non-acceleration does not require high dimension.
\end{remark}

The general dichotomy produces two-dimensional examples for
sufficiently large condition numbers. The following explicit result
shows that the cycling mechanism already occurs in one dimension at
condition number $25$: at the tuning optimal over quadratic
potentials with curvature in $[1,25]$, the OBABO chain is metastable
on an explicit smooth strongly convex potential.

\needspace{19\baselineskip}%
\begin{corollary}[Explicit one-dimensional instance]
\label{cor:LRPexplicit}
Set
\begin{equation}\label{eq:criticalparams}
 h_\star=\frac{2}{\sqrt{26}},\qquad
 \beta_\star=\frac49,\qquad
 \gamma_\star=-\frac1{h_\star}\log\frac49.
\end{equation}
There exist a potential $U\in C^\infty(\R)$ satisfying
$1\le U''\le25$ and a constant $C>0$ such that the following holds
for every $R\ge1$. Let
\[
 U_R(x)=R^2U(x/R),
 \qquad
 Q_R=Q_{h_\star,\gamma_\star}^{U_R}.
\]
Then $U_R\in\cU_{25}^1$, $Q_R$ has a unique invariant law $\pi_R$, and at least one of
the two explicit initial states defined in Section~\ref{sec:LRP},
denoted by $z_R$, satisfies
\begin{equation}\label{eq:LRPexplicitbound}
 W_2(\delta_{z_R},\pi_R)\le CR,
 \qquad
 \norm{\delta_{z_R}Q_R^n-\pi_R}_{\TV}\ge\frac38,
 \qquad
 0\le n\le
 \left\lfloor\frac1{16}e^{R^2/8000}\right\rfloor.
\end{equation}
\end{corollary}

The proof of Corollary~\ref{cor:LRPexplicit} is given in Section~\ref{sec:LRP}.

\subsection{Related work}

The gap between the order-$\kappa$ mixing time guarantee over the full smooth strongly convex class and the order-$\sqrt\kappa$ relaxation scale on Gaussian targets is not specific to OBABO. For the continuous dynamics, hypocoercivity and coupling arguments give quantitative upper bounds. Within the hypocoercive approach \citep{Villani2009,DolbeaultMouhotSchmeiser2015}, Cao, Lu and Wang \citep{CaoLuWang2023} show that an optimized friction choice at the critical scale gives an $L^2$ convergence rate of order $\sqrt K$ under convexity and a Poincar\'e inequality. The entropy bound of Lu \citep{Lu2026} described above gives the same order in relative entropy under their stated logarithmic Sobolev, regularity, and growth assumptions. Coupling arguments give Wasserstein contraction, by synchronous coupling in the strongly convex case \citep{ChengChatterjiBartlettJordan2018} and by couplings combining reflection and synchronous components beyond strong convexity \citep{EberleGuillinZimmer2019}; the resulting rates over the smooth strongly convex class are of nonaccelerated order. Non-reversible lift theory supplies a complementary limitation: lifting can improve non-asymptotic relaxation times by at most a square root \citep{EberleLorler2026}. It also provides a general framework for identifying optimal lifts.

For discretizations, synchronous couplings yield Wasserstein guarantees under strong convexity for both left-endpoint exponential integrators \citep{ChengChatterjiBartlettJordan2018,DalalyanRiouDurand2020,KimGruffazParkDurmus2026} and OBABO \citep{LeimkuhlerPaulinWhalley2024}. Schuh and Whalley \citep{SchuhWhalley2025} establish contraction for the Euler, BU and UBU schemes for non-convex potentials satisfying structural assumptions. Gouraud, Le Bris, Majka and Monmarch\'e \citep{GouraudLeBrisMajkaMonmarche2025} place unadjusted Hamiltonian Monte Carlo \citep{BouRabeeEberle2023,BouRabeeSchuh2023} and kinetic Langevin splittings in the common framework of unadjusted generalized Hamiltonian Monte Carlo. Their Gaussian analysis attains the ballistic rate through partial velocity refreshment. In the same framework, Chak and Monmarch\'e \citep{ChakMonmarche2026} prove contraction by reflection coupling for non-convex potentials with stochastic gradients.

Multi-step couplings give sharp Wasserstein-to-total-variation regularization for OBABO \citep{BouRabeeCoxSchieven2026}. Under a $\chi^2$-warm-start assumption, the shifted-composition framework of Altschuler, Chewi and Zhang \citep{AltschulerChewiZhang2026} gives a mixing guarantee with $\kappa^{5/6}$ dependence, up to logarithmic factors, for the randomized midpoint discretization of \eqref{eq:KLD}. This dependence is sublinear in $\kappa$ but falls short of the square-root scale. None of the cited results gives the $\sqrt\kappa$-guarantee for a discretization of kinetic Langevin dynamics. Appendix~\ref{app:exponential-integrator} rules out such a guarantee for the left-endpoint exponential integrator as well, by a purely Gaussian argument.

In practice, explicit curvature bounds for a given target are rarely
available. Generalized HMC may be run with one Verlet step per
iteration and without a Metropolis correction. If its partial
momentum refresh is split into two half steps, the resulting scheme
is OBABO. The tuning procedure of Hoffman and Sountsov uses an
estimate of the largest eigenvalue of an average Hessian to choose
the step size and an estimate of the largest eigenvalue of an
empirical covariance to choose the damping; on Gaussian targets,
these choices correspond to $h\asymp L^{-1/2}$ and
$\gamma\asymp K^{1/2}$ \citep{HoffmanSountsov2022}.
Corollary~\ref{cor:sharp}(ii) concerns the idealized fixed-parameter setting
in which the corresponding global curvature bounds are known
exactly. It shows that this curvature information alone does not
yield acceleration: no fixed choice of the OBABO step size and
friction gives a cold-start total variation mixing bound of order
$\sqrt\kappa$, up to logarithmic factors, over every target
satisfying the bounds.

This paper concerns unadjusted chains, specifically OBABO and the
other standard Strang splittings: the lower bounds
measure convergence of each chain to its own invariant law, so no
estimate of the asymptotic bias, that is, the discrepancy between
this invariant law and \eqref{eq:gibbs}, is required. Quantitative
bias estimates for OBABO are available in Monmarch\'e
\citep[Propositions~11 and~12]{Monmarche2021}, and Metropolis
adjustment removes the bias entirely, with ergodicity and mixing
guarantees in
\citep{BouRabeeVandenEijnden2010,BouRabeeOberdorster2024}.

\subsection{Organization of the paper}
Section~\ref{sec:HB} reinterprets the OBABO positions as a noisy heavy-ball recursion. Section~\ref{sec:GTD} restates the non-acceleration theorem of Goujaud, Taylor and Dieuleveut in normalized form and derives the strict cycling property needed here. Section~\ref{sec:proofs} proves the mixing time lower bounds. Section~\ref{sec:upper} proves a cold-start upper bound of order $\kappa$, up to a logarithmic factor, for an explicit fixed-parameter OBABO tuning, and shows that this diffusive scale is optimal: no numerically stable fixed-parameter tuning attains any scale $o(\kappa)$ (Corollary~\ref{cor:sharp}). Section~\ref{sec:BAOAB} extends the results to the remaining standard Strang splittings. Section~\ref{sec:discussion} discusses consequences, scope, and open problems. Appendix~\ref{app:exponential-integrator} proves the Gaussian non-acceleration result for the left-endpoint exponential integrator. Appendix~\ref{app:diffusion} determines the sharp worst-case asymptotic total variation rate over deterministic initial states for the kinetic Langevin diffusion on Gaussian targets.

\section{OBABO position as an exact noisy heavy-ball method}\label{sec:HB}

This section derives an exact noisy heavy-ball representation of the
OBABO position process. Proposition~\ref{prop:HB} gives the
representation. For quadratic potentials, we then compute the
associated position and phase-space matrices and characterize the
numerically stable parameter range in
Lemma~\ref{lem:quadratic-stability}.

Eliminating velocity produces a two-step recursion in the positions,
so its initialization requires a pair $(X_0,X_{-1})$. We introduce
the auxiliary position $X_{-1}$ to encode the initial velocity $V_0$.
The variable $X_{-1}$ is not an earlier position visited by the OBABO
chain; once the auxiliary variable $\xi_0^{(2)}$ of
Proposition~\ref{prop:HB} is specified, the relation below can be
read either as defining $X_{-1}$ from $(X_0,V_0)$ or as defining
$V_0$ from $(X_{-1},X_0)$.

Two initialization conventions will be useful. Taking
$\xi_0^{(2)}$ to be an independent standard Gaussian makes the
noise terms identically distributed from the first position update
onward. This convention is used later to lift an invariant law of the
$(X_{k-1},X_k)$-chain to phase space in the proof of
Theorem~\ref{thm:transfer}. For deterministic cold-start
initializations, we instead set $\xi_0^{(2)}=0$; the first noise
term in the resulting recursion then has a smaller covariance.

Recall from \eqref{eq:r_eta_sigma} that $r=e^{-\gamma h/2}$, $\beta=r^2=e^{-\gamma h}$, and $\sigma=(1-\beta)^{1/2}$.

\needspace{8\baselineskip}%
\begin{proposition}[Noisy heavy-ball representation]\label{prop:HB}
Let $(X_k,V_k)_{k\ge0}$ be the OBABO chain \eqref{eq:obabo1}-\eqref{eq:obabo5}. Let $\xi_0^{(2)}$ be a $\cN(0,\Id_d)$ variable independent of the pairs $(\xi_k^{(1)},\xi_k^{(2)})_{k\ge1}$ driving the chain. Alternatively, let $\xi_0^{(2)}$ be the zero vector. Define the auxiliary variable $X_{-1}$ by
\begin{equation}\label{eq:vpairrelation}
 V_0=r\left(\frac{X_0-X_{-1}}h-\frac h2\nabla U(X_0)\right)
      +\sigma\xi_0^{(2)}.
\end{equation}
Then
\begin{equation}\label{eq:noisyHB}
 X_{k+1}=X_k+\beta(X_k-X_{k-1})
 -s\nabla U(X_k)+\zeta_{k+1},
 \qquad
 s=\frac{h^2}{2}(1+\beta),
\end{equation}
where
\begin{equation}\label{eq:zetadef}
 \zeta_{k+1}=h\sigma\bigl(r\xi_k^{(2)}+\xi_{k+1}^{(1)}\bigr).
\end{equation}
Distinct $\zeta_k$ depend on disjoint noises and are therefore
independent centered Gaussians. If
$\xi_0^{(2)}\sim\cN(0,\Id_d)$, they are identically distributed
with covariance
\begin{equation}\label{eq:zetacov}
 \E[\zeta_k\zeta_k^\top]
 =h^2(1-\beta^2)\Id_d.
\end{equation}
If $\xi_0^{(2)}=0$, then
\[
 \zeta_1=h\sigma\xi_1^{(1)},\qquad
 \E[\zeta_1\zeta_1^\top]=h^2(1-\beta)\Id_d,
\]
while \eqref{eq:zetacov} holds for every $k\ge2$.
\end{proposition}

\begin{proof}
Index the noises so that the update from $(X_k,V_k)$ to
$(X_{k+1},V_{k+1})$ uses
$(\xi_{k+1}^{(1)},\xi_{k+1}^{(2)})$. The first
Ornstein-Uhlenbeck half-step, the first kick, and the position update
in \eqref{eq:obabo1}-\eqref{eq:obabo3} give
\begin{equation}\label{eq:xstep}
 X_{k+1}-X_k
 =h\left(rV_k+\sigma\xi_{k+1}^{(1)}
          -\frac h2\nabla U(X_k)\right).
\end{equation}
For $k\ge1$, the position update \eqref{eq:obabo3} of the step
producing $X_k$ shows that the intermediate velocity $v^{(b)}$ of
that step equals $(X_k-X_{k-1})/h$. The final kick and
Ornstein-Uhlenbeck half-step of the same step therefore give
\begin{equation}\label{eq:vfrompair}
 V_k=r\left(\frac{X_k-X_{k-1}}h
             -\frac h2\nabla U(X_k)\right)
      +\sigma\xi_k^{(2)}.
\end{equation}
For $k=0$, equation~\eqref{eq:vfrompair} is precisely the imposed
relation~\eqref{eq:vpairrelation}.
Substituting \eqref{eq:vfrompair} into \eqref{eq:xstep} and using $r^2=\beta$ gives
\begin{equation*}
 X_{k+1}-X_k
 =\beta(X_k-X_{k-1})
 -\frac{h^2}{2}(1+\beta)\nabla U(X_k)
 +h\sigma\bigl(r\xi_k^{(2)}+\xi_{k+1}^{(1)}\bigr),
\end{equation*}
which is \eqref{eq:noisyHB}-\eqref{eq:zetadef}. By \eqref{eq:zetadef}, the variable $\zeta_{k+1}$ depends only on the pair $(\xi_k^{(2)},\xi_{k+1}^{(1)})$, and these pairs share no variable for distinct $k$, so the $(\zeta_k)_{k\ge1}$ are independent centered Gaussians. When $\xi_0^{(2)}\sim\cN(0,\Id_d)$, their common covariance is $h^2\sigma^2(1+r^2)\,\Id_d$, and
\begin{equation*}
 h^2\sigma^2(1+r^2)=h^2(1-\beta)(1+\beta)=h^2(1-\beta^2),
\end{equation*}
which is \eqref{eq:zetacov}. If $\xi_0^{(2)}=0$, then
$\zeta_1=h\sigma\xi_1^{(1)}$ has covariance
$h^2(1-\beta)\Id_d$. From $\zeta_2$ onward, both Gaussian terms in
\eqref{eq:zetadef} are present, and the covariance is again given by
\eqref{eq:zetacov}.
\end{proof}

Suppressing the noise in \eqref{eq:noisyHB} gives Polyak's heavy-ball method with step size $s$ and momentum $\beta$ \citep{Polyak1964}. For the quadratic potential $U_\lambda(x)=\lambda\abs{x}^2/2$, the two-step position matrix, acting on the pair $(X_k,X_{k-1})$, is $A_\lambda(s,\beta)$ from \eqref{eq:A_lambda_intro}. The deterministic phase-space matrix, acting on $(X_k,V_k)$, is
\begin{equation}\label{eq:Mlambda}
 M_\lambda=
 \begin{pmatrix}
 1-h^2\lambda/2&hr\\
 -hr\lambda(1-h^2\lambda/4)&\beta(1-h^2\lambda/2)
 \end{pmatrix}.
\end{equation}
Direct computation gives
\begin{equation*}
 \operatorname{tr}M_\lambda=1+\beta-s\lambda,
 \qquad
 \det M_\lambda=\beta.
\end{equation*}
Both matrices therefore have characteristic polynomial
\begin{equation}\label{eq:charpoly}
 z^2-(1+\beta-s\lambda)z+\beta.
\end{equation}
In particular,
\begin{equation}\label{eq:rhosame}
 \rho(M_\lambda)=\rho(A_\lambda(s,\beta)).
\end{equation}

\begin{lemma}[Numerical stability]\label{lem:quadratic-stability}
For $\beta\in(0,1)$ and $s>0$, all matrices $A_\lambda(s,\beta)$, $\lambda\in[1,\kappa]$, are Schur stable, meaning that their spectral radii are strictly less than one, if and only if
\begin{equation}\label{eq:HBstability}
 0<s<\frac{2(1+\beta)}{\kappa}.
\end{equation}
For OBABO, using $s=h^2(1+\beta)/2$, this is equivalent to $0<h<2/\sqrt\kappa$. Moreover, if $h\ge2/\sqrt\kappa$, the OBABO chain for the potential $U_\kappa(x)=\kappa\abs x^2/2$ has no invariant probability law.
\end{lemma}

\begin{proof}
A $2\times2$ real matrix $M$ is Schur stable if and only if $\abs{\operatorname{tr}M}<1+\det M<2$. Since $A_\lambda(s,\beta)$ has trace $1+\beta-s\lambda$ and determinant $\beta\in(0,1)$, this condition reduces to $0<s\lambda<2(1+\beta)$. Uniformity over $\lambda\in[1,\kappa]$ gives \eqref{eq:HBstability}.

For the second assertion, let $h\ge2/\sqrt\kappa$, equivalently $s\kappa\ge2(1+\beta)$, and let $p$ denote the characteristic polynomial in \eqref{eq:charpoly} with $\lambda=\kappa$. Then
\begin{equation*}
 p(-1)=2(1+\beta)-s\kappa\le0,
\end{equation*}
while $p(z)\to\infty$ as $z\to-\infty$. Hence $p$ has a real root $z_-\le-1$.
For the potential $U_\kappa$, the $d$ coordinates of the OBABO chain evolve independently, so it suffices to consider one of them: the pair $Z_k\in\R^2$ formed by the $i$th components of $(X_k,V_k)$ satisfies the linear recursion $Z_{k+1}=M_\kappa Z_k+N\xi_{k+1}$, where $\xi_{k+1}\in\R^2$ collects the $i$th components of $(\xi_{k+1}^{(1)},\xi_{k+1}^{(2)})$ and
\begin{equation*}
 N=\begin{pmatrix}h\sigma&0\\ r\sigma(1-h^2\kappa/2)&\sigma\end{pmatrix}.
\end{equation*}
Let $w$ be a left eigenvector of $M_\kappa$ for $z_-$. The scalar sequence $Y_k=\langle w,Z_k\rangle$ satisfies $Y_{k+1}=z_-Y_k+\eta_{k+1}$ with independent centered Gaussians $\eta_{k+1}=\langle w,N\xi_{k+1}\rangle$ of variance $\omega^2=\abs{N^\top w}^2$; since $\det N=h\sigma^2>0$, the matrix $N^\top$ is invertible, so $N^\top w\ne0$ and $\omega^2>0$.

Suppose the full $d$-dimensional chain had an invariant probability law $\pi$, and start it at stationarity. The $i$th-coordinate marginal of $\pi$ is then invariant for the two-dimensional recursion above, so $Z_0$ and $Z_1$ have the same law, and hence so do $Y_0$ and $Y_1$. Since $Y_1=z_-Y_0+\eta_1$ with $\eta_1\sim\cN(0,\omega^2)$ independent of $Y_0$, their common characteristic function $\varphi$ satisfies $\varphi(t)=\E\bigl[e^{itY_1}\bigr]=\varphi(z_-t)\,e^{-\omega^2t^2/2}$ for all $t\in\R$. Iterating this identity $n$ times gives
\begin{equation*}
 \varphi(t)
 =\varphi(z_-^nt)\,
  \exp\Bigl(-\frac{\omega^2t^2}{2}\sum_{j=0}^{n-1}z_-^{2j}\Bigr).
\end{equation*}
Since $\abs\varphi\le1$ and $z_-^{2j}\ge1$ for every $j$, it follows that $\abs{\varphi(t)}\le e^{-n\omega^2t^2/2}$ for every $n\ge1$, so $\varphi(t)=0$ for $t\ne0$. Since $\varphi(0)=1$, this contradicts the continuity of $\varphi$ at $0$.
\end{proof}

\section{The deterministic heavy-ball dichotomy}\label{sec:GTD}

This section concerns only the deterministic heavy-ball method, the noise-free recursion obtained from \eqref{eq:noisyHB} by removing the noise variables. In Theorem~\ref{thm:dichotomy}, case \emph{(C)} holds when
\begin{equation}\label{eq:acceleratedquad}
 \rhoq(s,\beta;\kappa)<q_\kappa.
\end{equation}
Under \eqref{eq:acceleratedquad} and the remaining hypotheses of Lemma~\ref{lem:interiorcycle}, the results below produce a potential $U\in\cU_\kappa^2$ together with $m$ distinct points that satisfy the heavy-ball recursion for $U$ and whose linearization is exponentially stable. Section~\ref{sec:proofs} combines these points with small-noise estimates for the OBABO chain to prove case \emph{(C)}.

The starting point is the cycling result of Goujaud, Taylor and Dieuleveut \citep{GoujaudTaylorDieuleveut2025}, normalized to the curvature bounds $K=1$ and $L=\kappa$. Theorem~\ref{thm:GTD} states that the heavy-ball recursion advances through the vertices of a regular $m$-sided polygon on some $1$-strongly convex $C^1$ potential function with $\kappa$-Lipschitz gradient if and only if an explicit quadratic in $s$ is nonpositive. When this condition holds, \eqref{eq:GTDpsi} gives an explicit piecewise-quadratic potential function. Strict inequality is needed below because it gives positive-radius neighborhoods of the cycle vertices on which the gradient is affine. Lemma~\ref{lem:interiorcycle} proves this strict inequality for one period $m$. Proposition~\ref{prop:smoothcycle} then mollifies the resulting potential function without changing the dynamics near the cycle; the mollified potential belongs to $\cU_\kappa^2$, has the same cycle, and, by numerical stability, has an exponentially stable linearization along it.

To avoid conflict with the Langevin friction, we use $s$ for the heavy-ball step size denoted $\gamma$ in Goujaud, Taylor and Dieuleveut \citep{GoujaudTaylorDieuleveut2025}. Their results are stated for the class of $\mu$-strongly convex $C^1$ functions with $L$-Lipschitz gradient and in terms of the inverse condition number. Throughout this section, the heavy-ball parameters $(s,\beta)$ are fixed and satisfy the numerical stability condition \eqref{eq:HBstability}.

The construction below reverses the usual order of analysis: rather than fixing a potential function $\psi$ and studying the resulting heavy-ball recursion, it prescribes an $m$-periodic sequence of points $(x_t)_{t\in\mathbb Z}$ in $\R^2$ and then constructs a differentiable $\psi:\R^2\to\R$ for which these points satisfy the heavy-ball recursion
\begin{equation}\label{eq:GTD-cycle-eq}
 x_{t+1}=(1+\beta)x_t-\beta x_{t-1}-s\nabla\psi(x_t)
 \qquad\text{for every }t.
\end{equation}
The points are taken to be the $m$th roots of unity, that is, the vertices of a regular $m$-gon inscribed in the unit circle of $\R^2$, visited in cyclic order: for an integer $m\ge3$, write $\theta_m=2\pi/m$, let $R_m$ denote rotation of $\R^2$ by $\theta_m$, and set
\begin{equation*}
 x_t^\circ=R_m^tx_0^\circ,\qquad x_0^\circ=(1,0),
\end{equation*}
with indices interpreted modulo $m$. Figure~\ref{fig:gtd-cycle} illustrates the construction.

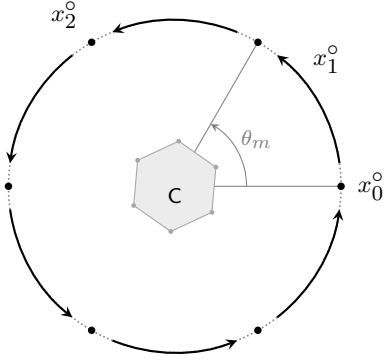
\begin{figure}[t]
\centering
\begin{tikzpicture}[>=stealth]
 \draw[densely dotted,gray,semithick] (0,0) circle (2.2);
 \draw[thin,gray] (0,0) -- (0:2.2);
 \draw[thin,gray] (0,0) -- (60:2.2);
 \fill[gray!15] (25:0.6) -- (85:0.6) -- (145:0.6) -- (205:0.6) -- (265:0.6) -- (325:0.6) -- cycle;
 \draw[gray!70] (25:0.6) -- (85:0.6) -- (145:0.6) -- (205:0.6) -- (265:0.6) -- (325:0.6) -- cycle;
 \foreach \a in {25,85,...,325} \fill[gray!70] (\a:0.6) circle (0.9pt);
 \node[font=\scriptsize] at (0,-0.12) {$\mathsf C$};
 \draw[->,gray] (0:0.95) arc (0:60:0.95);
 \node[font=\scriptsize,gray] at (30:1.25) {$\theta_m$};
 \foreach \a in {0,60,...,300} \fill (\a:2.2) circle (1.4pt);
 \node[right=3pt] at (0:2.2) {$x_0^\circ$};
 \node at (40:2.62) {$x_1^\circ$};
 \node[above left=2pt] at (120:2.2) {$x_2^\circ$};
 \foreach \a [evaluate=\a as \sa using \a+8, evaluate=\a as \ea using \a+52] in {0,60,...,300} {
   \draw[->,thick] (\sa:2.2) arc (\sa:\ea:2.2);
 }
\end{tikzpicture}
\caption{Schematic of the $m$th-roots-of-unity cycle of
Theorem~\ref{thm:GTD}, shown for $m=6$. One heavy-ball step \eqref{eq:GTD-cycle-eq} advances the cycle point $x_t^\circ$ to $x_{t+1}^\circ$, a rotation by $\theta_m=2\pi/m$. The gradient of the piecewise-quadratic potential function $\psi$ in \eqref{eq:GTDpsi} is determined by the metric projection onto the convex hull $\mathsf C$ of the points $Mx_t^\circ$.}
\label{fig:gtd-cycle}
\end{figure}

The first step of the construction is to determine the gradient values that the potential function $\psi$ must have at these points. Solving \eqref{eq:GTD-cycle-eq} for the gradient at $x_t^\circ$ gives
\begin{equation}\label{eq:GTDgrad}
 \nabla\psi(x_t^\circ)
 =\frac{(1+\beta)x_t^\circ-\beta x_{t-1}^\circ-x_{t+1}^\circ}{s}
 \qquad\text{for every }t.
\end{equation}
Thus $(x_t^\circ)$ satisfies \eqref{eq:GTD-cycle-eq} if and only if \eqref{eq:GTDgrad} holds. To put \eqref{eq:GTDgrad} in a form suitable for constructing $\psi$, define
\begin{equation}\label{eq:Mdef}
 M=\frac{(1+\beta-s)\Id_2-R_m-\beta R_m^{-1}}{(\kappa-1)s}.
\end{equation}
Since $x_{t+1}^\circ=R_mx_t^\circ$ and $x_{t-1}^\circ=R_m^{-1}x_t^\circ$, \eqref{eq:GTDgrad} can equivalently be written as
\begin{equation*}
 \nabla\psi(x_t^\circ)=x_t^\circ+(\kappa-1)Mx_t^\circ
 \qquad\text{for every }t.
\end{equation*}

Set
\begin{equation*}
 \mathsf C=\operatorname{conv}\{Mx_t^\circ:\,0\le t\le m-1\},
\end{equation*}
where $\operatorname{conv}$ denotes the convex hull, and consider the piecewise-quadratic potential function
\begin{equation}\label{eq:GTDpsi}
 \psi(x)=\frac{\kappa}{2}\abs{x}^2-\frac{\kappa-1}{2}d(x,\mathsf C)^2,
 \qquad
 \nabla\psi(x)=x+(\kappa-1)\proj_{\mathsf C}(x).
\end{equation}
Here $d(\cdot,\mathsf C)$ denotes the Euclidean distance to $\mathsf C$, and $\proj_{\mathsf C}(x)$ denotes the metric projection of $x$ onto $\mathsf C$, that is, the point of $\mathsf C$ nearest to $x$; it is well defined because $\mathsf C$ is nonempty, closed and convex. Comparing the gradient in \eqref{eq:GTDpsi} with \eqref{eq:GTDgrad}, and using $\kappa>1$, the potential function \eqref{eq:GTDpsi} satisfies \eqref{eq:GTDgrad} if and only if
\begin{equation}\label{eq:GTDproj}
 \proj_{\mathsf C}(x_t^\circ)=Mx_t^\circ
 \qquad\text{for every }t.
\end{equation}
The identities \eqref{eq:GTDproj} are not automatic. Each $Mx_t^\circ$ lies in $\mathsf C$ by construction, but whether it is the point of $\mathsf C$ nearest to $x_t^\circ$ depends on $(s,\beta,\kappa,m)$. When \eqref{eq:GTDproj} holds, $(x_t^\circ)$ satisfies \eqref{eq:GTD-cycle-eq} with $\psi$ given by \eqref{eq:GTDpsi}.

To restate \eqref{eq:GTDproj} in a form that can be checked, recall the variational characterization of the metric projection: for a nonempty closed convex set $\mathsf C\subseteq\R^2$ and $x\in\R^2$, a point $c\in\mathsf C$ satisfies $c=\proj_{\mathsf C}(x)$ if and only if
\begin{equation*}
 \langle x-c,\,y-c\rangle\le0
 \qquad\text{for every }y\in\mathsf C.
\end{equation*}
The left-hand side is affine in $y$, so it suffices to impose this condition for $y$ ranging over a set whose convex hull is $\mathsf C$. Taking $x=x_t^\circ$ and $c=Mx_t^\circ$, and letting $y$ range over the vertices $Mx_j^\circ$, the identities \eqref{eq:GTDproj} therefore hold if and only if
\begin{equation}\label{eq:GTDvarineq}
 \bigl\langle x_t^\circ-Mx_t^\circ,\;
 Mx_j^\circ-Mx_t^\circ\bigr\rangle\le0
 \qquad\text{for all }j,t.
\end{equation}
We call \eqref{eq:GTDvarineq} the \emph{projection inequalities}. Since $M$ is a linear combination of $\Id_2$, $R_m$ and $R_m^{-1}$, it commutes with $R_m$, so $Mx_t^\circ=R_m^tMx_0^\circ$. Moreover, since rotations preserve inner products, the inner product in \eqref{eq:GTDvarineq} depends on $j$ and $t$ only through $j-t$. The calculation in \citep[Section~B.1]{GoujaudTaylorDieuleveut2025} shows that, among the resulting $m-1$ inequalities, the one with $j-t=1$ implies all the others, and that it holds if and only if the following quadratic in $s$ is nonpositive:
\begin{equation*}
\begin{split}
 P_m(s,\beta;\kappa)
 =s^2&-2\left[\beta-\cos\theta_m+\kappa^{-1}(1-\beta\cos\theta_m)\right]s\\
 &+2\kappa^{-1}(1-\cos\theta_m)(1+\beta^2-2\beta\cos\theta_m).
\end{split}
\end{equation*}
This is the \emph{$m$th-roots-of-unity cycling quadratic}, the normalized form of the quadratic in \citep[Theorem~3.5]{GoujaudTaylorDieuleveut2025}, obtained from the convex interpolation calculation in \citep[Section~B.1]{GoujaudTaylorDieuleveut2025}.

Their calculation shows that \eqref{eq:GTDvarineq} holds if and only if $P_m(s,\beta;\kappa)\le0$.

When $P_m(s,\beta;\kappa)<0$, the projection inequalities \eqref{eq:GTDvarineq} are strict. For fixed $t$, each inequality defines a half-plane containing $x_t^\circ$, and the distance from $x_t^\circ$ to the nearest boundary gives a positive ball on which $\proj_{\mathsf C}$ is constant, so that, by \eqref{eq:GTDpsi}, the potential function $\psi$ is quadratic with Hessian $\Id_2$ on this ball. Figure~\ref{fig:rmax-geometry} illustrates this geometry. The precise statement follows.

\begin{theorem}[Goujaud-Taylor-Dieuleveut]\label{thm:GTD}
Let $\kappa>1$, let $(s,\beta)\in(0,\infty)\times(0,1)$ satisfy the numerical stability condition \eqref{eq:HBstability}, and fix an integer $m\ge3$. With the notation introduced above:
\begin{enumerate}
\item[\rm(i)] There exists a $C^1$ function $f:\R^2\to\R$ that is $1$-strongly convex and has $\kappa$-Lipschitz gradient such that
\begin{equation*}
 x_{t+1}^\circ=(1+\beta)x_t^\circ-\beta x_{t-1}^\circ-s\nabla f(x_t^\circ)
 \qquad\text{for every }t
\end{equation*}
if and only if $P_m(s,\beta;\kappa)\le0$. When this condition holds, the potential function $\psi$ defined in \eqref{eq:GTDpsi} is one such function and is piecewise quadratic.
\item[\rm(ii)] Suppose $P_m(s,\beta;\kappa)<0$, and let $\psi$ be the potential function in \eqref{eq:GTDpsi}. Then there is a radius $r_{\max}=r_{\max}(m,s,\beta,\kappa)>0$, given explicitly in the proof, such that $\psi$ is quadratic with Hessian $\Id_2$ on each ball $B(x_t^\circ,r_{\max})$, that is,
\begin{equation}\label{eq:psi-local}
 \nabla\psi(x_t^\circ+u)=\nabla\psi(x_t^\circ)+u,
 \qquad\abs{u}\le r_{\max}.
\end{equation}
\end{enumerate}
\end{theorem}

\begin{proof}
Part (i) is \citep[Theorem~3.5]{GoujaudTaylorDieuleveut2025} together with the interpolation argument of their Section~B.1, normalized as above; the formula for $\nabla\psi$ follows from $\nabla d(\cdot,\mathsf C)^2=2(\operatorname{id}-\proj_{\mathsf C})$.

For part (ii), we first show that $P_m(s,\beta;\kappa)<0$ makes every projection inequality strict. We then define $r_{\max}$ as the smallest distance from $x_0^\circ$ to the boundary lines of the corresponding half-planes. Finally, we show that every ball $B(x_t^\circ,r_{\max})$ remains inside those half-planes, so the metric projection is constant there.

\emph{Strict projection inequalities.} Write
\begin{equation*}
 I_{t,j}=\bigl\langle x_t^\circ-Mx_t^\circ,\;
 Mx_j^\circ-Mx_t^\circ\bigr\rangle,
\end{equation*}
so that the projection inequalities \eqref{eq:GTDvarineq}, which characterize when each $Mx_t^\circ$ is the point of $\mathsf C$ nearest to $x_t^\circ$, read $I_{t,j}\le0$ for all $j,t$; as noted above, $I_{t,j}=I_{0,j-t}$.

We first observe that $M$ is invertible. Writing
\begin{equation*}
 R_m=\cos\theta_m\,\Id_2+\sin\theta_m\,J,
 \qquad
 J=\begin{pmatrix}0&-1\\1&0\end{pmatrix},
\end{equation*}
and $R_m^{-1}=\cos\theta_m\,\Id_2-\sin\theta_m\,J$, the definition \eqref{eq:Mdef} of $M$ gives $M=a\Id_2+bJ$ with
\begin{equation*}
 a=\frac{(1+\beta-s)-(1+\beta)\cos\theta_m}{(\kappa-1)s},
 \qquad
 b=-\frac{(1-\beta)\sin\theta_m}{(\kappa-1)s}.
\end{equation*}
Since $\beta\in(0,1)$ and $\sin\theta_m>0$ for $m\ge3$, we have $b<0$; hence $\det M=a^2+b^2>0$. In particular $M(x_j^\circ-x_t^\circ)\ne0$ whenever $j\ne t$.

Since $x_j^\circ=R_m^jx_0^\circ$, the definition of $I_{0,j}$ reads
\begin{equation*}
 I_{0,j}
 =\bigl\langle(\Id_2-M)x_0^\circ,\;M(R_m^j-\Id_2)x_0^\circ\bigr\rangle
 =\bigl\langle x_0^\circ,\;(\Id_2-M)^\top M(R_m^j-\Id_2)\,x_0^\circ\bigr\rangle.
\end{equation*}
Every matrix appearing here has the form $p\Id_2+qJ$; such matrices satisfy $(p\Id_2+qJ)^\top=p\Id_2-qJ$ and multiply by the rule
\begin{equation*}
 (p\Id_2+qJ)(r\Id_2+uJ)=(pr-qu)\Id_2+(pu+qr)J.
\end{equation*}
Applying this rule first to $(\Id_2-M)^\top M=\bigl((1-a)\Id_2+bJ\bigr)\bigl(a\Id_2+bJ\bigr)=\bigl(a-a^2-b^2\bigr)\Id_2+bJ$, and then to the product with $R_m^j-\Id_2=\bigl(\cos(j\theta_m)-1\bigr)\Id_2+\sin(j\theta_m)J$, gives
\begin{equation*}
 (\Id_2-M)^\top M(R_m^j-\Id_2)=c_j\Id_2+c_j'J,
 \quad
 c_j=\bigl(1-\cos(j\theta_m)\bigr)\bigl(a^2+b^2-a\bigr)-b\sin(j\theta_m),
\end{equation*}
where the value of the scalar $c_j'$ is not needed: since $\abs{x_0^\circ}=1$ and $\langle x_0^\circ,Jx_0^\circ\rangle=0$, we get $I_{0,j}=c_j$. The half-angle identity $\sin(j\theta_m)=\bigl(1-\cos(j\theta_m)\bigr)\cot(j\theta_m/2)$, valid for $1\le j\le m-1$ since then $j\theta_m/2\in(0,\pi)$, gives
\begin{equation}\label{eq:I0j}
 I_{0,j}
 =\bigl(1-\cos(j\theta_m)\bigr)
 \Bigl[a^2+b^2-a-b\cot\bigl(j\theta_m/2\bigr)\Bigr].
\end{equation}
The prefactor is positive, and $j\theta_m/2$ increases through $(0,\pi)$ as $j$ runs from $1$ to $m-1$, on which interval $\cot$ is strictly decreasing; since $b<0$, the expression in brackets in \eqref{eq:I0j} is therefore strictly decreasing in $j$. Consequently
\begin{equation}\label{eq:I01dominates}
 I_{0,1}\le0\ \Longrightarrow\ I_{0,j}\le0
 \quad\text{for }1\le j\le m-1,
\end{equation}
and the same implication holds with both inequalities strict.

Normalized as above, the identity of \citep[Section~B.1]{GoujaudTaylorDieuleveut2025} reads
\begin{equation}\label{eq:PmIm}
 P_m(s,\beta;\kappa)
 =\frac{(\kappa-1)^2s^2}{\kappa(1-\cos\theta_m)}\,I_{0,1},
\end{equation}
and the prefactor is strictly positive. Together with \eqref{eq:I01dominates} and $I_{t,j}=I_{0,j-t}$, this gives the equivalence of \eqref{eq:GTDvarineq} with $P_m(s,\beta;\kappa)\le0$ used above, and shows that the hypothesis $P_m(s,\beta;\kappa)<0$ implies $I_{t,j}<0$ for all $j\ne t$.

\emph{The radius $r_{\max}$.} For $x\in\R^2$, the variational characterization of the metric projection, applied with $Mx_t^\circ$ as the candidate projection, requires
\begin{equation}\label{eq:projcone}
 \bigl\langle x-Mx_t^\circ,\;Mx_j^\circ-Mx_t^\circ\bigr\rangle\le0
 \qquad\text{for every }j.
\end{equation}
For $j\ne t$, the inequality indexed by $j$ defines a half-plane whose boundary is the line through $Mx_t^\circ$ perpendicular to $M(x_j^\circ-x_t^\circ)$, and the perpendicular distance from $x_t^\circ$ to this line equals $-I_{t,j}/\abs{M(x_j^\circ-x_t^\circ)}$. Since $I_{t,j}=I_{0,j-t}$ and $\abs{M(x_j^\circ-x_t^\circ)}=\abs{M(x_{j-t}^\circ-x_0^\circ)}$, the smallest of these distances does not depend on $t$. Set
\begin{equation}\label{eq:rmaxdef}
 r_{\max}
 =\min_{1\le j\le m-1}
 \frac{-I_{0,j}}{\abs{M(x_j^\circ-x_0^\circ)}},
\end{equation}
which is well defined because $M$ is invertible, and positive because every $I_{0,j}$ is negative. These quotients are the distances shown in Figure~\ref{fig:rmax-geometry}.

\emph{Local quadratic structure.} Let $\abs u\le r_{\max}$ and $j\ne t$. Using $I_{t,j}=I_{0,j-t}$, $\abs{M(x_j^\circ-x_t^\circ)}=\abs{M(x_{j-t}^\circ-x_0^\circ)}$, and the Cauchy-Schwarz inequality,
\begin{equation*}
 \bigl\langle x_t^\circ+u-Mx_t^\circ,\;Mx_j^\circ-Mx_t^\circ\bigr\rangle
 \le I_{t,j}+\abs u\,\abs{M(x_j^\circ-x_t^\circ)}
 \le0
\end{equation*}
by \eqref{eq:rmaxdef}; for $j=t$ the left-hand side vanishes. Thus \eqref{eq:projcone} holds at $x=x_t^\circ+u$. Because $\mathsf C$ is the convex hull of the points $Mx_j^\circ$, every $y\in\mathsf C$ can be written as $y=\sum_j\alpha_jMx_j^\circ$, where $\alpha_j\ge0$ and $\sum_j\alpha_j=1$; hence $y-Mx_t^\circ=\sum_j\alpha_j(Mx_j^\circ-Mx_t^\circ)$. The inequalities just proved therefore give
\begin{equation*}
 \bigl\langle x_t^\circ+u-Mx_t^\circ,\;y-Mx_t^\circ\bigr\rangle
 =\sum_j\alpha_j
 \bigl\langle x_t^\circ+u-Mx_t^\circ,\;Mx_j^\circ-Mx_t^\circ\bigr\rangle
 \le0
\end{equation*}
for every $y\in\mathsf C$. Since $Mx_t^\circ\in\mathsf C$, the variational characterization of the metric projection gives $\proj_{\mathsf C}(x_t^\circ+u)=Mx_t^\circ$, so by \eqref{eq:GTDpsi},
\begin{equation*}
 \nabla\psi(x_t^\circ+u)
 =x_t^\circ+u+(\kappa-1)Mx_t^\circ
 =\nabla\psi(x_t^\circ)+u,
\end{equation*}
which is \eqref{eq:psi-local}.
\end{proof}

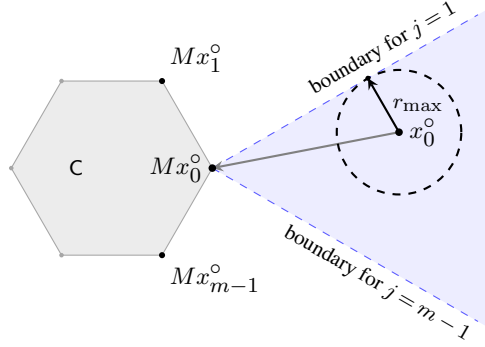
\begin{figure}[t]
\centering
\begin{tikzpicture}[scale=0.95,>=stealth]
 \fill[blue!7] (0,0) -- (3.9,2.2517) -- (3.9,-2.2517) -- cycle;
 \fill[gray!15] (0,0) -- (-0.7,1.2124) -- (-2.1,1.2124) -- (-2.8,0) -- (-2.1,-1.2124) -- (-0.7,-1.2124) -- cycle;
 \draw[gray!70] (0,0) -- (-0.7,1.2124) -- (-2.1,1.2124) -- (-2.8,0) -- (-2.1,-1.2124) -- (-0.7,-1.2124) -- cycle;
 \foreach \p in {(-2.1,1.2124),(-2.8,0),(-2.1,-1.2124)} \fill[gray!70] \p circle (0.9pt);
 \draw[dashed,blue!65] (0,0) -- (3.7,2.136);
 \draw[dashed,blue!65] (0,0) -- (3.7,-2.136);
 \node[font=\scriptsize,rotate=30,anchor=south] at (2.51,1.45) {boundary for $j=1$};
 \node[font=\scriptsize,rotate=-30,anchor=north] at (2.42,-1.4) {boundary for $j=m-1$};
 \coordinate (x0) at (2.6,0.5);
 \coordinate (p0) at (0,0);
 \coordinate (q) at (2.167,1.251);
 \draw[->,thick,gray] (x0) -- (p0);
 \draw[dashed,thick] (x0) circle (0.867);
 \draw[->,thick] (x0) -- (q) node[midway,right,font=\scriptsize] {$r_{\max}$};
 \fill (x0) circle (1.5pt);
 \node[anchor=west,inner sep=1.5pt] at (2.68,0.5) {$x_0^\circ$};
 \fill (p0) circle (1.5pt);
 \node[anchor=east,inner sep=1.5pt] at (-0.07,0) {$Mx_0^\circ$};
 \fill (-0.7,1.2124) circle (1.2pt);
 \node[anchor=south west,inner sep=1.5pt] at (-0.66,1.28) {$Mx_1^\circ$};
 \fill (-0.7,-1.2124) circle (1.2pt);
 \node[anchor=north west,inner sep=1.5pt] at (-0.66,-1.28) {$Mx_{m-1}^\circ$};
 \fill (q) circle (1pt);
 \node[font=\scriptsize] at (-1.9,0) {$\mathsf C$};
\end{tikzpicture}
\caption{Geometric interpretation of $r_{\max}$ in the proof of Theorem~\ref{thm:GTD}(ii), shown for $m=6$. The light blue region consists of the points whose metric projection onto $\mathsf C$ equals $Mx_0^\circ$: it is the intersection of the half-planes defined by \eqref{eq:projcone} at $t=0$. Each quotient in \eqref{eq:rmaxdef} is the distance from $x_0^\circ$ to the boundary line of one of these half-planes, and their minimum is $r_{\max}$, so the dashed ball $B(x_0^\circ,r_{\max})$ lies in the region.}
\label{fig:rmax-geometry}
\end{figure}

It remains to show that the hypothesis $\rhoq(s,\beta;\kappa)<q_\kappa$ places $(s,\beta)$ strictly inside the cycling region for at least one period. We first describe the boundaries of these regions. Whenever $s\mapsto P_m(s,\beta;\kappa)$ has real roots, denote them by
\begin{equation*}
 s_-(\beta,m;\kappa)\le s_+(\beta,m;\kappa).
\end{equation*}
These are the step-size thresholds denoted $\gamma_\pm$ in \citep[Notation~B.1]{GoujaudTaylorDieuleveut2025}, renamed here because $\gamma$ denotes the Langevin friction. By \citep[Notation~B.1 and Fact~B.3]{GoujaudTaylorDieuleveut2025}, there is a threshold $\beta_-(m;\kappa)$ such that the roots are real whenever $\beta\ge\beta_-(m;\kappa)$ and, for $s>0$,
\begin{equation*}
 P_m(s,\beta;\kappa)\le0
 \quad\Longleftrightarrow\quad
 \beta\ge\beta_-(m;\kappa)
 \ \text{ and }\
 s_-(\beta,m;\kappa)\le s\le s_+(\beta,m;\kappa).
\end{equation*}

To apply Theorem~\ref{thm:GTD}(ii) we need a single period $m$ with $P_m(s,\beta;\kappa)<0$, that is, $s$ must lie strictly between $s_-(\beta,m;\kappa)$ and $s_+(\beta,m;\kappa)$. The next lemma ensures that consecutive intervals $[s_-(\beta,m;\kappa),s_+(\beta,m;\kappa)]$ overlap on a set with nonempty interior, and Lemma~\ref{lem:interiorcycle} combines the overlap with \eqref{eq:acceleratedquad} to produce a period satisfying the strict inequality.

\begin{lemma}[Overlap of consecutive intervals]\label{lem:strictB6}
Assume $\kappa^{-1}<\bigl(\tfrac{3-\sqrt5}{4}\bigr)^2$. For every integer $m\ge3$ and every $\beta\in(0,1)$ with $\beta\ge\beta_-(m+1;\kappa)$,
\begin{equation}\label{eq:strictB6}
 s_-(\beta,m;\kappa)<s_+(\beta,m+1;\kappa).
\end{equation}
\end{lemma}

\begin{proof}
This sharpens \citep[Lemma~B.6]{GoujaudTaylorDieuleveut2025}; we follow their argument, keeping each inequality strict. Write $u=\kappa^{-1}$ and adopt the notation of their Notation~B.1, in normalized variables:
\begin{equation}\label{eq:AB-def}
 A_m=\beta-\cos\theta_m+u(1-\beta\cos\theta_m),
 \quad
 B_m=\sqrt{A_m^2-2u(1-\cos\theta_m)(1+\beta^2-2\beta\cos\theta_m)},
\end{equation}
so that $s_\pm(\beta,m;\kappa)=A_m\pm B_m$. We first claim the strict form of their inequality~(23),
\begin{equation}\label{eq:B6-23}
 (A_m-A_{m+1})^2<B_m^2-B_{m+1}^2 .
\end{equation}
Assuming \eqref{eq:B6-23} for the moment, we conclude as follows: since $B_{m+1}^2\ge0$, we get $(A_m-A_{m+1})^2<B_m^2$, so $B_m>0$ and $A_m-A_{m+1}\le\abs{A_m-A_{m+1}}<B_m$, whence
\begin{equation*}
 s_-(\beta,m;\kappa)=A_m-B_m<A_{m+1}\le A_{m+1}+B_{m+1}=s_+(\beta,m+1;\kappa),
\end{equation*}
which is \eqref{eq:strictB6}.

To prove \eqref{eq:B6-23}, expand as in \citep[equations~(25)-(26)]{GoujaudTaylorDieuleveut2025}: from \eqref{eq:AB-def},
\begin{align*}
 A_m-A_{m+1}&=(1+\beta u)(\cos\theta_{m+1}-\cos\theta_m),\\
 B_m^2-B_{m+1}^2&=\bigl[2\beta(1-u)^2-(1-\beta u)^2(\cos\theta_{m+1}+\cos\theta_m)\bigr](\cos\theta_{m+1}-\cos\theta_m).
\end{align*}
Since $0<\theta_{m+1}<\theta_m<\pi$ for $m\ge3$ and $\cos$ is strictly decreasing on $[0,\pi]$, we have $\cos\theta_{m+1}-\cos\theta_m>0$; dividing by it shows that \eqref{eq:B6-23} is equivalent to
\begin{equation*}
 (1+\beta u)^2(\cos\theta_{m+1}-\cos\theta_m)+(1-\beta u)^2(\cos\theta_{m+1}+\cos\theta_m)<2\beta(1-u)^2,
\end{equation*}
whose left-hand side equals $2(1+\beta^2u^2)\cos\theta_{m+1}-4\beta u\cos\theta_m$; dividing by $4\beta u>0$ gives the equivalent strict form of their inequality~(27),
\begin{equation}\label{eq:B6-27}
 \frac12\Bigl(\frac1{\beta u}+\beta u\Bigr)\cos\theta_{m+1}
 \;<\;\frac{(1-u)^2}{2u}+\cos\theta_m .
\end{equation}

For $m=3$ one has $\cos\theta_4=0$ and $\cos\theta_3=-\tfrac12$, so \eqref{eq:B6-27} is independent of $\beta$ and reads $0<\frac{(1-u)^2}{2u}-\frac12$, that is, $u^2-3u+1>0$. The smaller root of this quadratic in $u$ is $\tfrac{3-\sqrt5}{2}$, so the inequality holds for $u<\tfrac{3-\sqrt5}{2}$, and the hypothesis gives $u<\bigl(\tfrac{3-\sqrt5}{4}\bigr)^2<\tfrac{3-\sqrt5}{4}<\tfrac{3-\sqrt5}{2}$.

For $m\ge4$, set
\begin{equation*}
 c=\cos\theta_{m+1},\qquad
 d=\cos\theta_m,\qquad
 \xi_{m+1}=\frac1c-1.
\end{equation*}
Then $c>0$, since $\theta_{m+1}\le2\pi/5<\pi/2$. Since $\beta u\in(0,1)$ and $t\mapsto t+t^{-1}$ is strictly decreasing on $(0,1)$, the left-hand side of \eqref{eq:B6-27} is strictly decreasing in $\beta$. The right-hand side is independent of $\beta$. The bound used in the proof of \citep[Lemma~B.6]{GoujaudTaylorDieuleveut2025}, based on their Lemma~B.2, gives
\begin{equation*}
 \beta_-(m+1;\kappa)
 \;\ge\;
 \beta_0:=c\bigl(1+\sqrt u\,\xi_{m+1}\bigr).
\end{equation*}
Consequently, it is enough to verify \eqref{eq:B6-27} at $\beta=\beta_0$.

Substituting $\beta_0=c(1+\sqrt u\,\xi_{m+1})$ into \eqref{eq:B6-27} and multiplying by $2u$ gives
\begin{equation*}
 \frac1{1+\sqrt u\,\xi_{m+1}}
 +c^2\bigl(1+\sqrt u\,\xi_{m+1}\bigr)u^2
 <
 (1-u)^2+2ud.
\end{equation*}
We use
\begin{equation*}
 \frac1{1+\sqrt u\,\xi_{m+1}}
 \le
 1-\sqrt u\,\xi_{m+1}+u\xi_{m+1}^2,
\end{equation*}
which holds because $(1+t)(1-t+t^2)=1+t^3\ge1$ for $t=\sqrt u\,\xi_{m+1}\ge0$. Moreover, since $\sqrt u<1$,
\begin{equation*}
 c^2\bigl(1+\sqrt u\,\xi_{m+1}\bigr)
 \le
 c^2(1+\xi_{m+1})
 =c
 <1.
\end{equation*}
Thus the left-hand side is strictly smaller than
\begin{equation*}
 1-\sqrt u\,\xi_{m+1}+u\xi_{m+1}^2+u^2.
\end{equation*}
Since the right-hand side equals
\begin{equation*}
 1-2u(1-d)+u^2,
\end{equation*}
inequality \eqref{eq:B6-27} follows provided
\begin{equation*}
 \sqrt u\,\xi_{m+1}-u\xi_{m+1}^2>2u(1-d).
\end{equation*}
Dividing by $u\xi_{m+1}>0$, this condition is equivalent to
\begin{equation}\label{eq:B6-final}
 \frac1{\sqrt u}
 \;>\;
 \xi_{m+1}
 +\frac{2(1-\cos\theta_m)}{\xi_{m+1}} .
\end{equation}
By \citep[Lemma~B.7]{GoujaudTaylorDieuleveut2025}, $1\le(1-\cos\theta_m)/(1-\cos\theta_{m+1})\le\tfrac32$, and the resulting elementary bound in their proof shows that the right-hand side of \eqref{eq:B6-final} is at most $1/\cos\theta_5+2=3+\sqrt5$ for every $m\ge4$. Since $u<\bigl(\tfrac{3-\sqrt5}{4}\bigr)^2=(3+\sqrt5)^{-2}$, we have $1/\sqrt u>3+\sqrt5$, so \eqref{eq:B6-final} holds strictly, and the chain of sufficient bounds above yields the strict inequality \eqref{eq:B6-27} for all $\beta\ge\beta_-(m+1;\kappa)$, completing the proof.
\end{proof}

\begin{lemma}[Existence of a period with $P_m(s,\beta;\kappa)<0$]\label{lem:interiorcycle}
Let $C_\star>C_{\mathrm{GTD}}=(3+\sqrt5)^2$, let $\kappa\ge2C_\star$, and let $(s,\beta)\in(0,\infty)\times(0,1)$ satisfy the numerical stability condition \eqref{eq:HBstability} with
\begin{equation*}
 \rhoq(s,\beta;\kappa)<q_\kappa=\frac{1-C_\star/\kappa}{1+C_\star/\kappa}.
\end{equation*}
Then there exists an integer $m\ge3$ with
\begin{equation}\label{eq:strictPK}
 P_m(s,\beta;\kappa)<0.
\end{equation}
\end{lemma}

\begin{proof}
Write $u=\kappa^{-1}$, $\rho=\rhoq(s,\beta;\kappa)$, and $g=(1-u)/(1+u)$, and abbreviate $\beta_-(m)=\beta_-(m;\kappa)$, $s_\pm(\beta,m)=s_\pm(\beta,m;\kappa)$. Note $u\le(2C_\star)^{-1}<\tfrac12(3+\sqrt5)^{-2}$, so, since $(3-\sqrt5)(3+\sqrt5)=4$, in particular $u<\bigl(\tfrac{3-\sqrt5}{4}\bigr)^2$ and $u<1/16$; all results of \citep[Appendix~B]{GoujaudTaylorDieuleveut2025} invoked below apply.

\emph{Step 1: $\beta>13/42$.} By \citep[Corollary~2.3]{GoujaudTaylorDieuleveut2025}, $\rho\ge\rho_\ast(\kappa)=(1-\sqrt u)/(1+\sqrt u)$, the worst-case contraction factor of the optimal quadratic tuning defined in Section~\ref{sec:intro-main}; abbreviate $\rho_\ast=\rho_\ast(\kappa)$. For $t\in(0,g)$, define
\begin{equation*}
 \ell(t):=\frac{t(g-t)}{1-gt}.
\end{equation*}
The level-set parametrization \citep[Lemma~2.4]{GoujaudTaylorDieuleveut2025}, applied to the fixed pair $(s,\beta)$, gives
\begin{equation}\label{eq:LS-beta}
 \ell(\rho)\le\beta\le\rho^2.
\end{equation}
Moreover,
\begin{equation*}
 \ell'(t)=\frac{g-2t+gt^2}{(1-gt)^2}.
\end{equation*}
The roots of the numerator are $\rho_\ast$ and $\rho_\ast^{-1}>1$, so $\ell$ is strictly decreasing on $[\rho_\ast,g]$. By Corollary~2.3 and the hypothesis of the lemma,
\begin{equation*}
 \rho_\ast\le\rho<q_\kappa.
\end{equation*}
Since $C_\star>1$,
\begin{equation*}
 q_\kappa
 =\frac{1-C_\star u}{1+C_\star u}
 <
 \frac{1-u}{1+u}
 =g.
\end{equation*}
Thus $0<\rho_\ast\le\rho<q_\kappa<g$, and consequently
\begin{equation*}
 \beta\ \ge\ \ell(\rho)\ >\ \ell(q_\kappa).
\end{equation*}
A direct computation gives the exact identity
\begin{equation*}
 \frac{g-q_\kappa}{1-gq_\kappa}=\frac{C_\star-1}{C_\star+1},
\end{equation*}
so that $\ell(q_\kappa)=q_\kappa\,\frac{C_\star-1}{C_\star+1}$. Since $C_\star u\le\frac12$, we get $q_\kappa\ge\frac{1-1/2}{1+1/2}=\frac13$, and since $C_\star>C_{\mathrm{GTD}}>27$,
\begin{equation*}
 \frac{C_\star-1}{C_\star+1}>\frac{26}{28}=\frac{13}{14},
\end{equation*}
so $\beta>\ell(q_\kappa)>\frac13\cdot\frac{13}{14}=\frac{13}{42}>\frac15$.

\emph{Step 2: a strict lower bound on the step size.} Since $C_\star>50/3$, we have $\rho<q_\kappa<\frac{1-(50/3)u}{1+(50/3)u}$. If $s\le\frac{50}{3}u(1-\beta)$ held, then combining $\beta\le\rho^2$ from \eqref{eq:LS-beta} and the further level-set consequence $s\ge(1-\rho)(1-\beta/\rho)$ of \citep[Lemma~2.4]{GoujaudTaylorDieuleveut2025} with the algebra in \citep[Section~B.3.2]{GoujaudTaylorDieuleveut2025} would force $\rho\ge\frac{1-(50/3)u}{1+(50/3)u}$, a contradiction. Hence
\begin{equation}\label{eq:s-lower}
 s>\frac{50}{3}\,u\,(1-\beta).
\end{equation}
By \citep[Lemma~B.8]{GoujaudTaylorDieuleveut2025} there is an integer $m_0\ge2$ with
\begin{equation}\label{eq:Cbeta}
 \frac23\ \le\ C_\beta:=\frac{\beta-\cos\theta_{m_0}}{1-\beta}\ \le\ \frac32 .
\end{equation}
For $m_0=2$ one has $C_\beta=(1+\beta)/(1-\beta)>3/2$ because $\beta>1/5$; hence $m_0\ge3$. By \citep[Lemma~B.9]{GoujaudTaylorDieuleveut2025}, whose proof uses only $u\le1/16$ and the lower bound in \eqref{eq:Cbeta}, we get $\beta\ge\beta_-(m_0)$, so $s_\pm(\beta,m_0)$ are real, and the computation leading to \citep[equation~(29)]{GoujaudTaylorDieuleveut2025} gives
\begin{equation}\label{eq:gammaminus-bound}
\begin{split}
 s_-(\beta,m_0)
 &\le\frac{2u(1-\cos\theta_{m_0})(1+\beta^2-2\beta\cos\theta_{m_0})}{\beta-\cos\theta_{m_0}}\\
 &\le4\Bigl(C_\beta+\frac1{C_\beta}+2\Bigr)u(1-\beta)
 \le\frac{50}{3}\,u\,(1-\beta).
\end{split}
\end{equation}
Combining \eqref{eq:s-lower} and \eqref{eq:gammaminus-bound}, $s>s_-(\beta,m_0)$.

\emph{Step 3: selection of a period with strict cycling.} For every integer $3\le m\le m_0$ we have $\cos\theta_m\le\cos\theta_{m_0}$, hence $(\beta-\cos\theta_m)/(1-\beta)\ge C_\beta\ge2/3$, and \citep[Lemma~B.9]{GoujaudTaylorDieuleveut2025} again yields
\begin{equation}\label{eq:beta-admissible}
 \beta\ \ge\ \beta_-(m),\qquad 3\le m\le m_0 .
\end{equation}
Let
\begin{equation*}
 \bar m:=\min\{m\in\{3,\ldots,m_0\}:\ s>s_-(\beta,m)\},
\end{equation*}
which is well defined by Step~2. We claim $s<s_+(\beta,\bar m)$. If $\bar m=3$, then by \citep[Remark~B.5 and the initialization of the proof of Theorem~B.4]{GoujaudTaylorDieuleveut2025}, $s_+(\beta,3)\ge2(1+\beta)/\kappa$, while numerical stability \eqref{eq:HBstability} gives $s<2(1+\beta)/\kappa$; hence $s<s_+(\beta,3)$. If $\bar m\ge4$ and $s\ge s_+(\beta,\bar m)$ held, then Lemma~\ref{lem:strictB6} applied with $m=\bar m-1\ge3$, which is legitimate by \eqref{eq:beta-admissible}, would give
\begin{equation*}
 s_-(\beta,\bar m-1)<s_+(\beta,\bar m)\le s,
\end{equation*}
contradicting the minimality of $\bar m$. Therefore
\begin{equation*}
 s_-(\beta,\bar m)<s<s_+(\beta,\bar m).
\end{equation*}
In particular $s$ lies strictly between the two roots, so $P_{\bar m}(s,\beta;\kappa)<0$, proving \eqref{eq:strictPK} with $m=\bar m$.
\end{proof}

The next proposition puts the deterministic construction into a form suited to OBABO. Because the local Hessian at every cycle point of the Goujaud-Taylor-Dieuleveut potential function equals the identity, the linearization along the cycle is a constant-coefficient system, and its stability follows from numerical stability alone.

\begin{proposition}[Smooth attracting cycle with stable linearization]\label{prop:smoothcycle}
Assume the hypotheses of Lemma~\ref{lem:interiorcycle}, and let $m\ge3$ satisfy the strict cycling condition \eqref{eq:strictPK}. Then there exist distinct points $x_0^\circ,\ldots,x_{m-1}^\circ\in\R^2$ on the unit circle, a potential $U\in\cU_\kappa^2$, and constants $r_0>0$, $b_\ast<\infty$ such that, with indices modulo $m$,
\begin{equation}\label{eq:cycleequation-general}
 x_{j+1}^\circ=(1+\beta)x_j^\circ-\beta x_{j-1}^\circ-s\nabla U(x_j^\circ),
\end{equation}
\begin{equation}\label{eq:localaffine-general}
 \nabla U(x_j^\circ+u)=\nabla U(x_j^\circ)+u,
 \qquad \abs{u}\le r_0,
\end{equation}
and
\begin{equation}\label{eq:boundedremainder-general}
 \nabla U(x)=x+w(x),\qquad \sup_{x\in\R^2}\abs{w(x)}\le b_\ast.
\end{equation}
Consequently, $\nabla^2U=\Id_2$ on each ball $B(x_j^\circ,r_0)$, and the Jacobian of the deterministic heavy-ball update
\begin{equation*}
 (x,y)\longmapsto
 \bigl((1+\beta)x-\beta y-s\nabla U(x),\;x\bigr)
\end{equation*}
at every cycle state $(x_j^\circ,x_{j-1}^\circ)$ is
\begin{equation}\label{eq:periodic-linearization}
 \mathsf A_1=
 \begin{pmatrix}(1+\beta-s)\Id_2&\;-\beta\Id_2\\
                 \Id_2&\;0\end{pmatrix}.
\end{equation}
Therefore, the Jacobian of the $m$-step update at each cycle state is $\mathsf M=\mathsf A_1^m$, and
\begin{equation*}
 \rho(\mathsf M)
 =\rho(\mathsf A_1)^m
 \le\rhoq(s,\beta;\kappa)^m
 <1.
\end{equation*}
\end{proposition}

\begin{proof}
Apply Theorem~\ref{thm:GTD} with this $m$: since $P_m(s,\beta;\kappa)<0$, the $m$th roots of unity $x_t^\circ$, which are distinct points on the unit circle, satisfy \eqref{eq:GTD-cycle-eq} for the piecewise-quadratic potential function $\psi\in C^1$ of \eqref{eq:GTDpsi}, which is $1$-strongly convex and has $\kappa$-Lipschitz gradient, and \eqref{eq:psi-local} holds with $r_{\max}>0$.

Let $\varrho_\epsilon$ be a centered $C^\infty$ probability density supported in $B(0,\epsilon)$, with $0<\epsilon<r_{\max}/4$, and set $U=\varrho_\epsilon*\psi$, which is $C^\infty$. Since $\psi$ is $1$-strongly convex with $\kappa$-Lipschitz gradient,
\begin{equation*}
 \abs{x-y}^2\le\ip{\nabla\psi(x)-\nabla\psi(y)}{x-y},
 \qquad
 \abs{\nabla\psi(x)-\nabla\psi(y)}\le\kappa\abs{x-y}
\end{equation*}
for all $x,y\in\R^2$. Since $\psi\in C^1$ and $\varrho_\epsilon\in C_c^\infty$, we have $\nabla U=\varrho_\epsilon*\nabla\psi$. Fix $x,y\in\R^2$. For each $w\in\R^2$ the first inequality applied to the two points $x-w$ and $y-w$, whose difference is again $x-y$, gives
\begin{equation*}
 \ip{\nabla\psi(x-w)-\nabla\psi(y-w)}{x-y}\ge\abs{x-y}^2.
\end{equation*}
Integrating this in $w$ against the probability density $\varrho_\epsilon$ yields
\begin{equation*}
 \ip{\nabla U(x)-\nabla U(y)}{x-y}
 =\int\ip{\nabla\psi(x-w)-\nabla\psi(y-w)}{x-y}\varrho_\epsilon(w)\,dw
 \ge\abs{x-y}^2.
\end{equation*}
The same substitution in the second inequality gives $\abs{\nabla\psi(x-w)-\nabla\psi(y-w)}\le\kappa\abs{x-y}$ for every $w$, and integrating against $\varrho_\epsilon$ gives $\abs{\nabla U(x)-\nabla U(y)}\le\kappa\abs{x-y}$. Because $U$ is $C^\infty$, these two bounds pass to the Hessian: taking $y=x+tv$ with $\abs v=1$, dividing the first by $t^2$ and the second by $t$, and letting $t\downarrow0$ gives $\ip{\nabla^2U(x)v}v\ge1$ and $\abs{\nabla^2U(x)v}\le\kappa$ for every unit vector $v$. Since $\nabla^2U(x)$ is symmetric, its eigenvalues therefore lie in $[1,\kappa]$, so $U\in\cU_\kappa^2$.

For $\abs{u}\le r_0:=r_{\max}/2$ and $w\in\supp\varrho_\epsilon$, equation \eqref{eq:psi-local} applies at $x_j^\circ+u-w$, and centering gives
\begin{align*}
 \nabla U(x_j^\circ+u)
 &=\int\nabla\psi(x_j^\circ+u-w)\varrho_\epsilon(w)\,dw\\
 &=\nabla\psi(x_j^\circ)+u-\int w\varrho_\epsilon(w)\,dw
 =\nabla\psi(x_j^\circ)+u.
\end{align*}
Thus the gradients at the cycle points are unchanged by smoothing, so \eqref{eq:GTD-cycle-eq} becomes \eqref{eq:cycleequation-general}, and \eqref{eq:localaffine-general} holds. Convolving the gradient formula in \eqref{eq:GTDpsi} yields \eqref{eq:boundedremainder-general} with $b_\ast=(\kappa-1)\max_{y\in\mathsf C}\abs y$, because $\proj_{\mathsf C}$ takes values in the compact set $\mathsf C$.

Finally, differentiating \eqref{eq:localaffine-general} gives $\nabla^2U=\Id_2$ on each ball $B(x_j^\circ,r_0)$. The Jacobian of this update at $(x,y)$ is
\begin{equation*}
 \begin{pmatrix}(1+\beta)\Id_2-s\nabla^2U(x)&\;-\beta\Id_2\\
                 \Id_2&\;0\end{pmatrix},
\end{equation*}
which at $x=x_j^\circ$ reduces to $\mathsf A_1$, proving \eqref{eq:periodic-linearization}; since the update maps each cycle state to the next by \eqref{eq:cycleequation-general}, the chain rule gives $\mathsf M=\mathsf A_1^m$ as the Jacobian of the $m$-step update. In the block form \eqref{eq:periodic-linearization}, the two coordinates of $\R^2$ do not mix: writing $x=(x^{(1)},x^{(2)})$ and $y=(y^{(1)},y^{(2)})$, the linear map $\mathsf A_1$ sends each pair $(x^{(i)},y^{(i)})$ to $\bigl((1+\beta-s)x^{(i)}-\beta y^{(i)},\,x^{(i)}\bigr)$, that is, it acts on each pair by the companion matrix $A_1(s,\beta)$ of \eqref{eq:A_lambda_intro} with $\lambda=1$. Hence $\rho(\mathsf A_1)=\rho(A_1(s,\beta))$. Since $\rhoq(s,\beta;\kappa)$ is the maximum in \eqref{eq:A_lambda_intro} and $\lambda=1$ lies in $[1,\kappa]$, the spectral mapping theorem gives
\begin{equation*}
 \rho(\mathsf M)=\rho(\mathsf A_1)^m=\rho(A_1(s,\beta))^m\le\rhoq(s,\beta;\kappa)^m<q_\kappa^m<1.
\end{equation*}
\end{proof}

Only Lemma~\ref{lem:interiorcycle} and Proposition~\ref{prop:smoothcycle} are used beyond this section; Theorem~\ref{thm:GTD} and Lemma~\ref{lem:strictB6} enter only through their proofs. Both are used in the proof of Theorem~\ref{thm:dichotomy} in Section~\ref{sec:proofmain}. When $\rhoq(s,\beta;\kappa)<q_\kappa$, Lemma~\ref{lem:interiorcycle} produces a period $m$ with $P_m(s,\beta;\kappa)<0$, and the potential of Proposition~\ref{prop:smoothcycle} satisfies the hypotheses of Theorem~\ref{thm:transfer} in Section~\ref{sec:transfer}: the cycle equation \eqref{eq:cycleequation-general} and the local identity \eqref{eq:localaffine-general} become Assumption~\ref{ass:cycle}, the spectral radius bound on $\mathsf M$ becomes Assumption~\ref{ass:monodromy}, and \eqref{eq:boundedremainder-general} gives the global gradient bound \eqref{eq:transfer-bounded}. Theorem~\ref{thm:transfer} then yields the metastability asserted in case \emph{(C)} of Theorem~\ref{thm:dichotomy}. Since Lemma~\ref{lem:interiorcycle} and Proposition~\ref{prop:smoothcycle} depend only on the pair $(s,\beta)$, Section~\ref{sec:BAOAB} applies them unchanged to the remaining Strang splittings.

\section{Proofs of the mixing time lower bounds}\label{sec:proofs}

This section proves the mixing time lower bounds. Section~\ref{sec:quadraticbranch} treats the quadratic case. Section~\ref{sec:metastability-lemma} proves a general metastability lemma for noisy heavy-ball recursions, and Section~\ref{sec:transfer} applies it to derive OBABO metastability from an attracting cycle of the heavy-ball recursion. Section~\ref{sec:proofmain} proves Theorem~\ref{thm:dichotomy} and Corollary~\ref{cor:mixing}. Section~\ref{sec:LRP} presents the explicit instance of Corollary~\ref{cor:LRPexplicit}.

\subsection{The quadratic case}\label{sec:quadraticbranch}

We begin with an elementary fact about stable Gaussian autoregressive processes.

\pagebreak
\begin{lemma}[Total variation decay rates of Gaussian autoregressive processes]\label{lem:gaussianTVrate}
Let
\begin{equation*}
 Y_{n+1}=MY_n+G\xi_{n+1}
\end{equation*}
be a Gaussian autoregressive process on $\R^N$, with transition kernel $Q$, where $M\in\R^{N\times N}$ satisfies $\rho(M)<1$, $G\in\R^{N\times N'}$, the $\xi_n$ are i.i.d.\ standard Gaussian vectors in $\R^{N'}$, and the covariance $\Sigma_{n_0}$ of the $n_0$-step transition law is nonsingular for some $n_0\ge1$; this holds in particular when $GG^\top\succ0$, with $n_0=1$. Then the chain has a unique invariant law $\pi$, which is a centered Gaussian. Moreover,
\begin{equation}\label{eq:TVonepoint}
 \limsup_{n\to\infty}\norm{\delta_zQ^n-\pi}_{\TV}^{1/n}\le\rho(M)
 \qquad\text{for every initial state }z,
\end{equation}
and the limit exists and equals $\rho(M)$ for a suitable $z$. Finally, there are initial states $z,\widetilde z$ such that
\begin{equation}\label{eq:TVrootM}
 \lim_{n\to\infty}
 \norm{\delta_zQ^n-\delta_{\widetilde z}Q^n}_{\TV}^{1/n}=\rho(M).
\end{equation}
\end{lemma}

\begin{proof}
Set
\begin{equation*}
 \Sigma_n=\sum_{j=0}^{n-1}M^jGG^\top(M^\top)^j,
 \qquad
 \Sigma_\infty=\sum_{j=0}^{\infty}M^jGG^\top(M^\top)^j.
\end{equation*}
Since $\rho(M)<1$, the second series converges. Each summand is
positive semidefinite, so
$\Sigma_{n_0}\preceq\Sigma_n\preceq\Sigma_\infty$ for $n\ge n_0$,
and in particular $\Sigma_\infty\succ0$. Moreover,
\begin{equation*}
 \Sigma_\infty=GG^\top+M\Sigma_\infty M^\top,
\end{equation*}
so $\cN(0,\Sigma_\infty)$ is invariant.

The law $\cN(0,\Sigma_\infty)$ is also the unique invariant law. Indeed, iterating the
autoregression gives $Y_n=M^nY_0+\zeta_n$ with
$\zeta_n\sim\cN(0,\Sigma_n)$ independent of $Y_0$. If $\pi$ is
invariant and $\widehat\pi$ is its characteristic function, taking
$Y_0\sim\pi$ therefore yields
\begin{equation*}
 \widehat\pi(t)
 =\widehat\pi\bigl((M^\top)^nt\bigr)
 \exp\left(-\frac12t^\top\Sigma_nt\right).
\end{equation*}
Since $\rho(M)<1$, we have $(M^\top)^nt\to0$. Every
characteristic function is continuous at zero and satisfies
$\widehat\pi(0)=1$, so $\widehat\pi\bigl((M^\top)^nt\bigr)\to1$.
Since also $\Sigma_n\to\Sigma_\infty$, passing to the limit in the
preceding identity gives
$\widehat\pi(t)=\exp\left(-\frac12t^\top\Sigma_\infty t\right)$,
and thus $\pi=\cN(0,\Sigma_\infty)$; we write $\pi$ for this law below.

The remaining assertions compare Gaussian laws that differ in their
means and in their covariances, and we prepare one estimate for each
source of error. For the means: for $\Sigma\succ0$, equal-covariance
Gaussians satisfy
\begin{equation}\label{eq:equalcovTV}
 \norm{\cN(m,\Sigma)-\cN(\widetilde m,\Sigma)}_{\TV}
 =2\Phi\left(\frac12\abs{\Sigma^{-1/2}(m-\widetilde m)}\right)-1,
\end{equation}
where $\Phi$ is the standard normal distribution function; since
$\Phi$ is $1/\sqrt{2\pi}$-Lipschitz, in particular
$2\Phi(t/2)-1\le t/\sqrt{2\pi}$ for $t\ge0$. For the covariances:
reindexing the tail of the series defining $\Sigma_\infty$ gives
\begin{equation*}
 \Sigma_\infty-\Sigma_n=M^n\Sigma_\infty(M^\top)^n,
\end{equation*}
so the matrix
$E_n=\Sigma_\infty^{-1/2}(\Sigma_\infty-\Sigma_n)\Sigma_\infty^{-1/2}$
satisfies
\begin{equation*}
 0\preceq E_n,
 \qquad
 \norm{E_n}
 \le\frac{\lambda_{\max}(\Sigma_\infty)}{\lambda_{\min}(\Sigma_\infty)}
 \norm{M^n}^2.
\end{equation*}
By Gelfand's formula $\rho(M)=\lim_{n\to\infty}\norm{M^n}^{1/n}$
and $\rho(M)<1$, we have $\norm{M^n}\to0$; in particular,
$E_n\preceq\tfrac12\Id_N$ for all sufficiently large $n$. If $e_{n,1},\dots,e_{n,N}$ denote the eigenvalues of $E_n$, the
closed form of the Kullback-Leibler divergence between centered
Gaussian laws gives
\begin{equation*}
 D_{\mathrm{KL}}\bigl(\cN(0,\Sigma_n)\,\big\|\,\cN(0,\Sigma_\infty)\bigr)
 =\frac12\sum_{i=1}^N\bigl\{-e_{n,i}-\log(1-e_{n,i})\bigr\}
 \le\sum_{i=1}^Ne_{n,i}^2
 \le N\norm{E_n}^2,
\end{equation*}
using $-t-\log(1-t)\le2t^2$ for $t\in[0,\tfrac12]$. Pinsker's
inequality
$2\norm{\nu-\widetilde\nu}_{\TV}^2
\le D_{\mathrm{KL}}(\nu\,\|\,\widetilde\nu)$
then gives, for all sufficiently large $n$,
\begin{equation*}
 \norm{\cN(0,\Sigma_n)-\cN(0,\Sigma_\infty)}_{\TV}
 \le\sqrt{N/2}\,\norm{E_n}
 \le C_0\norm{M^n}^2,
 \qquad
 C_0=\sqrt{N/2}\,
 \frac{\lambda_{\max}(\Sigma_\infty)}{\lambda_{\min}(\Sigma_\infty)}.
\end{equation*}

If $\rho(M)=0$, then every eigenvalue of $M$ is zero, so the
Cayley-Hamilton theorem gives $M^N=0$. For $n\ge N$, the mean
$M^nz$ vanishes and $\Sigma_n=\Sigma_\infty$, so $\delta_zQ^n=\pi$
for every initial state $z$, and any two transition laws coincide.
Hence \eqref{eq:TVonepoint} holds with equality as a limit for every
$z$, and \eqref{eq:TVrootM} holds, with all quantities equal to
zero.

Suppose now that $\rho(M)>0$. Choose an eigenvalue $\lambda$ of
$M$ with $\abs\lambda=\rho(M)$. If $\lambda$ is real, take a
corresponding real eigenvector. If $\lambda$ is complex, the real
and imaginary parts of a corresponding complex eigenvector span a
two-dimensional real invariant subspace on which $M$ is similar to
$\rho(M)$ times a rotation. In either case, there are
$u\in\R^N\setminus\{0\}$ and constants $c,C>0$ such that
\begin{equation*}
 c\rho(M)^n\le\abs{M^nu}\le C\rho(M)^n,
 \qquad n\ge0.
\end{equation*}
Consequently, $\lim_{n\to\infty}\abs{M^nu}^{1/n}=\rho(M)$.

Choose any two initial states $z,\widetilde z$ such that
$z-\widetilde z=u$. Their $n$-step laws are Gaussian with common
covariance $\Sigma_n$ and means differing by $M^nu$. Therefore,
\eqref{eq:equalcovTV} gives, for $n\ge n_0$,
\begin{equation*}
 \norm{\delta_zQ^n-\delta_{\widetilde z}Q^n}_{\TV}
 =2\Phi(t_n/2)-1,
 \qquad
 t_n=\abs{\Sigma_n^{-1/2}M^nu}.
\end{equation*}
Since $\Sigma_{n_0}\preceq\Sigma_n\preceq\Sigma_\infty$, we have
\begin{equation*}
 \lambda_{\max}(\Sigma_\infty)^{-1/2}\abs{M^nu}
 \le t_n\le
 \lambda_{\min}(\Sigma_{n_0})^{-1/2}\abs{M^nu}.
\end{equation*}
Taking $n$th roots and using the choice of $u$ yields $\lim_{n\to\infty}t_n^{1/n}=\rho(M)$. Because $\rho(M)<1$, this also implies $t_n\to0$. Since $2\Phi(t/2)-1\sim t/\sqrt{2\pi}$ as $t\downarrow0$, we have $\bigl((2\Phi(t_n/2)-1)/t_n\bigr)^{1/n}\to1$. Consequently,
\begin{equation*}
 \lim_{n\to\infty}
 \norm{\delta_zQ^n-\delta_{\widetilde z}Q^n}_{\TV}^{1/n}
 =\lim_{n\to\infty}
 t_n^{1/n}\left(\frac{2\Phi(t_n/2)-1}{t_n}\right)^{1/n}
 =\rho(M),
\end{equation*}
which proves \eqref{eq:TVrootM}.

Next, fix an arbitrary initial state $z$. By the triangle
inequality and the two estimates above, for all sufficiently large
$n$,
\begin{align*}
 \norm{\delta_zQ^n-\pi}_{\TV}
 &\le\norm{\cN(M^nz,\Sigma_n)-\cN(0,\Sigma_n)}_{\TV}
 +\norm{\cN(0,\Sigma_n)-\cN(0,\Sigma_\infty)}_{\TV}\\
 &\le\frac{\lambda_{\min}(\Sigma_{n_0})^{-1/2}}{\sqrt{2\pi}}
 \abs{M^nz}+C_0\norm{M^n}^2.
\end{align*}
Since $\abs{M^nz}\le\norm{M^n}\abs z$ and, by Gelfand's formula,
$\norm{M^n}^{1/n}\to\rho(M)<1$, the $n$th root of the right-hand
side has limit superior at most $\rho(M)$, which proves
\eqref{eq:TVonepoint}.

Finally, take the initial state $u$. For all sufficiently large
$n$, the reverse triangle inequality and the covariance estimate
above give
\begin{align}
 \norm{\delta_uQ^n-\pi}_{\TV}
 &\ge2\Phi(t_n/2)-1-C_0\norm{M^n}^2\nonumber\\
 &\ge\frac{t_n}{2\sqrt{2\pi}}-C_0\norm{M^n}^2,
 \label{eq:TVlower}
\end{align}
where the second inequality uses
$2\Phi(t/2)-1\ge t/(2\sqrt{2\pi})$ for all sufficiently small
$t\ge0$. We showed above that $t_n^{1/n}\to\rho(M)$, whereas
Gelfand's formula gives
$\bigl(C_0\norm{M^n}^2\bigr)^{1/n}\to\rho(M)^2<\rho(M)$.
Consequently, the right-hand side of \eqref{eq:TVlower} is
positive for all sufficiently large $n$, and its $n$th root
converges to $\rho(M)$. Thus \eqref{eq:TVlower} gives
$\liminf_{n\to\infty}\norm{\delta_uQ^n-\pi}_{\TV}^{1/n}
\ge\rho(M)$. Combined with \eqref{eq:TVonepoint}, this proves
that the limit exists and equals $\rho(M)$ for the initial state
$u$.
\end{proof}

Beyond the quadratic case treated in this section and its extensions in Section~\ref{sec:BAOAB}, Lemma~\ref{lem:gaussianTVrate} has two further applications, given in the appendices. Applied to the left-endpoint exponential integrator, it shows that every fixed choice of step size and friction is either unstable on a Gaussian target with curvature in $[1,\kappa]$, or admits such a target and a deterministic initial state whose asymptotic total variation contraction factor is at least $1-64/\kappa$ (Proposition~\ref{prop:exp-euler-gaussian}). Applied to the kinetic Langevin diffusion \eqref{eq:KLD} on a Gaussian target with smallest curvature $K$, it shows that, for every friction parameter, some deterministic initial state has asymptotic total variation decay exponent at most $\sqrt K$, and that the resulting worst-case exponent is exactly $\sqrt K$ at critical friction (Proposition~\ref{prop:diffusion-coldstart}).

Recall from \eqref{eq:hbparamsintro} that $q_\kappa=(1-C_\star/\kappa)/(1+C_\star/\kappa)$, the nonaccelerated rate appearing in case \emph{(Q)} of Theorem~\ref{thm:dichotomy}. The next proposition establishes this case: if the worst-case spectral radius $\rhoq(s,\beta;\kappa)$ from \eqref{eq:A_lambda_intro} is at least $q_\kappa$, then Lemma~\ref{lem:gaussianTVrate}, applied to the OBABO chain on a Gaussian target whose phase matrix attains this spectral radius, yields two initial states satisfying \eqref{eq:quadrootrateintro}.

\begin{proposition}\label{prop:quadraticbranch}
Assume the setting of Theorem~\ref{thm:dichotomy}. If $\rhoq(s,\beta;\kappa)\ge q_\kappa$, then case \emph{(Q)} of Theorem~\ref{thm:dichotomy} holds.
\end{proposition}

\begin{proof}
Choose $\lambda\in[1,\kappa]$ attaining the maximum in \eqref{eq:A_lambda_intro}. For $U_\lambda(x)=\lambda\abs{x}^2/2$ on $\R^2$, the OBABO chain is a Gaussian autoregressive process. Let $\cM_\lambda$ denote its phase matrix. Up to a permutation of coordinates, $\cM_\lambda=M_\lambda\oplus M_\lambda$ with $M_\lambda$ from \eqref{eq:Mlambda}, and hence, by \eqref{eq:rhosame},
\begin{equation*}
 \rho(\cM_\lambda)=\rho(M_\lambda)=\rho(A_\lambda(s,\beta))=\rhoq(s,\beta;\kappa)\ge q_\kappa.
\end{equation*}
The one-step noise covariance is nonsingular because the map from the Gaussian variables $(\xi^{(1)},\xi^{(2)})$ to the state in \eqref{eq:obabo1}-\eqref{eq:obabo5} is triangular with diagonal blocks $h\sigma\Id$ and $\sigma\Id$. Numerical stability gives $\rho(\cM_\lambda)<1$. Lemma~\ref{lem:gaussianTVrate}, applied with $M=\cM_\lambda$, $N=N'=4$, and $n_0=1$, gives a unique invariant law and initial states $z,\widetilde z\in\R^4$ with
\begin{equation*}
 \lim_{n\to\infty}
 \norm{\delta_z(Q_{h,\gamma}^{U_\lambda})^n-\delta_{\widetilde z}(Q_{h,\gamma}^{U_\lambda})^n}_{\TV}^{1/n}
 =\rho(\cM_\lambda)=\rhoq(s,\beta;\kappa)\ge q_\kappa,
\end{equation*}
which is \eqref{eq:quadrootrateintro}.
\end{proof}

For comparison, applying Lemma~\ref{lem:gaussianTVrate} under the optimal quadratic tuning recovers the accelerated Gaussian benchmark.

\begin{corollary}[Optimal Gaussian benchmark]\label{cor:optimalGaussian}
Let $\kappa>1$ and set
\begin{equation*}
 h_*=\frac{2}{\sqrt{\kappa+1}},
 \qquad
 \gamma_*=\sqrt{\kappa+1}\,\log\frac{\sqrt\kappa+1}{\sqrt\kappa-1},
\end{equation*}
so that
\begin{equation*}
 \beta_*=e^{-\gamma_*h_*}=\left(\frac{\sqrt\kappa-1}{\sqrt\kappa+1}\right)^2,
 \qquad
 s_*=\frac{h_*^2}{2}(1+\beta_*)=\frac{4}{(\sqrt\kappa+1)^2}.
\end{equation*}
Let $U_H(x)=\frac12x^\top Hx$ on $\R^d$ with $\Id_d\preceq H\preceq\kappa\Id_d$, and write $Q=Q_{h_*,\gamma_*}^{U_H}$. Then $Q$ has a unique invariant law $\pi_H$, and $\pi_H$ is a centered Gaussian. For every initial state $z\in\R^d\times\R^d$,
\begin{equation*}
 \limsup_{n\to\infty}\norm{\delta_zQ^n-\pi_H}_{\TV}^{1/n}
 \le\sqrt{\beta_*}=\frac{\sqrt\kappa-1}{\sqrt\kappa+1},
\end{equation*}
and for suitable $z$ the limit exists and equals $\sqrt{\beta_*}$. Moreover, there are initial states $z,\widetilde z$ with
\begin{equation*}
 \lim_{n\to\infty}\norm{\delta_zQ^n-\delta_{\widetilde z}Q^n}_{\TV}^{1/n}
 =\sqrt{\beta_*}.
\end{equation*}
\end{corollary}

\begin{proof}
Since $\gamma_*h_*=2\log\frac{\sqrt\kappa+1}{\sqrt\kappa-1}$, the stated value of $\beta_*=e^{-\gamma_*h_*}$ holds, and $1+\beta_*=\frac{2(\kappa+1)}{(\sqrt\kappa+1)^2}$ gives $s_*=\frac{h_*^2}{2}(1+\beta_*)=\frac{4}{(\sqrt\kappa+1)^2}$. Since $0<h_*<2/\sqrt\kappa$, the tuning is numerically stable by Lemma~\ref{lem:quadratic-stability}.

Write $H=O\Lambda O^\top$ with $O$ orthogonal and $\Lambda$ diagonal with entries $\lambda_1,\dots,\lambda_d\in[1,\kappa]$. The change of variables $(x,v)\mapsto(O^\top x,O^\top v)$ transforms the chain into $d$ independent two-dimensional chains with phase matrices $M_{\lambda_i}$ from \eqref{eq:Mlambda}, so the chain is a Gaussian autoregressive process $Z_{n+1}=M_HZ_n+G\xi_{n+1}$ with $M_H$ orthogonally similar to the direct sum of the $M_{\lambda_i}$. The coefficient $1+\beta_*-s_*\lambda$ in \eqref{eq:charpoly} is affine in $\lambda$, equal to $2\sqrt{\beta_*}$ at $\lambda=1$ and to $-2\sqrt{\beta_*}$ at $\lambda=\kappa$; hence $\abs{1+\beta_*-s_*\lambda}\le2\sqrt{\beta_*}$ for every $\lambda\in[1,\kappa]$, and both roots of \eqref{eq:charpoly} have modulus $\sqrt{\beta_*}$: real and repeated at the endpoints, a complex conjugate pair for $\lambda\in(1,\kappa)$. Every eigenvalue of $M_H$ therefore has modulus $\sqrt{\beta_*}$, and $\rho(M_H)=\sqrt{\beta_*}<1$.

The one-step noise covariance is nonsingular, as in the proof of Proposition~\ref{prop:quadraticbranch}. All conclusions now follow from Lemma~\ref{lem:gaussianTVrate}, applied with $M=M_H$ and $n_0=1$: the chain has a unique invariant law $\pi_H$, a centered Gaussian; \eqref{eq:TVonepoint} gives the upper bound, with $\rho(M_H)=\sqrt{\beta_*}$, for every initial state, and equality as a limit for a suitable initial state; and \eqref{eq:TVrootM} provides the pair of initial states.
\end{proof}

\subsection{Metastability of noisy heavy-ball recursions}\label{sec:metastability-lemma}

Lemma~\ref{lem:metastability} gives the confinement and separation estimates used in the cycling case. It applies to an abstract noisy heavy-ball recursion and uses the noise variables only through one-step tail bounds; in particular, they need not be independent. Section~\ref{sec:transfer} applies it to OBABO, and Section~\ref{sec:BAOAB} applies it to the remaining Strang splittings.

The underlying small-noise mechanism is classical: stable deterministic attractors can persist for exponentially long times under small random perturbations; see \citep[Chapter~6]{FreidlinWentzell2012}. Discrete-time versions were developed by Kifer \citep{Kifer1990}; see also Faure and Schreiber \citep{FaureSchreiber2014} for randomly perturbed maps near stable periodic orbits. The contribution here is the explicit confinement and total variation separation estimates needed for the mixing time lower bounds. Figure~\ref{fig:transfer-construction} illustrates the argument. The assumptions and notation of Lemma~\ref{lem:metastability} follow.

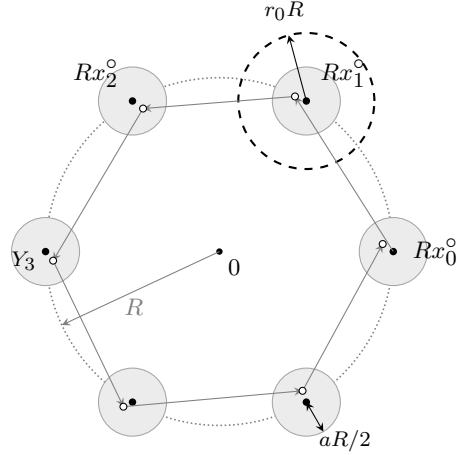
\begin{figure}[t]
\centering
\begin{tikzpicture}[>=stealth]
 \draw[densely dotted,gray,semithick] (0,0) circle (2.3);
 \fill (0,0) circle (1.2pt);
 \node[below right] at (0,0) {$0$};
 \draw[->,gray] (0,0) -- (205:2.3);
 \node[gray] at (213:1.35) {$R$};
 \foreach \a in {0,60,...,300} {
   \fill[gray!15] (\a:2.3) circle (0.45);
   \draw[gray!70] (\a:2.3) circle (0.45);
 }
 \foreach \a in {0,60,...,300} \fill (\a:2.3) circle (1.4pt);
 \node[right=4pt] at (0:2.3) {$Rx_0^\circ$};
 \node[above right=2pt] at (60:2.3) {$Rx_1^\circ$};
 \node[above left=2pt] at (120:2.3) {$Rx_2^\circ$};
 \draw[dashed,black,thick] (60:2.3) circle (0.9);
 \draw[->] (60:2.3) -- +(105:0.9);
 \node[font=\scriptsize] at (75:3.30) {$r_0R$};
 \draw[<->] (300:2.3) -- +(-60:0.45);
 \node[font=\scriptsize] at (1.66,-2.50) {$aR/2$};
 \draw[->,thin,gray] (2.300,0) -- (1.000,2.052);
 \draw[->,thin,gray] (1.000,2.052) -- (-1.010,1.892);
 \draw[->,thin,gray] (-1.010,1.892) -- (-2.200,-0.120);
 \draw[->,thin,gray] (-2.200,-0.120) -- (-1.270,-2.052);
 \draw[->,thin,gray] (-1.270,-2.052) -- (1.100,-1.842);
 \draw[->,thin,gray] (1.100,-1.842) -- (2.160,0.100);
 \foreach \p in {(1.000,2.052),(-1.010,1.892),(-2.200,-0.120),(-1.270,-2.052),(1.100,-1.842),(2.160,0.100)}
   \draw[fill=white] \p circle (1.3pt);
 \node[font=\scriptsize, left=3pt] at (-2.200,-0.120) {$Y_3$};
\end{tikzpicture}
\caption{Schematic of Lemma~\ref{lem:metastability} for a six-point
roots-of-unity cycle. The gray dotted circle has radius $R$ and
passes through the dilated cycle points $Rx_j^\circ$. The black
dashed circle centered at $Rx_1^\circ$ is the boundary of the
ball
$B(Rx_1^\circ,r_0R)$, on which the gradient of $U_R$ has the
linear form specified in \eqref{eq:transfer-local}. The shaded
balls have radius $aR/2$, and the open dots are the noisy
iterates. On the event $\max_{i\le n}\abs{\zeta_i}\le\delta'R$,
each iterate remains in the shaded ball around the corresponding
cycle point. The balls around consecutive cycle points are
disjoint, giving the total variation separation in
Lemma~\ref{lem:metastability}(ii). For OBABO, this event fails by
time $n$ with probability at most $2ne^{-c_0R^2}$.}
\label{fig:transfer-construction}
\end{figure}

\emph{Cycle and linear stability.} Fix $s>0$, $\beta\in(0,1)$, and
integers $d\ge1$ and $m\ge2$, and let $U:\R^d\to\R$ be
differentiable.

\begin{assumption}[Periodic orbit and local gradient identity]\label{ass:cycle}
There exist an $m$-periodic sequence $(x_j^\circ)_{j\in\mathbb Z}$
in $\R^d$ with $x_0^\circ,\dots,x_{m-1}^\circ$ distinct, an
$m$-periodic sequence $(H_j)_{j\in\mathbb Z}$ of symmetric
$d\times d$ matrices, and a constant $r_0>0$ such that, for every
$j\in\mathbb Z$,
\begin{equation}\label{eq:transfer-cycle}
 x_{j+1}^\circ
 =(1+\beta)x_j^\circ-\beta x_{j-1}^\circ
 -s\nabla U(x_j^\circ)
\end{equation}
and
\begin{equation}\label{eq:transfer-local}
 \nabla U(x_j^\circ+u)
 =\nabla U(x_j^\circ)+H_ju,
 \qquad u\in\R^d,\ \abs u\le r_0.
\end{equation}
\end{assumption}

Fix points $(x_j^\circ)_{j\in\mathbb Z}$, matrices
$(H_j)_{j\in\mathbb Z}$, and a constant $r_0>0$ as in
Assumption~\ref{ass:cycle}; indices are understood modulo $m$.

The matrices $H_j$ in \eqref{eq:transfer-local} may vary from one cycle point to another. In the construction used for the main results, $H_j=\Id_d$ at every cycle point, by Proposition~\ref{prop:smoothcycle}. In the explicit instance of Section~\ref{sec:LRP}, which is one-dimensional, $H_j=25$ at every cycle point.

Define
\begin{equation*}
 \mathsf A_j=
 \begin{pmatrix}(1+\beta)\Id_d-sH_j&\;-\beta\Id_d\\
                 \Id_d&\;0\end{pmatrix},
 \qquad
 \mathsf B=\begin{pmatrix}\Id_d\\0\end{pmatrix}.
\end{equation*}
These matrices govern deviations from the cycle; the corresponding
noisy error recursion is derived in the proof of
Lemma~\ref{lem:metastability}.
\begin{assumption}[Spectral radius condition]\label{ass:monodromy}
The matrices $\mathsf A_j$ satisfy
\begin{equation}\label{eq:monodromy-general}
 \rho\bigl(\mathsf A_{m-1}\cdots\mathsf A_0\bigr)<1.
\end{equation}
\end{assumption}

Because $\beta>0$, every $\mathsf A_j$ is invertible. Writing
$\mathsf M_j=\mathsf A_{j+m-1}\cdots\mathsf A_j$ for the
full-cycle products, periodicity gives
$\mathsf M_{j+1}=\mathsf A_j\mathsf M_j\mathsf A_j^{-1}$. The
full-cycle products are therefore similar and have a common
spectral radius, so \eqref{eq:monodromy-general} gives
$\rho(\mathsf M_j)<1$ for every $j\in\mathbb Z$. For $j\in\mathbb Z$ and $0\le\ell<k$, define
\begin{equation*}
 \Psi_j(k,\ell)
 =\mathsf A_{j+k-1}\cdots\mathsf A_{j+\ell},
 \qquad
 \Psi_j(\ell,\ell)=\Id_{2d}.
\end{equation*}
These products decay geometrically. Indeed, write $k-\ell=pm+r$
with integers $p\ge0$ and $0\le r<m$. Grouping the factors of
$\Psi_j(k,\ell)$ into $p$ complete periods and a remainder, and
using $\mathsf A_{i+m}=\mathsf A_i$, gives
\begin{equation*}
 \Psi_j(k,\ell)
 =\Psi_j(k,\ell+pm)\,\mathsf M_{j+\ell}^{\,p},
\end{equation*}
where $\Psi_j(k,\ell+pm)$ is a product of $r$ consecutive
matrices $\mathsf A_i$. Up to periodicity there are only $m$
full-cycle products and finitely many products of fewer than $m$
consecutive matrices $\mathsf A_i$; since every full-cycle
product has spectral radius less than one, Gelfand's formula
gives constants
$C\ge1$ and $q\in(0,1)$ such that
$\norm{\Psi_j(k,\ell)}\le Cq^{k-\ell}$ for all
$j\in\mathbb Z$ and all $0\le\ell\le k$. Fix such a pair
$(C,q)$. In particular,
\begin{equation}\label{eq:impulsegain-general}
 G_0:=\max_{0\le j<m}\sup_{k\ge1}
 \sum_{i=1}^k\norm{\Psi_j(k,i)\mathsf B}
 \le\frac{C}{1-q}<\infty.
\end{equation}
Set
\begin{equation}\label{eq:tuberadius-general}
 d_\circ=\min_{0\le j<m}\abs{x_{j+1}^\circ-x_j^\circ}>0,
 \qquad
 a=\frac14\min\{r_0,d_\circ\}.
\end{equation}

\emph{Noisy recursion.} Fix $R\ge1$, $j\in\mathbb Z$, an
integer $n\ge1$, and constants $\varepsilon_0\ge0$ and
$\delta'>0$, and set $U_R(x)=R^2U(x/R)$.

Let $(Y_k)_{k\ge-1}$ and $(\zeta_i)_{i\ge1}$ be $\R^d$-valued
random sequences, defined on a common probability space,
satisfying
\begin{equation}\label{eq:abstract-recursion}
 Y_{k+1}=(1+\beta)Y_k-\beta Y_{k-1}-s\nabla U_R(Y_k)+\zeta_{k+1},
 \qquad k\ge0.
\end{equation}

Define the initial
error $E_0=(Y_0-Rx_j^\circ,\;Y_{-1}-Rx_{j-1}^\circ)$ and the
event
\begin{equation}\label{eq:good-general}
 \mathcal G=\bigl\{\abs{E_0}\le\varepsilon_0R\bigr\}\cap
 \Bigl\{\max_{1\le i\le n}\abs{\zeta_i}\le\delta'R\Bigr\}.
\end{equation}
The sequence $(Y_k)_{k\ge-1}$ and the event $\mathcal G$ depend
on $j$; when two values of $j$ are compared, they are denoted
$(Y_k^{(j)})_{k\ge-1}$ and $\mathcal G_j$.

\begin{lemma}[Metastability of noisy heavy-ball recursions]\label{lem:metastability}
Suppose Assumptions~\ref{ass:cycle} and~\ref{ass:monodromy}
hold, and let $R\ge1$, $j\in\mathbb Z$, $n\ge1$,
$\varepsilon_0\ge0$, $\delta'>0$, the sequences
$(Y_k)_{k\ge-1}$ and $(\zeta_i)_{i\ge1}$ satisfying the
recursion \eqref{eq:abstract-recursion}, and the event
$\mathcal G$ in \eqref{eq:good-general} be as above. Assume
\begin{equation}\label{eq:smallness-general}
 C\varepsilon_0+G_0\delta'\le\frac{a}{2},
\end{equation}
where $C$ is the constant fixed after
Assumption~\ref{ass:monodromy}, and $G_0$ and $a$ are defined in
\eqref{eq:impulsegain-general} and \eqref{eq:tuberadius-general}.
Then the following conclusions hold.
\begin{enumerate}
\item[\rm(i)] On $\mathcal G$, $\abs{Y_k-Rx_{j+k}^\circ}\le aR/2$ for $0\le k\le n$.
\item[\rm(ii)] Let $(Y_k^{(j)})$ and $(Y_k^{(j+1)})$ be sequences as above with consecutive indices $j$ and $j+1$. Then
\begin{equation*}
 \norm{\Law(Y_n^{(j)})-\Law(Y_n^{(j+1)})}_{\TV}
 \ge1-\Pp(\mathcal G_j^c)-\Pp(\mathcal G_{j+1}^c).
\end{equation*}
The same lower bound holds for the laws of random elements $Z_n^{(j)}$ and $Z_n^{(j+1)}$ if there is a common measurable map $f$ such that $Y_n^{(r)}=f(Z_n^{(r)})$ for $r\in\{j,j+1\}$.
\item[\rm(iii)] Suppose $E_0=0$ and
\begin{equation*}
 \sup_{i\ge1}\Pp\bigl(\abs{\zeta_i}>\delta'R\bigr)\le Ae^{-bR^2},
\end{equation*}
where $A,b>0$ are independent of $R$, $n$, and $i$. Then $\Pp(\mathcal G^c)\le Ane^{-bR^2}$. If the sequences with consecutive indices $j$ and $j+1$ both satisfy these assumptions, then
\begin{equation*}
 \norm{\Law(Y_n^{(j)})-\Law(Y_n^{(j+1)})}_{\TV}\ge1-2Ane^{-bR^2}.
\end{equation*}
\end{enumerate}
\end{lemma}

To apply Lemma~\ref{lem:metastability}, one verifies
Assumptions~\ref{ass:cycle} and~\ref{ass:monodromy}, represents
the process under study in the form \eqref{eq:abstract-recursion},
and chooses $\varepsilon_0$ and $\delta'$ so that
\eqref{eq:smallness-general} holds; part (iii)
additionally requires $E_0=0$ and the stated tail bound on the
noise variables. The constants $C$, $q$, $G_0$, and $a$ are then
fixed as above, while $n$, $\varepsilon_0$, and $\delta'$ remain
free subject to \eqref{eq:smallness-general}.

\begin{proof}
Let $e_k=Y_k-Rx_{j+k}^\circ$ and $E_k=(e_k,e_{k-1})$, and work on
the event $\mathcal G$. We prove (i) by induction on $k$: for
every $k\in\{0,\dots,n\}$, we show that $\abs{e_i}\le aR/2$
for all $0\le i\le k$. For the base
case $k=0$, the bound $\abs{E_0}\le\varepsilon_0R$ on
$\mathcal G$, the inequality \eqref{eq:smallness-general}, and
$C\ge1$ give $\abs{e_0}\le\varepsilon_0R\le aR/2$.

Let $k<n$ and suppose that $\abs{e_i}\le aR/2$ for all
$0\le i\le k$; we show the same bound for $e_{k+1}$. Since
$U_R(x)=R^2U(x/R)$, the gradients satisfy
$\nabla U_R(Rx)=R\nabla U(x)$, so \eqref{eq:transfer-local}
dilates to
\begin{equation*}
 \nabla U_R(Rx_i^\circ+u)
 =\nabla U_R(Rx_i^\circ)+H_iu,
 \qquad \abs u\le r_0R,\quad i\in\mathbb Z.
\end{equation*}
For $0\le i\le k$ we have $\abs{e_i}\le aR/2<r_0R$, so this
identity with the index $j+i$ and $u=e_i$ gives
$\nabla U_R(Y_i)=\nabla U_R(Rx_{j+i}^\circ)+H_{j+i}e_i$.
Subtracting \eqref{eq:transfer-cycle}, multiplied by $R$, from
\eqref{eq:abstract-recursion} therefore gives the error recursion
\begin{equation*}
 E_{i+1}=\mathsf A_{j+i}E_i+\mathsf B\zeta_{i+1},
 \qquad 0\le i\le k.
\end{equation*}
Iterating this recursion from $i=0$ to $i=k$ yields
$E_{k+1}=\Psi_j(k+1,0)E_0
+\sum_{i=1}^{k+1}\Psi_j(k+1,i)\mathsf B\zeta_i$, and hence,
on $\mathcal G$,
\begin{equation*}
 \abs{E_{k+1}}
 \le\norm{\Psi_j(k+1,0)}\abs{E_0}
 +\sum_{i=1}^{k+1}\norm{\Psi_j(k+1,i)\mathsf B}\abs{\zeta_i}
 \le\bigl(C\varepsilon_0+G_0\delta'\bigr)R\le\frac{aR}2,
\end{equation*}
using \eqref{eq:smallness-general} in the last step. In
particular $\abs{e_{k+1}}\le aR/2$, which completes the induction and proves (i).

For (ii), set
\begin{equation*}
 B_n=B(Rx_{j+n}^\circ,aR/2),
 \qquad
 \widetilde B_n=B(Rx_{j+n+1}^\circ,aR/2).
\end{equation*}
The distance between their centers is at least $d_\circ R$,
whereas the sum of their radii is $aR$. Since $4a\le d_\circ$ by
\eqref{eq:tuberadius-general}, the balls $B_n$ and
$\widetilde B_n$ are disjoint. Part (i),
applied at $j$ and $j+1$, consequently gives
\begin{align*}
 \norm{\Law(Y_n^{(j)})-\Law(Y_n^{(j+1)})}_{\TV}
 &\ge\Pp\bigl(Y_n^{(j)}\in B_n\bigr)
  -\Pp\bigl(Y_n^{(j+1)}\in B_n\bigr)\\
 &\ge1-\Pp(\mathcal G_j^c)-\Pp(\mathcal G_{j+1}^c).
\end{align*}
Applying a common measurable map cannot increase total variation
distance. Since the laws of the $Y_n^{(r)}$ are the images of the
laws of the $Z_n^{(r)}$ under $f$, the lower bound for the former
implies the same lower bound for the latter.

Finally, under the assumptions of (iii), the initial-error event
holds surely, and the union bound gives
\begin{equation*}
 \Pp(\mathcal G^c)
 \le\sum_{i=1}^n\Pp\bigl(\abs{\zeta_i}>\delta'R\bigr)
 \le Ane^{-bR^2}.
\end{equation*}
This calculation does not require the noise variables
$(\zeta_i)_{i\ge1}$ to be independent. Applying (ii) at $j$ and $j+1$ proves the final assertion.
\end{proof}

\subsection{From an attracting cycle to OBABO metastability}\label{sec:transfer}

Theorem~\ref{thm:transfer} below is the main stochastic ingredient of the cycling case. It applies Lemma~\ref{lem:metastability} to the OBABO positions, and it establishes existence, uniqueness, and moment estimates for the invariant law. Together, these results convert the separation between chains started near consecutive cycle points into metastability relative to the invariant law.

\begin{theorem}[Metastability from an attracting cycle]\label{thm:transfer}
Fix $\kappa>1$ and OBABO parameters $h,\gamma>0$ with $h<2/\sqrt\kappa$, and set $\beta=e^{-\gamma h}$, $s=h^2(1+\beta)/2$, and $\tau^2=h^2(1-\beta^2)$. Suppose $U\in\cU_\kappa^d$ for $d\in\{1,2\}$ satisfies Assumptions~\ref{ass:cycle} and~\ref{ass:monodromy} with these parameters $(s,\beta)$ and that, in addition, there are $\lambda_\ast\in[1,\kappa]$ and $b_\ast>0$ with
\begin{equation}\label{eq:transfer-bounded}
 \nabla U(x)=\lambda_\ast x+w(x),
 \qquad \sup_x\abs{w(x)}\le b_\ast.
\end{equation}
Then for $U_R(x)=R^2U(x/R)$ and $Q_R=Q_{h,\gamma}^{U_R}$ there are constants $c_0,C_0>0$, independent of $R$, such that $Q_R$ has a unique invariant law $\pi_R$ and, for every $R\ge1$ and every $j\in\mathbb Z$, there exist two initial states $z_R^{(j)},z_R^{(j+1)}\in\R^{2d}$, constructed explicitly in \eqref{eq:cyclepair-general}-\eqref{eq:cyclevelocity-general} below from the cycle points with indices $j$ and $j+1$, such that at least one of them, denoted by $z_R$, satisfies
\begin{equation*}
 W_2(\delta_{z_R},\pi_R)\le C_0(1+R),
\end{equation*}
and
\begin{equation*}
 \norm{\delta_{z_R}Q_R^n-\pi_R}_{\TV}\ge\frac38,
 \qquad
 0\le n\le N_R:=\left\lfloor\frac1{16}e^{c_0R^2}\right\rfloor.
\end{equation*}
\end{theorem}

\begin{proof}
With $a$ as in \eqref{eq:tuberadius-general} and $G_0$ as in \eqref{eq:impulsegain-general}, set
\begin{equation*}
 \delta=\frac{a}{2G_0},
 \qquad
 c_0=\frac{\delta^2}{2\tau^2}>0.
\end{equation*}
For fixed $j\in\mathbb Z$, initialize
\begin{equation}\label{eq:cyclepair-general}
 (X_{-1},X_0)=(Rx_{j-1}^\circ,Rx_j^\circ)
\end{equation}
and choose
\begin{equation}\label{eq:cyclevelocity-general}
 V_0=r\left(\frac{R(x_j^\circ-x_{j-1}^\circ)}h
 -\frac h2\nabla U_R(Rx_j^\circ)\right),
\end{equation}
where $r=e^{-\gamma h/2}$ as in \eqref{eq:r_eta_sigma}; this choice corresponds to $\xi_0^{(2)}=0$ in \eqref{eq:vpairrelation}. Write $z_R^{(j)}=(Rx_j^\circ,V_0)$ for the resulting initial state of the OBABO chain. By Proposition~\ref{prop:HB}, the positions $(X_k)_{k\ge-1}$ of the chain started at $z_R^{(j)}$ satisfy the recursion \eqref{eq:abstract-recursion} with $Y_k=X_k$ and $E_0=0$, and the noise variables are $\zeta_i\sim\cN(0,\tau^2\Id_d)$ for $i\ge2$ and $\zeta_1\sim\cN(0,h^2(1-\beta)\Id_d)$, with $h^2(1-\beta)\le\tau^2$.

We bound the noise tails. Let $Z\sim\cN(0,\Id_d)$. For $d=2$, the variable $\abs Z^2$ has the chi-squared distribution with two degrees of freedom, that is, the exponential distribution with mean two, so $\Pp(\abs Z>t)=e^{-t^2/2}$ for $t\ge0$. For $d=1$, the Chernoff bound gives $\Pp(\abs Z>t)\le2e^{-t^2/2}$ for $t\ge0$. Since $\zeta_i$ has the law of $\sigma_iZ$ with $\sigma_i\le\tau$, in both cases
\begin{equation*}
 \Pp\bigl(\abs{\zeta_i}>\delta R\bigr)
 \le2e^{-\delta^2R^2/(2\tau^2)}
 =2e^{-c_0R^2},
 \qquad i\ge1.
\end{equation*}
Lemma~\ref{lem:metastability}(iii), applied with $\varepsilon_0=0$, $\delta'=\delta$, $A=2$, and $b=c_0$, therefore gives the lower bound $1-4ne^{-c_0R^2}$ for the total variation distance between the laws of the positions at time $n$ of the chains started at $z_R^{(j)}$ and $z_R^{(j+1)}$. Since the positions are a measurable function of the state, the final assertion of Lemma~\ref{lem:metastability}(ii) then gives
\begin{equation}\label{eq:pairTV-general}
 \norm{\delta_{z_R^{(j)}}Q_R^n-
       \delta_{z_R^{(j+1)}}Q_R^n}_{\TV}
 \ge1-4ne^{-c_0R^2}.
\end{equation}
At $N_R=\lfloor e^{c_0R^2}/16\rfloor$ the right-hand side is at least $3/4$.

It remains to establish existence, uniqueness, and a quadratic moment bound for the invariant law. By \eqref{eq:transfer-bounded},
\begin{equation}\label{eq:dilatedremainder-general}
 \nabla U_R(x)=\lambda_\ast x+Rw(x/R),
 \qquad \sup_x\abs{Rw(x/R)}\le b_\ast R.
\end{equation}
For the invariant-law argument, consider the homogeneous
position-pair chain $S_k=(X_k,X_{k-1})$ satisfying
\begin{equation}\label{eq:pair-general}
 S_{k+1}=\mathsf A_\ast S_k+F_R(S_k)+\mathsf B\zeta_{k+1},
 \qquad \sup_y\abs{F_R(y)}\le sb_\ast R,
\end{equation}
where $(\zeta_k)_{k\ge1}$ are i.i.d.\ $\cN(0,\tau^2\Id_d)$
variables independent of $S_0$ and, for
$y=(y_1,y_2)\in\R^d\times\R^d$,
\begin{equation*}
 \mathsf A_\ast=
 \begin{pmatrix}(1+\beta-s\lambda_\ast)\Id_d&\;-\beta\Id_d\\
                 \Id_d&\;0\end{pmatrix},
 \qquad
 F_R(y)=\bigl(-sRw(y_1/R),\,0\bigr).
\end{equation*}
The linear part is governed by $\mathsf A_\ast$, and the
remainder $F_R$ is bounded. By Proposition~\ref{prop:HB}, this is
the position-pair recursion of the OBABO chain under the Gaussian
auxiliary-variable convention $\xi_0^{(2)}\sim\cN(0,\Id_d)$.
Numerical stability gives
$\rho(\mathsf A_\ast)<1$, so by Gelfand's formula the series
$P=\sum_{i=0}^\infty(\mathsf A_\ast^\top)^i\mathsf A_\ast^i$
converges; it satisfies $P\succeq\Id_{2d}$ and solves the
discrete Lyapunov equation
\begin{equation}\label{eq:lyap-general}
 \mathsf A_\ast^\top P\mathsf A_\ast-P=-\Id_{2d}.
\end{equation}
Set
\begin{equation*}
 \alpha=\frac{1}{2\norm P}\in(0,1),
 \qquad
 K=\bigl(2\norm{\mathsf A_\ast^\top P}^2+\norm P\bigr)s^2b_\ast^2
 +\tau^2\operatorname{tr}(P).
\end{equation*}
For $\cV(y)=y^\top Py$, expansion of \eqref{eq:pair-general} gives
\begin{align}
 \E[\cV(S_{k+1})\mid S_k=y]
 &=\cV(y)-\abs y^2+2y^\top\mathsf A_\ast^\top PF_R(y)
   +F_R(y)^\top PF_R(y)+\tau^2\operatorname{tr}(\mathsf B^\top P\mathsf B)\notag\\
 &\le \cV(y)-\frac12\abs y^2+KR^2
 \le(1-\alpha)\cV(y)+KR^2.
 \label{eq:drift-general}
\end{align}
The first inequality uses the Cauchy-Schwarz inequality, Young's
inequality $2ab\le\varepsilon a^2+\varepsilon^{-1}b^2$ with
$\varepsilon=\tfrac12$ in the form
\begin{equation*}
 2y^\top\mathsf A_\ast^\top PF_R(y)
 \le\frac12\abs y^2+2\norm{\mathsf A_\ast^\top P}^2\abs{F_R(y)}^2,
\end{equation*}
the bounds $\abs{F_R(y)}\le sb_\ast R$ and
$\operatorname{tr}(\mathsf B^\top P\mathsf B)
\le\operatorname{tr}(P)$, and $R\ge1$; the second uses
$\abs y^2\ge\cV(y)/\norm P$ and $\norm P\ge1$.
The transition kernel of $(S_k)_{k\ge0}$ is Feller because
$y\mapsto\mathsf A_\ast y+F_R(y)$ is continuous. Start the
chain at $S_0=0$. Taking expectations in the drift
inequality \eqref{eq:drift-general} and iterating gives
$\E[\cV(S_k)]\le KR^2/\alpha$ for every $k\ge0$; since
$P\succeq\Id_{2d}$, the second moments $\E\abs{S_k}^2$ are
bounded uniformly in $k$, and Markov's inequality shows that the
averaged laws $\mu_N=N^{-1}\sum_{k=1}^N\Law(S_k)$ are tight. By
the Krylov-Bogoliubov theorem, every weak limit point $\bar\pi_R$
of $(\mu_N)_{N\ge1}$ is an invariant law. Since
$y\mapsto\abs y^2$ is nonnegative and continuous, the
portmanteau theorem gives
$\int\abs y^2\,\bar\pi_R(dy)
\le\liminf_{N\to\infty}\int\abs y^2\,d\mu_N
\le KR^2/\alpha$; hence, with $C=K/\alpha$,
\begin{equation}\label{eq:pairmom-general}
 \int\abs y^2\,\bar\pi_R(dy)\le C(1+R^2).
\end{equation}

In summary, the Markov chain \eqref{eq:pair-general} on $\R^{2d}$ admits an invariant law $\bar\pi_R$ satisfying the moment bound \eqref{eq:pairmom-general}; uniqueness of $\bar\pi_R$ is not needed, while uniqueness of the invariant law on phase space is established below.

To lift $\bar\pi_R$ to phase space, sample $(X_0,X_{-1})\sim\bar\pi_R$ and an independent $\xi_0^{(2)}\sim\cN(0,\Id_d)$, define $V_0$ by \eqref{eq:vpairrelation}, and let $\pi_R$ denote the law of $(X_0,V_0)$. The pair update uses the noise variable
\begin{equation*}
 \zeta_1=h\sigma(r\xi_0^{(2)}+\xi_1^{(1)})\sim\cN(0,\tau^2\Id_d),
\end{equation*}
which is independent of $(X_0,X_{-1})$, so $(X_1,X_0)\sim\bar\pi_R$. The variable $X_1$ depends on $\xi_1^{(1)}$ but not on the final-half-step noise $\xi_1^{(2)}$, so $\xi_1^{(2)}$ is independent of $(X_1,X_0)$, and \eqref{eq:vfrompair} expresses $V_1$ in terms of $(X_1,X_0)$ and $\xi_1^{(2)}$. The triples $(X_1,X_0,\xi_1^{(2)})$ and $(X_0,X_{-1},\xi_0^{(2)})$ therefore have the same law, so $(X_1,V_1)\sim\pi_R$, and $\pi_R$ is invariant.

We bound the second moment of $\pi_R$. By \eqref{eq:vpairrelation} and \eqref{eq:dilatedremainder-general},
\begin{equation*}
 V_0=\frac rh(X_0-X_{-1})-\frac{rh}2\lambda_\ast X_0
 -\frac{rh}2Rw(X_0/R)+\sigma\xi_0^{(2)}.
\end{equation*}
The Cauchy-Schwarz inequality
$(a_1+a_2+a_3+a_4)^2\le4(a_1^2+a_2^2+a_3^2+a_4^2)$, the bounds
$\abs{X_0-X_{-1}}^2\le2(\abs{X_0}^2+\abs{X_{-1}}^2)$ and
$\abs{Rw(X_0/R)}\le b_\ast R$, the identity
$\E\abs{\xi_0^{(2)}}^2=d$, and the moment bound
\eqref{eq:pairmom-general} with $C=K/\alpha$ and $R\ge1$ give
\begin{equation*}
 \E\abs{V_0}^2
 \le\Bigl(\frac{8r^2}{h^2}+r^2h^2\lambda_\ast^2\Bigr)
 \frac{2K}{\alpha}R^2+r^2h^2b_\ast^2R^2+4\sigma^2d.
\end{equation*}
Adding $\E\abs{X_0}^2\le(2K/\alpha)R^2$ and using $\sigma\le1$
and $R\ge1$,
\begin{equation}\label{eq:phasemom-general}
 \int\abs z^2\,\pi_R(dz)\le K'R^2,
 \qquad
 K':=\Bigl(1+\frac{8r^2}{h^2}+r^2h^2\lambda_\ast^2\Bigr)
 \frac{2K}{\alpha}+r^2h^2b_\ast^2+4d.
\end{equation}

For fixed $z=(x,v)$, the map
\begin{equation*}
 (\xi^{(1)},\xi^{(2)})\longmapsto(X_1,V_1)
\end{equation*}
is a triangular $C^1$ bijection: $X_1$ is affine in $\xi^{(1)}$
with coefficient $h\sigma\Id_d$, and, for fixed $\xi^{(1)}$,
$V_1$ is affine in $\xi^{(2)}$ with coefficient $\sigma\Id_d$.
Its Jacobian determinant is the positive constant
$h^d\sigma^{2d}$. The change-of-variables formula therefore gives
$Q_R$ an everywhere positive jointly continuous density
$q_R(z,z')$. If $z_k\to z$, then $q_R(z_k,z')\to q_R(z,z')$ for every $z'$ by joint continuity, and each density integrates to one, so Scheff\'e's lemma \citep[Theorem~16.12]{Billingsley1995} gives
\begin{equation*}
 \int\abs{q_R(z_k,z')-q_R(z,z')}\,dz'\longrightarrow0,
\end{equation*}
so $Q_R$ is strong Feller. Its strictly positive density also
implies irreducibility. The standard strong Feller irreducibility
criterion therefore shows that $Q_R$ has at most one invariant
probability measure, and hence the invariant law $\pi_R$
constructed above is unique. By \eqref{eq:cyclepair-general}, \eqref{eq:cyclevelocity-general}, and \eqref{eq:dilatedremainder-general}, the initial states satisfy, for $i\in\{j,j+1\}$,
\begin{equation*}
 \abs{z_R^{(i)}}\le z_{\max}R,
 \qquad
 z_{\max}:=x_{\max}\Bigl(1+\frac{2r}h+\frac{rh}2\lambda_\ast\Bigr)
 +\frac{rh}2b_\ast,
\end{equation*}
where $x_{\max}:=\max_{0\le i<m}\abs{x_i^\circ}$. The product
coupling gives
\begin{equation*}
 W_2(\delta_z,\pi_R)^2
 \le\int\abs{z-y}^2\,\pi_R(dy)
 \le2\abs z^2+2\int\abs y^2\,\pi_R(dy),
\end{equation*}
so \eqref{eq:phasemom-general} gives
\begin{equation*}
 W_2(\delta_{z_R^{(j)}},\pi_R)+W_2(\delta_{z_R^{(j+1)}},\pi_R)
 \le C_0(1+R),
 \qquad
 C_0:=2\sqrt2\,\bigl(z_{\max}^2+K'\bigr)^{1/2}.
\end{equation*}
At time $N_R$, \eqref{eq:pairTV-general} and the triangle inequality through $\pi_R$ show that one of the two starts, denoted $z_R$, has TV distance at least $3/8$. TV distance to an invariant law is nonincreasing, so the same bound holds for every $0\le n\le N_R$.
\end{proof}

\subsection{Proofs of Theorem~\ref{thm:dichotomy} and Corollary~\ref{cor:mixing}}\label{sec:proofmain}

\begin{proof}[Proof of Theorem~\ref{thm:dichotomy}]
If $\rhoq(s,\beta;\kappa)\ge q_\kappa$, Proposition~\ref{prop:quadraticbranch} gives case (Q). Otherwise Lemma~\ref{lem:interiorcycle} produces an integer $m\ge3$ with $P_m(s,\beta;\kappa)<0$, and Proposition~\ref{prop:smoothcycle} supplies a smooth roots-of-unity cycle of period $m$, with distinct points, satisfying Assumption~\ref{ass:cycle} and the bound \eqref{eq:transfer-bounded} with $d=2$, $\lambda_\ast=1$, and the constant local Hessians $H_j\equiv\Id_2$; the matrices $\mathsf A_j$ then all equal $\mathsf A_1$, and Assumption~\ref{ass:monodromy} holds because $\rho(\mathsf A_1^m)=\rho(\mathsf A_1)^m<1$. Theorem~\ref{thm:transfer} gives case (C).
\end{proof}

\begin{proof}[Proof of Corollary~\ref{cor:mixing}]
(i) Let $M$ be the phase matrix of the OBABO chain on $U_\lambda$, and let $\Sigma_n$ denote its $n$-step noise covariance, which is nonsingular for $n\ge1$ as in the proof of Proposition~\ref{prop:quadraticbranch}. By \eqref{eq:equalcovTV} and the global bound $2\Phi(t/2)-1\le t/\sqrt{2\pi}$, the total variation distance between the $n$-step laws from any two initial states $z,\widetilde z$ is at most $\lambda_{\min}(\Sigma_1)^{-1/2}\norm{M^n}\abs{z-\widetilde z}/\sqrt{2\pi}$, so Gelfand's formula $\rho(M)=\lim_{n\to\infty}\norm{M^n}^{1/n}$ gives, for any two initial states,
\begin{equation*}
 \limsup_{n\to\infty}
 \norm{\delta_z(Q_{h,\gamma}^{U_\lambda})^n-
       \delta_{\widetilde z}(Q_{h,\gamma}^{U_\lambda})^n}_{\TV}^{1/n}
 \le\rho(M).
\end{equation*}
Consequently, \eqref{eq:quadrootrateintro} implies $\rho(M)\ge q_\kappa$. As in the proof of Proposition~\ref{prop:quadraticbranch}, the chain is a Gaussian autoregressive process with nonsingular one-step noise covariance. Choose $z-\widetilde z=u\ne0$ as follows: if $M$ has a real eigenvalue of modulus $\rho(M)$, let $u$ be a corresponding real eigenvector, so that $\abs{M^nu}=\rho(M)^n\abs u$; otherwise take $u$ in the two-dimensional real invariant subspace of a complex conjugate pair of modulus $\rho(M)$, on which $M$ acts, in a suitable basis, as a rotation scaled by $\rho(M)$. In either case $\abs{M^nu}\ge c_1\rho(M)^n$ for every $n\ge0$ and some $c_1>0$. Since $\Sigma_n\preceq\Sigma_\infty$, the equal-covariance formula \eqref{eq:equalcovTV} gives, with $Q=Q_{h,\gamma}^{U_\lambda}$ and constants $c_2,c>0$ independent of $n$,
\begin{equation*}
 \norm{\delta_z Q^n-\delta_{\widetilde z}Q^n}_{\TV}
 \ \ge\ 2\Phi\bigl(c_2\,\rho(M)^n\bigr)-1
 \ \ge\ c\,\rho(M)^n,
 \qquad n\ge1.
\end{equation*}
The second inequality holds with
\begin{equation*}
 c=\min_{0\le t\le1}\frac{2\Phi(c_2t)-1}{t}:
\end{equation*}
the ratio extends continuously to the positive value $c_2\sqrt{2/\pi}$ at $t=0$, so it is continuous and positive on the compact interval $[0,1]$ and its minimum $c$ is positive, and $\rho(M)^n\in(0,1]$ for every $n\ge1$. Moreover $c\le2\Phi(c_2)-1<1$, so the bound also holds at $n=0$, because distinct point masses are at total variation distance one. By the triangle inequality, for every $n\ge0$ at least one state in $\{z,\widetilde z\}$ is at total variation distance at least $c\rho(M)^n/2$ from $\pi_\lambda$. Let $n^*=\lfloor\log(c/4\eps)/\log(1/q_\kappa)\rfloor$, so that $c\,q_\kappa^{n^*}/2\ge2\eps$. At time $n^*$, choose $z_\eps\in\{z,\widetilde z\}$ with $\norm{\delta_{z_\eps}Q^{n^*}-\pi_\lambda}_{\TV}\ge2\eps>\eps$; since total variation distance to an invariant law is nonincreasing, $t_{\mathrm{mix}}(z_\eps,\eps;Q,\pi_\lambda)>n^*$. Finally, $\log(1/q_\kappa)\le3C_\star/\kappa$ for $\kappa\ge2C_\star$, so $n^*\ge\frac{\kappa}{3C_\star}\log\frac{c}{4\eps}-1$.

(ii) By \eqref{eq:cyclemetaintro}, $\norm{\delta_{z_R}(Q_{h,\gamma}^{U_R})^n-\pi_R}_{\TV}\ge3/8>\eps$ for every $0\le n\le\lfloor\tfrac1{16}e^{c_0R^2}\rfloor$; the moment bound is \eqref{eq:cyclemomentintro}.
\end{proof}

\subsection{An explicit one-dimensional instance}\label{sec:LRP}

We now prove Corollary~\ref{cor:LRPexplicit} by constructing an
explicit one-dimensional example. Although the
example is not needed for the proof of Theorem~\ref{thm:dichotomy}, it
gives closed-form constants and a transparent illustration of how
an attracting cycle yields metastability.

At the parameters \eqref{eq:criticalparams}, Proposition~\ref{prop:HB} gives
\begin{equation*}
 s=\frac19,\qquad \beta=\frac49,
\end{equation*}
the heavy-ball parameters optimal for quadratics with curvature in $[1,25]$. Let
\begin{equation*}
 \vartheta(t)=
 \begin{cases}0,&t\le0,\\e^{-1/t},&t>0,\end{cases}
 \qquad
 \chi(t)=\frac{\vartheta(t+1)}{\vartheta(t+1)+\vartheta(1-t)}.
\end{equation*}
Here $\vartheta\in C^\infty(\R)$ vanishes on $(-\infty,0]$ and is
positive on $(0,\infty)$, and $\chi$ is the standard smooth step
function: $\chi\in C^\infty(\R)$ is nondecreasing, with $\chi=0$
on $(-\infty,-1]$ and $\chi=1$ on $[1,\infty)$.
Fix $\delta_0=1/100$ and define $U(0)=U'(0)=0$ and
\begin{equation*}
 U''(x)=25-24\chi\left(\frac{x-1}{\delta_0}\right)
              +24\chi\left(\frac{x-2}{\delta_0}\right).
\end{equation*}
Then $U\in C^\infty(\R)$, $1\le U''\le25$, and outside the intervals $(1-\delta_0,1+\delta_0)$ and $(2-\delta_0,2+\delta_0)$,
\begin{equation}\label{eq:Upieces}
 U'(x)=
 \begin{cases}
 25x,&x\le1-\delta_0,\\
 x+24,&1+\delta_0\le x\le2-\delta_0,\\
 25x-24,&x\ge2+\delta_0.
 \end{cases}
\end{equation}
This is a smooth version of the example in Lessard, Recht and Packard \citep[Section~4.6]{LessardRechtPackard2016}. Figure~\ref{fig:cycling-region} shows the tuning relative to the numerical stability and strict cycling conditions.

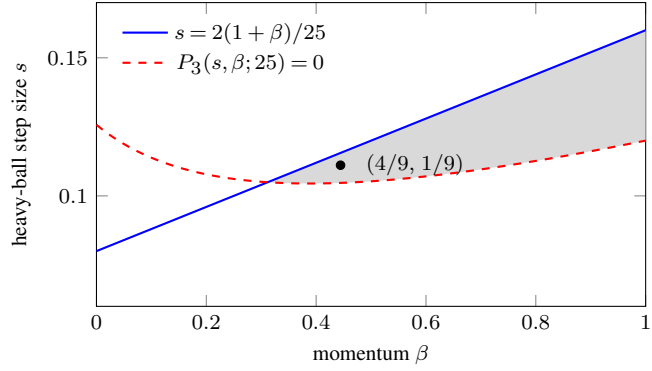
\begin{figure}[t]
\centering
\begin{tikzpicture}
\begin{axis}[
  width=0.62\textwidth, height=5.6cm,
  xlabel={momentum $\beta$}, ylabel={heavy-ball step size $s$},
  xmin=0, xmax=1, ymin=0.06, ymax=0.17,
  tick label style={font=\scriptsize},
  label style={font=\scriptsize},
  legend style={font=\scriptsize, draw=none, fill=none},
  legend pos=north west,
]
\addplot[name path=stab, thick, blue, domain=0:1, samples=100] {0.08*(1+x)};
\addlegendentry{$s=2(1+\beta)/25$}
\addplot[name path=cyc, thick, red, dashed, domain=0:1, samples=200]
  {(x+0.5+0.04*(1+x/2)) - sqrt((x+0.5+0.04*(1+x/2))^2 - 0.12*(1+x+x^2))};
\addlegendentry{$P_3(s,\beta;25)=0$}
\addplot[gray!30, forget plot] fill between[of=stab and cyc, soft clip={domain=0.3118:1}];
\addplot[only marks, mark=*, mark size=1.6pt, forget plot] coordinates {(0.444444,0.111111)};
\node[font=\scriptsize, anchor=west] at (axis cs:0.475,0.1113) {$(4/9,\,1/9)$};
\end{axis}
\end{tikzpicture}
\caption{Numerically stable tunings and the period-$3$ cycling region at
$\kappa=25$, in the plane of momentum $\beta$ and heavy-ball step size $s$.
Below the line $s=2(1+\beta)/25$ the tuning is numerically stable
(Lemma~\ref{lem:quadratic-stability}); above the curve
$P_3(s,\beta;25)=0$, within the displayed range, the strict cycling
condition $P_3<0$ holds. The marked point is the heavy-ball tuning
optimal for quadratics with curvature in $[1,25]$, used in the
instance of Section~\ref{sec:LRP}.}
\label{fig:cycling-region}
\end{figure}

Define the cycle points from Lessard, Recht and Packard
\citep[Appendix~B]{LessardRechtPackard2016} by
\begin{equation*}
 a=\frac{2592}{1225},\qquad
 b=\frac{792}{1225},\qquad
 c=-\frac{2208}{1225},
\end{equation*}
and define the period-three sequence $(p_j)_{j\in\mathbb Z}$ by
\begin{equation*}
 p_{3j}=a,\qquad p_{3j+1}=b,\qquad p_{3j+2}=c,
 \qquad j\in\mathbb Z.
\end{equation*}
These are the values denoted $r$, $p$, $q$ in \citep[equation~(B.3)]{LessardRechtPackard2016}, visited in the same cyclic order; the two definitions differ only in the value assigned to the indices divisible by three.

The next lemma establishes the two properties of $(p_k)$ used below. It shows that $(p_k)$ is an exact period-three orbit of the noise-free position recursion at the parameters $s=1/9$ and $\beta=4/9$, and it bounds the effect of additive perturbations on this orbit with explicit constants. The lemma differs from the treatment in \citep[Appendix~B]{LessardRechtPackard2016} in two respects. First, the potential there is $C^1$ with piecewise-linear gradient, while $U$ here is $C^\infty$. Second, the perturbation there enters only through the initial condition, and the conclusion is that the cycle is attracting; here an additive perturbation enters at every step, and the conclusion is the uniform bound \eqref{eq:LRPperturbation}.

\begin{lemma}[Exact cycle and stability under additive perturbations]
\label{lem:LRPcycle}
The deterministic recurrence
\begin{equation*}
 x_{k+1}=\frac{13}{9}x_k-\frac49x_{k-1}-\frac19U'(x_k)
\end{equation*}
has the period-three orbit $(p_k)$. Consider the perturbed recurrence
\begin{equation*}
 x_{k+1}
 =\frac{13}{9}x_k-\frac49x_{k-1}-\frac19U'(x_k)+\eps_{k+1},
 \qquad x_{-1}=p_{-1},\quad x_0=p_0,
\end{equation*}
and set $e_k=x_k-p_k$. Suppose that $x_k$ and $p_k$ lie in the same
interval of \eqref{eq:Upieces} for every $0\le k\le n-1$. Then, for
$0\le k\le n-1$,
\begin{equation}\label{eq:LRPerror}
 e_{k+1}=-\frac43e_k-\frac49e_{k-1}+\eps_{k+1},
\end{equation}
and
\begin{equation}\label{eq:LRPperturbation}
 \max_{0\le k\le n}\abs{e_k}
 \le9\max_{1\le j\le n}\abs{\eps_j}.
\end{equation}
In particular, if
$\max_{1\le j\le n}\abs{\eps_j}\le1/180$, then
$\abs{e_k}\le1/20$ for every $0\le k\le n$, and $x_k$ and $p_k$
lie in the same interval of \eqref{eq:Upieces} at every such step.
\end{lemma}

\begin{proof}
Since $a>2+\delta_0$ and $b,c<1-\delta_0$, equation
\eqref{eq:Upieces} gives
\begin{equation*}
 U'(a)=25a-24,\qquad U'(b)=25b,\qquad U'(c)=25c.
\end{equation*}
Direct substitution gives
\begin{align*}
 \frac{13}{9}a-\frac49c-\frac19(25a-24)&=b,\\
 \frac{13}{9}b-\frac49a-\frac19(25b)&=c,\\
 \frac{13}{9}c-\frac49b-\frac19(25c)&=a,
\end{align*}
which proves that $(p_k)$ is a period-three orbit.

When $x_k$ and $p_k$ lie in the same interval of
\eqref{eq:Upieces}, one has
$U'(x_k)-U'(p_k)=25e_k$. Subtracting the recurrence for $p_k$ from
that for $x_k$ gives \eqref{eq:LRPerror}. Since the characteristic
polynomial is
\begin{equation*}
 r^2+\frac43r+\frac49=\left(r+\frac23\right)^2,
\end{equation*}
iteration of \eqref{eq:LRPerror} yields
\begin{equation*}
 e_k=\sum_{j=1}^k
 (k-j+1)\left(-\frac23\right)^{k-j}\eps_j.
\end{equation*}
Therefore
\begin{equation*}
 \abs{e_k}
 \le\max_{1\le j\le k}\abs{\eps_j}
 \sum_{i=0}^{\infty}(i+1)\left(\frac23\right)^i
 =9\max_{1\le j\le k}\abs{\eps_j},
\end{equation*}
proving \eqref{eq:LRPperturbation}.

Finally,
\begin{equation*}
 a-\frac1{20}>2+\delta_0,\qquad
 b+\frac1{20}<1-\delta_0,\qquad
 c+\frac1{20}<1-\delta_0.
\end{equation*}
Thus the $1/20$-neighborhood of each cycle point lies in the same
interval of \eqref{eq:Upieces}. If
$\max_j\abs{\eps_j}\le1/180$, the preceding bound and induction in
$k$ give $\abs{e_k}\le1/20$ at every step, completing the proof.
\end{proof}

Lemma~\ref{lem:LRPcycle} applies as long as every perturbation is
at most $1/180$. Dilation produces this small-perturbation regime
with high probability. For $U_R(x)=R^2U(x/R)$, the cycle points
scale to $Rp_k$ and the perturbation tolerance scales to $R/180$,
while the position noise of the OBABO chain does not grow with $R$:
for either initialization below, the position-noise variances are
bounded above by
\begin{equation}\label{eq:LRPnoise}
 h_\star^2(1-\beta_\star^2)=\frac{10}{81}.
\end{equation}
Initialize the chain on the cycle with $(X_{-1},X_0)=(Rc,Ra)$ and the corresponding velocity
\begin{equation}\label{eq:zcycle}
 z_R^{\mathrm{cyc}}
 =\left(Ra,\,
 r\left(\frac{R(a-c)}{h_\star}-\frac{h_\star}{2}U_R'(Ra)\right)\right)
 =\left(Ra,\frac{360\sqrt{26}}{637}R\right),
\end{equation}
and let $z_R^0=(0,0)$.

\begin{proposition}[Explicit separation]\label{prop:LRPseparation}
For every $R\ge1$ and $n\ge0$,
\begin{equation}\label{eq:LRPpairTV}
 \norm{\delta_{z_R^{\mathrm{cyc}}}Q_R^n-
       \delta_{z_R^0}Q_R^n}_{\TV}
 \ge1-4n\exp\left(-\frac{R^2}{8000}\right).
\end{equation}
Moreover, there is a constant $C<\infty$, independent of $R$, such
that $Q_R$ has a unique invariant law $\pi_R$ satisfying
\begin{equation*}
 \int(\abs{x}^2+\abs{v}^2)\,\pi_R(dx,dv)\le CR^2,
 \qquad R\ge1.
\end{equation*}
Consequently the bound \eqref{eq:LRPexplicitbound} holds, proving Corollary~\ref{cor:LRPexplicit}.
\end{proposition}

\begin{proof}
The case $n=0$ is immediate, so assume $n\ge1$. For the chain initialized on the cycle, set $Y_k=X_k/R$. Since $U_R'(Ry)=RU'(y)$, the variables $Y_k$ satisfy the perturbed recurrence of Lemma~\ref{lem:LRPcycle} with perturbations $\eps_j=\zeta_j/R$. Hence, on the event $\max_{1\le j\le n}\abs{\zeta_j}\le R/180$, one has $\abs{X_k-Rp_k}\le R/20$ for $0\le k\le n$. The origin is a fixed point of the same deterministic recurrence: by \eqref{eq:Upieces}, $U_R'(x)=25x$ for $x\le(1-\delta_0)R$, so displacements from zero obey the error recursion \eqref{eq:LRPerror} with the same perturbation bound \eqref{eq:LRPperturbation}, and the tube $\abs x\le R/20$ lies inside this affine region; hence, on the same event, the chain initialized at the origin stays in the $R/20$ tube around zero. The tubes are disjoint because $\min\{\abs a,\abs b,\abs c\}>1/10$. By \eqref{eq:LRPnoise} and the Gaussian tail bound,
\begin{equation*}
 \Pp(\abs{\zeta_j}>R/180)\le2e^{-R^2/8000}.
\end{equation*}
The two laws therefore concentrate, on the respective events, on disjoint tubes, and as in Lemma~\ref{lem:metastability}(ii) a union bound for each chain gives \eqref{eq:LRPpairTV}.

For the invariant-law assertion, Theorem~\ref{thm:transfer} applies with $H_j=25$ and $\lambda_\ast=25$: $h_\star<2/5$, $U'(x)=25x+w(x)$ with $\abs{w(x)}\le24$, and the local error matrix has repeated eigenvalue $-2/3$. The theorem gives uniqueness, while \eqref{eq:phasemom-general} gives $\int(\abs x^2+\abs v^2)\,d\pi_R\le CR^2$, with $C$ independent of $R$.

At $N_R=\lfloor16^{-1}e^{R^2/8000}\rfloor$, the right-hand side of \eqref{eq:LRPpairTV} is at least $3/4$. The positive-time window is nonempty once
\begin{equation*}
 R\ge\sqrt{8000\log16}\approx148.9.
\end{equation*}
For smaller $R$, the statement remains valid but the displayed time window may contain only $n=0$.

The triangle inequality through $\pi_R$ makes one of the two starts at least $3/8$ from stationarity. Monotonicity of TV distance to an invariant law extends the bound to every earlier time, while the moment bound and \eqref{eq:zcycle} give $W_2(\delta_{z_R},\pi_R)\le CR$.
\end{proof}

\section{Proof of the mixing time upper bound and sharpness of the diffusive scale}\label{sec:upper}

This section shows that the diffusive scale is optimal for fixed-parameter OBABO tunings. Theorem~\ref{thm:upper} provides a cold-start upper bound of the form \eqref{eq:falseuniform} with $e^{-cn/\kappa}$ in place of $e^{-cn/\tau(\kappa)}$, for an explicit numerically stable tuning; its proof combines the Wasserstein contraction estimates of Leimkuhler, Paulin and Whalley \citep{LeimkuhlerPaulinWhalley2024} with the regularization estimate of Bou-Rabee, Cox and Schieven \citep{BouRabeeCoxSchieven2026}. Chak and Monmarch\'e \citep[Theorem~2.2]{ChakMonmarche2026} prove a related Wasserstein-to-total-variation regularization estimate for generalized Hamiltonian Monte Carlo, allowing nonconvex potentials and stochastic gradients. In the exact-gradient setting, their result requires the Hessian of the potential to be Lipschitz. By contrast, the estimate of \citep{BouRabeeCoxSchieven2026} applies directly to OBABO for every step size and every number of steps under the assumptions of \eqref{eq:classUkappa}, with a fully explicit, dimension-free constant, and is therefore the one used here. The section concludes with the proof of Corollary~\ref{cor:sharp}.

\begin{theorem}[Diffusive upper bound]\label{thm:upper}
There exists a numerically stable fixed-parameter OBABO tuning rule such that, for every $\kappa>1$, every $d\ge1$, and every $U\in\cU_\kappa^d$, the kernel $Q=Q_{h_\kappa,\gamma_\kappa}^U$ has a unique invariant law $\pi$, and for every $z\in\R^d\times\R^d$ and every $n\ge0$,
\begin{equation*}
 \norm{\delta_zQ^n-\pi}_{\TV}
 \le C\kappa\bigl(1+W_2(\delta_z,\pi)\bigr)
 \exp\left(-\frac{cn}{\kappa}\right),
 \qquad
 C=143,
 \quad
 c=\frac{1}{128(1-e^{-1})}.
\end{equation*}
\end{theorem}

\begin{proof}
A fixed-parameter tuning rule assigns parameters to every curvature triple $(K,L,d)$, and the numerical stability condition \eqref{eq:stabletuningintro} quantifies over all such triples, so we first specify the rule in these variables and then evaluate it on the normalized class. Consider the fixed-parameter rule
\begin{equation*}
 h_{K,L,d}=\frac1{4\sqrt L},
 \qquad
 \gamma_{K,L,d}=4\sqrt L,
\end{equation*}
which depends on $L$ alone and satisfies \eqref{eq:stabletuningintro}. On the class $\cU_\kappa^d$, where $K=1$ and $L=\kappa$, it gives $h_\kappa=\frac1{4\sqrt\kappa}$ and $\gamma_\kappa=4\sqrt\kappa$, so that $\gamma_\kappa h_\kappa=1$ and $\beta=e^{-\gamma_\kappa h_\kappa}=e^{-1}$. It satisfies the step size condition $h<(1-e^{-\gamma h})/(2\sqrt\kappa)$ of \citep[Theorem~5.2]{LeimkuhlerPaulinWhalley2024}, since $\tfrac14<(1-e^{-1})/2$. Set
\begin{equation*}
 a=\kappa^{-1},\qquad
 b=\frac{h_\kappa}{1-e^{-1}},\qquad
 c(h_\kappa)=\frac{h_\kappa^2}{4(1-e^{-1})}=\frac{1}{64(1-e^{-1})\kappa},
\end{equation*}
and define the weighted norm on phase space by
\begin{equation*}
 \norm{(x,v)}_{a,b}^2=\abs x^2+2b\langle x,v\rangle+a\abs v^2;
\end{equation*}
a direct check gives $b^2<a/4$, so this is indeed a norm. That theorem states that two OBABO chains driven by the same sequence of Gaussian variables and started at $z,\widetilde z$ satisfy, for every $n\ge1$, almost surely,
\begin{equation*}
 \norm{Z_n-\widetilde Z_n}_{a,b}
 \le7\bigl(1-c(h_\kappa)\bigr)^{(n-1)/2}\norm{z-\widetilde z}_{a,b}.
\end{equation*} Squaring this estimate and applying \citep[Proposition~2.4]{LeimkuhlerPaulinWhalley2024}, which converts such a contraction into a Wasserstein estimate and shows that the kernel has a unique invariant law $\pi$, with finite second moment, gives
\begin{align*}
 W_2^2(\delta_zQ^n,\pi)
 &\le3\cdot\frac{49}{1-c(h_\kappa)}\max\{a,a^{-1}\}
 \bigl(1-c(h_\kappa)\bigr)^{n}W_2^2(\delta_z,\pi)\\
 &\le147\kappa\bigl(1-c(h_\kappa)\bigr)^{n-1}W_2^2(\delta_z,\pi),
\end{align*}
since $\max\{a,a^{-1}\}=\kappa$ for $\kappa\ge1$. Hence, for every $n\ge1$,
\begin{equation*}
 W_1(\delta_zQ^n,\pi)\le W_2(\delta_zQ^n,\pi)
 \le\sqrt{147\kappa}\,\bigl(1-c(h_\kappa)\bigr)^{(n-1)/2}W_2(\delta_z,\pi).
\end{equation*}
By \citep[Theorem~3.2 and Corollary~3.3]{BouRabeeCoxSchieven2026}, applied with $r$ steps for an integer $r\ge1$,
\begin{equation}\label{eq:TVreg}
 \norm{\nu Q^r-\widetilde\nu Q^r}_{\TV}\le\mathcal R_r\,W_1(\nu,\widetilde\nu)
 \qquad\text{for all probability measures }\nu,\widetilde\nu,
\end{equation}
where
\begin{equation*}
 \mathcal R_r=\gamma_\kappa^{-1/2}
 \left[\frac{5}{(rh_\kappa)^{3/2}}
 +\frac{12+5\gamma_\kappa}{(rh_\kappa)^{1/2}}
 +\frac{\kappa(\gamma_\kappa h_\kappa)^{1/2}}{(1-e^{-\gamma_\kappa h_\kappa})^{1/2}}
 \Bigl(1+\frac{rh_\kappa}{1+r\gamma_\kappa h_\kappa}\Bigr)(rh_\kappa)^{1/2}\right].
\end{equation*}
The estimate \eqref{eq:TVreg} requires no restriction on the step size, and $\mathcal R_r$ does not depend on the dimension. The number $r$ of regularization steps is free. To obtain a convenient constant uniformly over $\kappa>1$, we take $r=24$. At the tuning above, $\gamma_\kappa h_\kappa=1$ and $24h_\kappa=6/\sqrt\kappa$, so, with $c_\ast=\bigl(2\sqrt{2(1-e^{-1})}\bigr)^{-1}$,
\begin{equation*}
 \mathcal R_{24}
 =\Bigl(\frac{5}{12\sqrt6}+\frac{10}{\sqrt6}+2\sqrt3\,c_\ast\Bigr)\sqrt\kappa
 +\sqrt6+\frac{12\sqrt3}{25}\,c_\ast
 \ \le\ 8.62\sqrt\kappa,
 \qquad\kappa\ge1.
\end{equation*}
For $n\ge25$, write $Q^n=Q^{n-24}Q^{24}$: the first $n-24$ steps contract in Wasserstein distance, and the last $24$ steps convert Wasserstein distance into total variation through \eqref{eq:TVreg}. Since $\pi Q^{24}=\pi$, applying \eqref{eq:TVreg} with $\nu=\delta_zQ^{n-24}$ and $\widetilde\nu=\pi$ gives
\begin{equation*}
 \norm{\delta_zQ^n-\pi}_{\TV}
 \le\sqrt{147\kappa}\,\mathcal R_{24}\bigl(1-c(h_\kappa)\bigr)^{(n-25)/2}W_2(\delta_z,\pi)
 \le C\kappa\bigl(1+W_2(\delta_z,\pi)\bigr)e^{-cn/\kappa},
\end{equation*}
using $(1-t)^{(n-25)/2}\le e^{25t/2}e^{-tn/2}$ for $t=c(h_\kappa)$, the bound $e^{25t/2}\le\exp\bigl(\tfrac{25}{128(1-e^{-1})}\bigr)$, valid since $\kappa>1$ gives $t\le\tfrac{1}{64(1-e^{-1})}$, together with $tn/2=cn/\kappa$ and $\sqrt{147}\cdot8.62\cdot\exp\bigl(\tfrac{25}{128(1-e^{-1})}\bigr)\le143=C$. For $n\le24$ the left side is at most one, while $C\kappa e^{-cn/\kappa}\ge143\,e^{-24c}>1$, so the bound holds for every $n\ge0$.
\end{proof}

\begin{proof}[Proof of Corollary~\ref{cor:sharp}]
The first assertion follows by solving the bound in
Theorem~\ref{thm:upper}: its right-hand side is at most $\eps$ once
$n\ge(\kappa/c)\log(C\kappa(1+W_2(\delta_z,\pi))/\eps)$, so
\begin{equation*}
 t_{\mathrm{mix}}(z,\eps;Q,\pi)
 \le\left\lceil\frac{\kappa}{c}
 \log\frac{C\kappa(1+W_2(\delta_z,\pi))}{\eps}\right\rceil.
\end{equation*}
Set $C_1=\max\{C,\,1/c+1,\,e\}=143$; here $1/c+1=128(1-e^{-1})+1<83$. Write $L=\log(C_1\kappa(1+W_2(\delta_z,\pi))/\eps)$; since $C_1\ge e$, $\kappa>1$, and $\eps\in(0,1)$, we have $L\ge1$. Therefore
\begin{equation*}
 \left\lceil\frac{\kappa}{c}
 \log\frac{C\kappa(1+W_2(\delta_z,\pi))}{\eps}\right\rceil
 \le\frac{\kappa}{c}L+1
 \le\Bigl(\frac1c+1\Bigr)\kappa L
 \le C_1\kappa L,
\end{equation*}
using $C\le C_1$ in the first inequality and $\kappa L\ge1$ in the second, which is the display in part \emph{(i)}. For the second, assume for contradiction that, for every
$\kappa\ge\kappa_0$, every target in $\cU_\kappa^2$ whose kernel has
an invariant law $\pi$, every initial state $z$, and every $n$,
\begin{equation*}
 \norm{\delta_z(Q_{h_\kappa,\gamma_\kappa}^U)^n-\pi}_{\TV}
 \le C\kappa^a\bigl(1+W_2(\delta_z,\pi)\bigr)^b
 \exp\left(-\frac{cn}{\tau(\kappa)}\right);
\end{equation*}
that is, \eqref{eq:falseuniform} never holds. Fix
$C_\star>C_{\mathrm{GTD}}$ and note $1-q_\kappa\le2C_\star/\kappa$.
Since $\tau(\kappa)=o(\kappa)$, we have $c\kappa/\tau(\kappa)\ge5C_\star$
for all sufficiently large $\kappa$. Fix such a $\kappa$ with
$\kappa\ge\kappa_0$ and $\kappa>C_\star^2$. If
$c/\tau(\kappa)\ge1$, then
$e^{-c/\tau(\kappa)}\le e^{-1}<q_\kappa$, since
$C_\star/\kappa<1/C_\star<1/27$; if $c/\tau(\kappa)<1$, then
$1-e^{-c/\tau(\kappa)}\ge\frac{c}{2\tau(\kappa)}\ge\frac{5C_\star}{2\kappa}
>1-q_\kappa$. In either case,
\begin{equation}\label{eq:qbeatsballistic}
 q_\kappa>e^{-c/\tau(\kappa)}.
\end{equation}
Apply Theorem~\ref{thm:dichotomy} to the tuning
$(h_\kappa,\gamma_\kappa)$. In case (Q), the assumed bound applied to
the two initial states $z$ and $\widetilde z$ of
\eqref{eq:quadrootrateintro}, followed by the triangle inequality
through the invariant law, gives
\begin{equation*}
 \norm{\delta_z(Q_{h_\kappa,\gamma_\kappa}^{U_\lambda})^n
       -\delta_{\widetilde z}(Q_{h_\kappa,\gamma_\kappa}^{U_\lambda})^n}_{\TV}
 \le C_{z,\widetilde z,\kappa}\,e^{-cn/\tau(\kappa)}
\end{equation*}
with a constant independent of $n$. Taking $n$th roots and letting
$n\to\infty$ contradicts \eqref{eq:quadrootrateintro} and
\eqref{eq:qbeatsballistic}. In case (C), apply the assumed bound at
the state $z_R$ and the time $N_R=\lfloor e^{c_0R^2}/16\rfloor$.
Equations \eqref{eq:cyclemomentintro} and \eqref{eq:cyclemetaintro}
give
\begin{equation*}
 \frac38
 \le C\kappa^a\bigl(1+C_0(1+R)\bigr)^b
      \exp\left(-\frac{cN_R}{\tau(\kappa)}\right).
\end{equation*}
Since $\kappa$ is fixed and $0<\tau(\kappa)<\infty$, the right-hand
side tends to zero as $R\to\infty$, a contradiction. The final
assertion is the case $\tau(\kappa)=\kappa^\alpha$ with $\alpha<1$.
\end{proof}

\begin{remark}[Other Strang splittings]\label{rem:upper-other}
The contraction estimates of \citep{LeimkuhlerPaulinWhalley2024} cover further splitting schemes, including BAOAB, so the first stage of the argument is not specific to OBABO. The regularization estimate of \citep{BouRabeeCoxSchieven2026} is stated for OBABO, and we therefore state Theorem~\ref{thm:upper} for OBABO only; the analogous upper bounds for the other splittings of Section~\ref{sec:BAOAB} require the corresponding regularization estimates, which we do not pursue here.
\end{remark}

\section{Extension to standard Strang splittings}\label{sec:BAOAB}

This section extends the mixing time lower bounds from OBABO to the remaining standard Strang splittings of \eqref{eq:KLD}. For $t>0$, the three splitting components on $\R^d\times\R^d$ are
\begin{align*}
 \mathrm A_t(x,v)&=(x+tv,\,v),\qquad
 \mathrm B_t(x,v)=(x,\,v-t\nabla U(x)),\\
 \mathrm O_t(x,v)&=(x,\,e^{-\gamma t}v+(1-e^{-2\gamma t})^{1/2}\xi),
\end{align*}
where $\xi\sim\cN(0,\Id_d)$ is drawn independently at each application of $\mathrm O_t$. The six standard Strang splittings with step size $h$ compose the three letters in a palindrome, the middle letter as a full step and the outer letters as half steps: OBABO and BOAOB, with middle $\mathrm A_h$; OABAO and AOBOA, with middle $\mathrm B_h$; and BAOAB and ABOBA, with middle $\mathrm O_h$. Eliminating the velocity yields, for every splitting, an exact noisy heavy-ball recursion with momentum $\beta=e^{-\gamma h}$: the heavy-ball step size is $s=h^2(1+\beta)/2$ for OBABO, BAOAB, ABOBA, and OABAO, and $s=\sqrt\beta\,h^2$ for AOBOA and BOAOB, and the Gaussian noise terms are independent or $1$-dependent with explicit covariances (Propositions~\ref{prop:HB-BAOAB} and~\ref{prop:HB-palindromic}). The dichotomy and the resulting cold-start lower bounds then carry over to every splitting (Theorems~\ref{thm:dichotomy-BAOAB} and~\ref{thm:palindromic}, Corollaries~\ref{cor:main-BAOAB} and~\ref{cor:main-palindromic}). The lower bounds are therefore not specific to OBABO.

We treat BAOAB separately and first because it is a widely used splitting \citep{LeimkuhlerMatthews2013}. Its deterministic position recursion is the same heavy-ball recursion as for OBABO, but its eliminated noise is $1$-dependent rather than independent. Consequently, the position-pair process is no longer Markov, and the invariant-law argument must be carried out directly on phase space. The remaining four splittings add further complications: transformed recursion variables for ABOBA and AOBOA, and, for OABAO, recursion variables that are not functions of the current state together with a modified initialization. They are treated together in Section~\ref{sec:palindromic}.

The invariant-law arguments for the splittings below share a common structure: one step of the chain on phase space is a Schur-stable linear map, plus Gaussian noise and a remainder that is bounded uniformly in the state and in the within-step noise. The next lemma isolates the corresponding drift estimate.

\begin{lemma}[Invariant law for perturbed linear recursions]\label{lem:drift}
Let $\mathsf M\in\R^{d_1\times d_1}$ satisfy $\rho(\mathsf M)<1$, let $\mathsf G\in\R^{d_1\times d_2}$, and let $F:\R^{d_1}\times\R^{d_2}\to\R^{d_1}$ be continuous with $B:=\sup_{z,\xi}\abs{F(z,\xi)}<\infty$. Let $(\xi_k)_{k\ge1}$ be independent $\cN(0,\Id_{d_2})$ variables, and consider the Markov chain on $\R^{d_1}$ defined by
\begin{equation}\label{eq:perturbed-linear}
 Z_{k+1}=\mathsf MZ_k+\mathsf G\xi_{k+1}+F(Z_k,\xi_{k+1}),
 \qquad k\ge0.
\end{equation}
Then the chain admits an invariant probability law $\pi$ satisfying
\begin{equation}\label{eq:drift-moment}
 \int\abs z^2\,\pi(dz)
 \le8\norm P^3\bigl(\operatorname{tr}(\mathsf G\mathsf G^\top)+B^2\bigr),
 \qquad
 P:=\sum_{j=0}^\infty(\mathsf M^\top)^j\mathsf M^j.
\end{equation}
\end{lemma}

\begin{proof}
Since $\rho(\mathsf M)<1$, the series defining $P$ converges; $P$ is symmetric, satisfies $P\succeq\Id_{d_1}$, and solves the discrete Lyapunov equation $\mathsf M^\top P\mathsf M=P-\Id_{d_1}$. Set $V(z)=z^\top Pz$ and $W_{k+1}=\mathsf G\xi_{k+1}+F(Z_k,\xi_{k+1})$, so that $Z_{k+1}=\mathsf MZ_k+W_{k+1}$ and $\E[\abs{W_{k+1}}^2\mid Z_k]\le2\operatorname{tr}(\mathsf G\mathsf G^\top)+2B^2$. For every $\theta>0$, Young's inequality $2a^\top Pw\le\theta\,a^\top Pa+\theta^{-1}w^\top Pw$ and the discrete Lyapunov equation for $P$ give
\begin{equation*}
 \E\bigl[V(Z_{k+1})\mid Z_k=z\bigr]
 \le(1+\theta)\bigl(V(z)-\abs z^2\bigr)
 +2(1+\theta^{-1})\norm P\bigl(\operatorname{tr}(\mathsf G\mathsf G^\top)+B^2\bigr).
\end{equation*}
The difference $V(z)-\abs z^2$ is nonnegative because $P\succeq\Id_{d_1}$, and it is at most $(1-1/\norm P)V(z)$ because $V(z)\le\norm P\abs z^2$. The choice $\theta=1/(2\norm P-1)$ gives $(1+\theta)(1-1/\norm P)\le1-\alpha$ with $\alpha=1/(2\norm P)$, so
\begin{equation*}
 \E\bigl[V(Z_{k+1})\mid Z_k=z\bigr]
 \le(1-\alpha)V(z)+K_0,
 \qquad
 K_0:=4\norm P^2\bigl(\operatorname{tr}(\mathsf G\mathsf G^\top)+B^2\bigr),
\end{equation*}
and iterating from $Z_0=0$ gives $\E V(Z_k)\le K_0/\alpha$ for every $k\ge0$. Since $F$ is continuous and bounded, dominated convergence shows that $z\mapsto\E f(\mathsf Mz+\mathsf G\xi_1+F(z,\xi_1))$ is continuous for every bounded continuous $f$, so the chain is Feller. Since $V(z)\ge\abs z^2$, the Ces\`aro averages $\pi_n=n^{-1}\sum_{k=0}^{n-1}\mathrm{Law}(Z_k)$ satisfy $\int V\,d\pi_n\le K_0/\alpha$ and are tight; a weak limit point $\pi$ of a subsequence exists by Prokhorov's theorem, and the Krylov-Bogoliubov theorem shows that $\pi$ is an invariant probability law. For every $\ell>0$ the function $V\wedge\ell$ is bounded and continuous, so $\int V\wedge\ell\,d\pi\le K_0/\alpha$; letting $\ell\to\infty$ by monotone convergence gives $\int V\,d\pi\le K_0/\alpha=8\norm P^3(\operatorname{tr}(\mathsf G\mathsf G^\top)+B^2)$. The bound \eqref{eq:drift-moment} follows since $\abs z^2\le V(z)$.
\end{proof}

Lemma~\ref{lem:drift} asserts existence only. Uniqueness fails in this generality: for $\mathsf M=0$, $\mathsf G=0$, and $F(z,\xi)=f(z)$ with $f$ bounded, continuous, and equal to the identity on the unit ball, every point of the unit ball is a fixed point of the chain, and each point mass there is invariant. In the applications below, uniqueness is established separately: the one-step kernels for BOAOB and OABAO, and the two-step kernels for BAOAB, ABOBA, and AOBOA, have everywhere positive continuous densities.

\subsection{BAOAB}\label{sec:baoab-sub}

Fix $h,\gamma>0$ and write
\begin{equation*}
 \beta=e^{-\gamma h},\qquad \varsigma=(1-\beta^2)^{1/2}.
\end{equation*}
One step, driven by a single $\xi\sim\cN(0,\Id_d)$, maps the state $(x,v)$ to the updated state $(x^+,v^+)$ through the intermediate quantities $v^{(a)},x^{(b)},v^{(c)}$:
\begin{align}
 v^{(a)}&=v-\frac h2\nabla U(x),\label{eq:baoab1}\\
 x^{(b)}&=x+\frac h2 v^{(a)},\nonumber\\
 v^{(c)}&=\beta v^{(a)}+\varsigma\xi,\nonumber\\
 x^+&=x^{(b)}+\frac h2 v^{(c)},\label{eq:baoab4}\\
 v^+&=v^{(c)}-\frac h2\nabla U(x^+).\label{eq:baoab5}
\end{align}
The first and last lines are the two $\mathrm B$ half steps, the
second and fourth lines the two $\mathrm A$ half steps, and the
middle line the full $\mathrm O$ step. Iterating this update with i.i.d.\ copies of $\xi$ defines a Markov chain on $\R^d\times\R^d$, and we denote its transition kernel by $\widetilde Q_{h,\gamma}^U$; see \citep{LeimkuhlerMatthews2013}.

\begin{proposition}[Noisy heavy-ball representation of BAOAB]\label{prop:HB-BAOAB}
Let $(X_k,V_k)_{k\ge0}$ be the BAOAB chain, with $\xi_k$ the Gaussian variable used in the $k$th step. Let $\xi_0\sim\cN(0,\Id_d)$ be independent of $(\xi_k)_{k\ge1}$, or alternatively let $\xi_0$ be the zero vector, and define the auxiliary variable $X_{-1}$ by
\begin{equation}\label{eq:vpair-baoab}
 V_0=\frac{2\beta}{(1+\beta)h}(X_0-X_{-1})
 +\frac{\varsigma}{1+\beta}\,\xi_0-\frac h2\nabla U(X_0).
\end{equation}
Then, for $k\ge0$,
\begin{equation}\label{eq:noisyHB-baoab}
 X_{k+1}=X_k+\beta(X_k-X_{k-1})-s\nabla U(X_k)+\widetilde\zeta_{k+1},
 \qquad
 s=\frac{h^2}{2}(1+\beta),
\end{equation}
where
\begin{equation}\label{eq:zeta-baoab}
 \widetilde\zeta_{k+1}=\frac h2\varsigma\,(\xi_k+\xi_{k+1}).
\end{equation}
If $\xi_0\sim\cN(0,\Id_d)$, the sequence $(\widetilde\zeta_k)_{k\ge1}$ is a centered stationary $1$-dependent Gaussian sequence with
\begin{equation}\label{eq:zetacov-baoab}
 \E[\widetilde\zeta_k\widetilde\zeta_k^\top]=\frac{h^2}{2}(1-\beta^2)\Id_d,
 \qquad
 \E[\widetilde\zeta_k\widetilde\zeta_{k+1}^\top]=\frac{h^2}{4}(1-\beta^2)\Id_d.
\end{equation}
If \eqref{eq:vpair-baoab} is imposed with $\xi_0=0$, then $\widetilde\zeta_1=\frac h2\varsigma\,\xi_1$ has the smaller covariance $\frac{h^2}{4}(1-\beta^2)\Id_d$, while \eqref{eq:zetacov-baoab} holds from the second update onward.
\end{proposition}

\begin{proof}
From \eqref{eq:baoab1}-\eqref{eq:baoab4}, one step gives
\begin{equation}\label{eq:xstep-baoab}
 X_{k+1}=X_k+\frac h2(1+\beta)\Bigl(V_k-\frac h2\nabla U(X_k)\Bigr)
 +\frac h2\varsigma\,\xi_{k+1}.
\end{equation}
Within step $k$, $X_k-X_{k-1}=\frac h2(v^{(a)}_k+v^{(c)}_k)$ and $v^{(c)}_k=\beta v^{(a)}_k+\varsigma\xi_k$; eliminating $v^{(a)}_k$ and inserting $V_k=v^{(c)}_k-\frac h2\nabla U(X_k)$ from \eqref{eq:baoab5} yields
\begin{equation}\label{eq:vfrompair-baoab}
 V_k=\frac{2\beta}{(1+\beta)h}(X_k-X_{k-1})
 +\frac{\varsigma}{1+\beta}\,\xi_k-\frac h2\nabla U(X_k),
\end{equation}
which for $k=0$ is the imposed relation \eqref{eq:vpair-baoab}. Substituting \eqref{eq:vfrompair-baoab} into \eqref{eq:xstep-baoab} and collecting the gradient terms, whose coefficient is $\frac{h^2}{4}(1+\beta)+\frac{h^2}{4}(1+\beta)=s$, gives \eqref{eq:noisyHB-baoab}-\eqref{eq:zeta-baoab}. The covariances \eqref{eq:zetacov-baoab} are immediate, and $\widetilde\zeta_j$, $\widetilde\zeta_k$ are independent for $\abs{j-k}\ge2$.
\end{proof}

Thus BAOAB has the same deterministic heavy-ball recursion as OBABO at the same $(h,\gamma)$, while the independent noise is replaced by a $1$-dependent sequence with half the covariance. Moreover, the phase-space matrix $\widetilde{\mathsf M}_\lambda$ of one step on the quadratic potential $U_\lambda(x)=\lambda\abs x^2/2$ has trace
\begin{equation*}
 (1+\beta)(1-h^2\lambda/2)=1+\beta-s\lambda
\end{equation*}
and determinant $\beta$: the $\mathrm B$ and $\mathrm A$ substeps are unit-determinant shears, and in each coordinate the $\mathrm O$ substep multiplies the velocity by $\beta$. Hence $\widetilde{\mathsf M}_\lambda$ has the characteristic polynomial \eqref{eq:charpoly}, and Lemma~\ref{lem:quadratic-stability} holds without change.

\begin{theorem}[Non-acceleration dichotomy for BAOAB]\label{thm:dichotomy-BAOAB}
Theorem~\ref{thm:dichotomy} holds after replacing $Q_{h,\gamma}^U$ by the BAOAB kernel $\widetilde Q_{h,\gamma}^U$ and the time window in \eqref{eq:cyclemetaintro} by $0\le n\le\lfloor\tfrac1{32}e^{c_0R^2}\rfloor$, with the same parameters \eqref{eq:hbparamsintro}.
\end{theorem}

\begin{proof}
We indicate the modifications to the proofs of Section~\ref{sec:proofs}; the deterministic results of Section~\ref{sec:GTD} are unchanged, since they depend only on $(s,\beta)$.

\emph{Quadratic case.} On $U_\lambda$ in dimension two, the chain is a Gaussian autoregressive process whose phase matrix consists of two copies of $\widetilde{\mathsf M}_\lambda$, so its spectral radius again equals $\rhoq(s,\beta;\kappa)$ at the maximizing mode. The one-step noise is $d$-dimensional and hence degenerate on the $2d$-dimensional phase space, but the two-step covariance is nonsingular: given the current state, $X_1$ is an affine function of $\xi_1$ with derivative $\frac h2\varsigma\Id_d$; given $\xi_1$, the next position $X_2$ is affine in $\xi_2$ with the same derivative; and $(X_1,X_2)\mapsto(X_2,V_2)$ is an affine bijection, since its linear part is invertible:
\begin{equation}\label{eq:V2X1}
 \frac{\partial V_2}{\partial X_1}
 =h\nabla^2U(X_1)-\frac4h\,\Id_d
\end{equation}
is negative definite whenever $h^2\nabla^2U\prec4\,\Id_d$, which numerical stability guarantees. Lemma~\ref{lem:gaussianTVrate} applies with $n_0=2$ and therefore yields two initial states whose asymptotic pairwise total variation contraction factor is $\rhoq(s,\beta;\kappa)$, proving case \emph{(Q)}.

\emph{Tube estimate.} Lemma~\ref{lem:metastability} uses the noise variables only through the bound in the event \eqref{eq:good-general}, so their $1$-dependence is immaterial. Set $\delta=a/(2G_0)$, with $a$ from \eqref{eq:tuberadius-general} and $G_0$ from \eqref{eq:impulsegain-general}. Here $\abs{\widetilde\zeta_i}\le\frac h2\varsigma(\abs{\xi_{i-1}}+\abs{\xi_i})$, so, writing $\mathcal G$ for the event \eqref{eq:good-general} with $\varepsilon_0=0$ and $\delta'=\delta$,
\begin{equation*}
 \Bigl\{\max_{0\le i\le n}\abs{\xi_i}\le\frac{\delta R}{h\varsigma}\Bigr\}
 \subseteq\mathcal G,
 \qquad
 \Pp(\mathcal G^c)\le2(n+1)e^{-c_0R^2},
 \qquad
 c_0=\frac{\delta^2}{2h^2(1-\beta^2)},
\end{equation*}
by the same radial Gaussian tail bound applied to each $\xi_i$. Lemma~\ref{lem:metastability}(ii) then gives the pair separation $1-4(n+1)e^{-c_0R^2}\ge1-8ne^{-c_0R^2}$ for $n\ge1$, that is, \eqref{eq:pairTV-general} with $8n$ in place of $4n$, and the window $N_R=\lfloor\tfrac1{32}e^{c_0R^2}\rfloor$ gives the separation $3/4$ as before.

\emph{Invariant law.} The position-pair chain is no longer Markov, so we run the drift argument directly on the chain $Z_k=(X_k,V_k)$ on phase space. With $\nabla U_R(x)=\lambda_\ast x+Rw(x/R)$ as in \eqref{eq:dilatedremainder-general}, one BAOAB step is
\begin{equation*}
 Z_{k+1}=\widetilde{\mathsf M}_{\lambda_\ast}Z_k+\widetilde{\mathsf G}\xi_{k+1}+\widetilde F_R(Z_k,\xi_{k+1}),
 \qquad
 \sup_{z,\xi}\abs{\widetilde F_R(z,\xi)}\le CR,
\end{equation*}
where $\widetilde{\mathsf G}$ is the constant input matrix and $\widetilde F_R$ collects the bounded gradient remainders; the final kick evaluates $\nabla U_R$ at $X_{k+1}$, which depends on $\xi_{k+1}$, so the remainder depends on the within-step noise as well as on the state, as \eqref{eq:perturbed-linear} allows. The map $\widetilde F_R$ is continuous, and $\rho(\widetilde{\mathsf M}_{\lambda_\ast})<1$ by numerical stability, as noted after Proposition~\ref{prop:HB-BAOAB}. Lemma~\ref{lem:drift} with $B=CR$ and the moment bound \eqref{eq:drift-moment} therefore yield an invariant law $\pi_R$ with $\int\abs z^2\,\pi_R(dz)\le C(1+R^2)$. For uniqueness, the two-step kernel has an everywhere positive, jointly continuous density: $(\xi_1,\xi_2)\mapsto(X_1,X_2)$ is a triangular $C^1$ bijection with constant Jacobian determinant $(\tfrac h2\varsigma)^{2d}$, and $(X_1,X_2)\mapsto(X_2,V_2)$ is a $C^1$ bijection: at fixed $X_2$, the map $X_1\mapsto V_2$ is, by \eqref{eq:V2X1} applied with $U_R$, up to an additive constant, the negative gradient of
\begin{equation*}
 \Psi(x)=\frac2h\abs x^2-hU_R(x).
\end{equation*}
Under \eqref{eq:stabletuningintro},
\begin{equation*}
 \nabla^2\Psi(x)=\frac4h\Id_d-h\nabla^2U_R(x)
 \succeq\Bigl(\frac4h-h\kappa\Bigr)\Id_d\succ0.
\end{equation*}
Hence $\Psi$ is strongly convex, so $\nabla\Psi$, and therefore $X_1\mapsto V_2$, is a global $C^1$ diffeomorphism. Scheff\'e's lemma \citep[Theorem~16.12]{Billingsley1995} then gives the strong Feller property for $(\widetilde Q^{U_R}_{h,\gamma})^2$, positivity gives Lebesgue irreducibility, and the invariant law is unique. The moment bound, the cycle initialization \eqref{eq:vpair-baoab} with $\xi_0=0$, and the triangle inequality complete the proof exactly as in Theorem~\ref{thm:transfer}.
\end{proof}

\begin{corollary}[No fixed-parameter acceleration for BAOAB]\label{cor:main-BAOAB}
Corollary~\ref{cor:sharp}(ii) holds after replacing OBABO by BAOAB and
$Q_{h_\kappa,\gamma_\kappa}^U$ by
$\widetilde Q_{h_\kappa,\gamma_\kappa}^U$ throughout: no
numerically stable fixed-parameter BAOAB tuning admits a cold-start
exponential bound, uniform over the class \eqref{eq:classUkappa},
with convergence rate of larger order than $\kappa^{-1}$, under the same polynomial-prefactor condition.
\end{corollary}

\begin{proof}
Repeat the proof of the lower-bound assertion of
Corollary~\ref{cor:sharp}, applying
Theorem~\ref{thm:dichotomy-BAOAB} in place of
Theorem~\ref{thm:dichotomy}. The smaller time window
$\lfloor\tfrac1{32}e^{c_0R^2}\rfloor$ affects only a fixed constant
and does not change the argument.
\end{proof}

\subsection{The remaining Strang splittings}\label{sec:palindromic}

We now treat ABOBA, AOBOA, BOAOB, and OABAO. Throughout, $\beta=e^{-\gamma h}$, $r=e^{-\gamma h/2}$, $\sigma=(1-\beta)^{1/2}$, and $\varsigma=(1-\beta^2)^{1/2}$, and $(X_k,V_k)_{k\ge0}$ denotes the chain of the indicated splitting, with $\xi_{k+1}$, respectively $(\xi_{k+1}^{(1)},\xi_{k+1}^{(2)})$ in order of application, the standard Gaussian variables used in the $(k{+}1)$st step.

\needspace{8\baselineskip}%
\begin{proposition}[Noisy heavy-ball representations of the remaining Strang splittings]\label{prop:HB-palindromic}
The following identities hold.
\begin{enumerate}
\item[\rm(i)] \emph{ABOBA.} The force-evaluation points $Y_{m+1}=X_m+\frac h2V_m$, together with $Y_0:=X_0-\frac h2V_0$, satisfy $Y_{m+1}-Y_m=hV_m$ for $m\ge0$ and, for $m\ge1$,
\begin{equation}\label{eq:HB-aboba}
 Y_{m+1}=Y_m+\beta(Y_m-Y_{m-1})-s\nabla U(Y_m)+h\varsigma\,\xi_m,
 \qquad s=\frac{h^2}{2}(1+\beta);
\end{equation}
the noise variables $h\varsigma\xi_m$ are independent with covariance $h^2(1-\beta^2)\Id_d$.
\item[\rm(ii)] \emph{AOBOA.} The force-evaluation points $Y_{m+1}=X_m+\frac h2V_m$, together with $Y_0:=X_0-\frac h2V_0$, satisfy $Y_{m+1}-Y_m=hV_m$ for $m\ge0$ and, for $m\ge1$,
\begin{equation}\label{eq:HB-aoboa}
 Y_{m+1}=Y_m+\beta(Y_m-Y_{m-1})-s\nabla U(Y_m)+h\sigma\bigl(r\xi_m^{(1)}+\xi_m^{(2)}\bigr),
 \qquad s=\sqrt\beta\,h^2;
\end{equation}
the noise variables are independent with covariance $h^2(1-\beta^2)\Id_d$.
\item[\rm(iii)] \emph{BOAOB.} Let $\xi_0^{(2)}\sim\cN(0,\Id_d)$ be independent of the variables driving the chain, or alternatively let $\xi_0^{(2)}$ be the zero vector, and define the auxiliary variable $X_{-1}$ by
\begin{equation}\label{eq:vpair-boaob}
 V_0=\frac rh(X_0-X_{-1})+\sigma\xi_0^{(2)}-\frac h2\nabla U(X_0).
\end{equation}
Then, for $k\ge0$,
\begin{equation}\label{eq:HB-boaob}
 X_{k+1}=X_k+\beta(X_k-X_{k-1})-s\nabla U(X_k)+h\sigma\bigl(r\xi_k^{(2)}+\xi_{k+1}^{(1)}\bigr),
 \qquad s=\sqrt\beta\,h^2;
\end{equation}
the noise variables are independent. If $\xi_0^{(2)}\sim\cN(0,\Id_d)$, every noise variable has covariance $h^2(1-\beta^2)\Id_d$; if $\xi_0^{(2)}=0$, the first has the smaller covariance $h^2(1-\beta)\Id_d$, and each subsequent one has covariance $h^2(1-\beta^2)\Id_d$.
\item[\rm(iv)] \emph{OABAO.} Set $\chi_m=r\xi_{m-1}^{(2)}+\xi_m^{(1)}$ for $m\ge2$. The force-evaluation points $Y_m=X_{m-1}+\frac h2\bigl(rV_{m-1}+\sigma\xi_m^{(1)}\bigr)$ satisfy, for $m\ge2$,
\begin{equation}\label{eq:HB-oabao}
 Y_{m+1}=Y_m+\beta(Y_m-Y_{m-1})-s\nabla U(Y_m)+\frac h2\sigma\bigl(\chi_m+\chi_{m+1}\bigr),
 \qquad s=\frac{h^2}{2}(1+\beta).
\end{equation}
The variables $\sigma\chi_m$, $m\ge2$, are independent $\cN(0,(1-\beta^2)\Id_d)$, so the noise is a centered $1$-dependent Gaussian sequence with the covariances \eqref{eq:zetacov-baoab}; in law it coincides with the BAOAB noise \eqref{eq:zeta-baoab}.
\end{enumerate}
\end{proposition}

\begin{proof}
(i) One ABOBA step from $(X_m,V_m)$ reads $Y_{m+1}=X_m+\frac h2V_m$, then
\begin{equation*}
 V_{m+1}=\beta V_m-\frac h2(1+\beta)\nabla U(Y_{m+1})+\varsigma\xi_{m+1},
 \qquad
 X_{m+1}=Y_{m+1}+\frac h2V_{m+1},
\end{equation*}
because both half-kicks act at the same point $Y_{m+1}$, the position being unchanged by $\mathrm O_h$. Hence $Y_{m+2}=X_{m+1}+\frac h2V_{m+1}=Y_{m+1}+hV_{m+1}$, so $V_m=(Y_{m+1}-Y_m)/h$ for $m\ge0$ with the stated $Y_0$, and eliminating $V$ gives \eqref{eq:HB-aboba}.

(ii) One AOBOA step reads $Y_{m+1}=X_m+\frac h2V_m$, then
\begin{equation*}
 V_{m+1}=\beta V_m-rh\nabla U(Y_{m+1})+\sigma\bigl(r\xi_{m+1}^{(1)}+\xi_{m+1}^{(2)}\bigr),
 \qquad
 X_{m+1}=Y_{m+1}+\frac h2V_{m+1},
\end{equation*}
by composing $\mathrm O_{h/2}\,\mathrm B_h\,\mathrm O_{h/2}$ on the velocity. The same elimination as in (i) gives \eqref{eq:HB-aoboa}, and the covariance is $h^2\sigma^2(r^2+1)\Id_d=h^2(1-\beta^2)\Id_d$.

(iii) Write $v^{(b)}_{k+1}=r\bigl(V_k-\frac h2\nabla U(X_k)\bigr)+\sigma\xi_{k+1}^{(1)}$ for the velocity after the first two substeps, so that $X_{k+1}=X_k+hv^{(b)}_{k+1}$ and $V_{k+1}=rv^{(b)}_{k+1}+\sigma\xi_{k+1}^{(2)}-\frac h2\nabla U(X_{k+1})$. Using
\begin{equation*}
 V_k=rv^{(b)}_k+\sigma\xi_k^{(2)}-\frac h2\nabla U(X_k),
 \qquad
 v^{(b)}_k=\frac{X_k-X_{k-1}}h,
\end{equation*}
in the definition of $v^{(b)}_{k+1}$ gives
\begin{equation*}
 v^{(b)}_{k+1}=\frac{\beta}{h}(X_k-X_{k-1})-rh\nabla U(X_k)+\sigma\bigl(r\xi_k^{(2)}+\xi_{k+1}^{(1)}\bigr),
\end{equation*}
valid for $k\ge1$, and for $k=0$ by the imposed relation; multiplying by $h$ gives \eqref{eq:HB-boaob}.

(iv) Write $v^{(a)}_{m}=rV_{m-1}+\sigma\xi_m^{(1)}$ and $v^{(c)}_m=v^{(a)}_m-h\nabla U(Y_m)$ for the velocities before and after the middle kick of the $m$th step, so that $X_m=Y_m+\frac h2v^{(c)}_m$ and $V_m=rv^{(c)}_m+\sigma\xi_m^{(2)}$. Then $v^{(a)}_{m+1}=\beta v^{(c)}_m+\sigma\chi_{m+1}$, hence
\begin{equation*}
 Y_{m+1}-Y_m=\frac h2\bigl(v^{(c)}_m+v^{(a)}_{m+1}\bigr)
 =\frac h2(1+\beta)v^{(c)}_m+\frac h2\sigma\chi_{m+1},
\end{equation*}
and combining this with $v^{(c)}_{m+1}=\beta v^{(c)}_m+\sigma\chi_{m+1}-h\nabla U(Y_{m+1})$ yields \eqref{eq:HB-oabao}. The $\sigma\chi_m$ are independent with covariance $\sigma^2(r^2+1)\Id_d=(1-\beta^2)\Id_d$ because distinct $\chi_m$ involve disjoint pairs of the driving variables, and the covariances of the noise terms $\frac h2\sigma(\chi_m+\chi_{m+1})$ agree with \eqref{eq:zetacov-baoab}.
\end{proof}

For ABOBA and AOBOA, the map $(x,v)\mapsto\bigl(x+\frac h2v,\;x-\frac h2v\bigr)$ is a linear bijection of phase space under which the chain becomes the Markov chain of consecutive pairs $(Y_{m+1},Y_m)$. Total variation distances are invariant under this bijection, and Wasserstein distances change by at most a factor depending only on $h$.

\needspace{8\baselineskip}%
\begin{theorem}[Non-acceleration dichotomy for the remaining Strang splittings]\label{thm:palindromic}
Let $\dagger$ denote one of the four splittings of Proposition~\ref{prop:HB-palindromic}, with heavy-ball step size $s_\dagger$ as given there. Fix $C_\star>C_{\mathrm{GTD}}$, let $\kappa\ge2C_\star$, and assume $0<s_\dagger<2(1+\beta)/\kappa$. Then Theorem~\ref{thm:dichotomy} holds after replacing the OBABO kernel by the kernel of the splitting $\dagger$, the definition of $s$ in \eqref{eq:hbparamsintro} by $s_\dagger$, and the stability condition $h<2/\sqrt\kappa$ by $0<s_\dagger<2(1+\beta)/\kappa$, and, for OABAO, with the time window in \eqref{eq:cyclemetaintro} replaced by $0\le n\le\lfloor\tfrac1{32}e^{c_0R^2}\rfloor$.
\end{theorem}

\begin{proof}
We first show that every splitting has the same quadratic characteristic polynomial, with $s$ replaced by $s_\dagger$. This proves case \emph{(Q)} when $\rhoq(s_\dagger,\beta;\kappa)\ge q_\kappa$. In the complementary case, the deterministic results of Section~\ref{sec:GTD} provide the attracting cycle used below. We then apply Lemma~\ref{lem:metastability} by identifying the recursion variables, controlling their noise terms, and constructing the required initial states.

On a quadratic potential the one-step phase matrix of each splitting has determinant $\beta$ in every coordinate, since $\mathrm A$ and $\mathrm B$ substeps are unit-determinant shears and the $\mathrm O$ substeps contract each velocity coordinate by a total factor $\beta$, and its trace on the curvature-$\lambda$ mode is $1+\beta-s_\dagger\lambda$. For ABOBA and AOBOA this follows from Proposition~\ref{prop:HB-palindromic}: in the pair coordinates given by the linear bijection above, the one-step linear map is the companion matrix of the pair recursion, and the trace and determinant are unchanged by the change of coordinates. For BOAOB and OABAO we verify the trace by composing the five substeps directly on $U_\lambda$: in the phase-space coordinates $(x,v)$, the one-step matrices are
\begin{equation*}
 \mathsf M^{\mathrm{BOAOB}}_\lambda=
 \begin{pmatrix}
  1-\frac{s_\dagger\lambda}{2} & rh\\[2pt]
  -\frac{h\lambda}{2}\bigl(1+\beta-\frac{s_\dagger\lambda}{2}\bigr) & \beta-\frac{s_\dagger\lambda}{2}
 \end{pmatrix},
 \qquad
 \mathsf M^{\mathrm{OABAO}}_\lambda=
 \begin{pmatrix}
  1-\frac{h^2\lambda}{2} & rh\bigl(1-\frac{h^2\lambda}{4}\bigr)\\[2pt]
  -rh\lambda & \beta\bigl(1-\frac{h^2\lambda}{2}\bigr)
 \end{pmatrix},
\end{equation*}
with $s_\dagger=\sqrt\beta\,h^2$ for BOAOB and $s_\dagger=h^2(1+\beta)/2$ for OABAO, and direct expansion gives trace $1+\beta-s_\dagger\lambda$ and determinant $\beta$ in both cases. Each mode therefore has the characteristic polynomial \eqref{eq:charpoly} with $s=s_\dagger$, so Lemma~\ref{lem:quadratic-stability} applies with $(s_\dagger,\beta)$. The deterministic results of Section~\ref{sec:GTD} depend only on $(s,\beta)$ and apply with $(s_\dagger,\beta)$ as well.

\emph{Quadratic case.} Assume first that
\begin{equation*}
 \rhoq(s_\dagger,\beta;\kappa)\ge q_\kappa.
\end{equation*}
Because $\lambda\mapsto\rho(A_\lambda(s_\dagger,\beta))$ is continuous on the compact interval $[1,\kappa]$, choose $\lambda_{\mathrm Q}\in[1,\kappa]$ such that
\begin{equation*}
 \rho(A_{\lambda_{\mathrm Q}}(s_\dagger,\beta))=\rhoq(s_\dagger,\beta;\kappa).
\end{equation*}
On $\R^2$, let
\begin{equation*}
 U_{\lambda_{\mathrm Q}}(x)=\frac{\lambda_{\mathrm Q}}2\abs x^2.
\end{equation*}
We now verify the nondegeneracy hypothesis of Lemma~\ref{lem:gaussianTVrate} for each splitting. For ABOBA and AOBOA, in the pair coordinates above the chain on $U_{\lambda_{\mathrm Q}}$ is a Gaussian autoregressive process whose noise enters the leading pair component with the nonsingular covariance $h^2(1-\beta^2)\Id_d$, so the two-step covariance is nonsingular and Lemma~\ref{lem:gaussianTVrate} applies with $n_0=2$. For BOAOB and OABAO the one-step noise covariance on phase space is nonsingular: given the state, the map $(\xi^{(1)},\xi^{(2)})\mapsto(x^+,v^+)$ is block triangular with diagonal blocks $h\sigma\Id_d$ and $\sigma\Id_d$ for BOAOB, and $h\sigma(1-\tfrac{h^2\lambda_{\mathrm Q}}4)\Id_d$ and $\sigma\Id_d$ for OABAO, the first factor being positive by numerical stability. Lemma~\ref{lem:gaussianTVrate} therefore supplies two initial states $z,\widetilde z\in\R^4$ whose asymptotic pairwise total variation contraction factor is $\rho(A_{\lambda_{\mathrm Q}}(s_\dagger,\beta))=\rhoq(s_\dagger,\beta;\kappa)$, proving case \emph{(Q)} for the four splittings.

\emph{Cycling case.} Assume henceforth that
\begin{equation*}
 \rhoq(s_\dagger,\beta;\kappa)<q_\kappa.
\end{equation*}
As in the proof of Theorem~\ref{thm:dichotomy}, Lemma~\ref{lem:interiorcycle} produces an integer $m\ge3$ with $P_m(s_\dagger,\beta;\kappa)<0$, and Proposition~\ref{prop:smoothcycle}, applied with $(s_\dagger,\beta)$, supplies the potential $U\in\cU_\kappa^2$, the cycle $(x_j^\circ)$, Assumptions~\ref{ass:cycle} and~\ref{ass:monodromy}, and the bound \eqref{eq:transfer-bounded} with $d=2$, $\lambda_\ast=1$, and $H_j\equiv\Id_2$. For $R\ge1$, set $U_R(x)=R^2U(x/R)$ and write $Q_R^\dagger$ for the phase-space transition kernel of the splitting $\dagger$ applied to $U_R$. The remaining steps establish case \emph{(C)} for this cycling-case data.

The recursion supplied to Lemma~\ref{lem:metastability} is the position sequence $\widehat Y_k=X_k$ for BOAOB, as it is for BAOAB in Theorem~\ref{thm:dichotomy-BAOAB}, and the shifted force-evaluation sequences $\widehat Y_k=Y_{k+1}$ for ABOBA and AOBOA and $\widehat Y_k=Y_{k+2}$ for OABAO, with the noise and cycle indices shifted accordingly.

\emph{Tube estimate.} For ABOBA and AOBOA the noise variables are independent with covariance $h^2(1-\beta^2)\Id_2$. The same holds for BOAOB from the second update onward, while, under the convention $\xi_0^{(2)}=0$, its first noise variable has the smaller covariance $h^2(1-\beta)\Id_2$. Set
\begin{equation*}
 \bar\delta=\frac{a}{2G_0},
 \qquad
 c_0=\frac{\bar\delta^2}{2h^2(1-\beta^2)},
\end{equation*}
with $a$ from \eqref{eq:tuberadius-general} and $G_0$ from \eqref{eq:impulsegain-general}. For the cycle initializations specified below, $E_0=0$. Moreover, for each of these three splittings,
\begin{equation*}
 \sup_{i\ge1}\Pp\bigl(\abs{\zeta_i^\dagger}>\bar\delta R\bigr)\le2e^{-c_0R^2},
\end{equation*}
where $\zeta_i^\dagger$ denotes the noise in the corresponding heavy-ball recursion.

For OABAO, set $\chi_1=\xi_1^{(1)}$, corresponding to the convention $\xi_0^{(2)}=0$. The variables $\sigma\chi_\ell$, $\ell\ge1$, are independent, with covariance $(1-\beta^2)\Id_2$ for $\ell\ge2$ and the smaller covariance $(1-\beta)\Id_2$ for $\ell=1$. Under the shift $\widehat Y_k=Y_{k+2}$, the increments used for $n-1$ updates are $\widehat\zeta_i=\zeta_{i+2}$ for $1\le i\le n-1$, where, by Proposition~\ref{prop:HB-palindromic}(iv),
\begin{equation*}
 \zeta_{m+1}=\frac h2\sigma(\chi_m+\chi_{m+1}).
\end{equation*}
For every $\delta_\dagger>0$, set $c_0=\delta_\dagger^2/(2h^2(1-\beta^2))$. The radial Gaussian tail bound and a union bound give
\begin{equation*}
 \Pp\Bigl(\max_{1\le\ell\le n+1}h\sigma\abs{\chi_\ell}>\delta_\dagger R\Bigr)
 \le2(n+1)e^{-c_0R^2}.
\end{equation*}
On the complementary event,
\begin{equation*}
 \max_{1\le i\le n-1}\abs{\widehat\zeta_i}
 =\max_{3\le k\le n+1}\abs{\zeta_k}
 \le\delta_\dagger R.
\end{equation*}
The same event also controls the shifted initial error, as shown below. We use these two bounds to estimate the probability of the good event \eqref{eq:good-general} directly.

\emph{Separating events and initialization.} For BOAOB the recursion variables are the positions. For $j\in\mathbb Z$, initialize at
\begin{equation*}
 z_R^{(j)}
 =\Bigl(Rx^\circ_j,\;
 \frac rhR\bigl(x^\circ_j-x^\circ_{j-1}\bigr)-\frac h2\nabla U_R(Rx^\circ_j)\Bigr),
\end{equation*}
which is \eqref{eq:vpair-boaob} with $\xi_0^{(2)}=0$ and $(X_{-1},X_0)=(Rx^\circ_{j-1},Rx^\circ_j)$, so $E_0=0$; the invariant-law and moment estimates are treated below. For ABOBA and AOBOA the tube events concern the pairs $(Y_{m+1},Y_m)$, which are state events under the linear bijection above, and total variation is unchanged; initialize at
\begin{equation*}
 z_R^{(j)}
 =\Bigl(\frac R2\bigl(x^\circ_j+x^\circ_{j-1}\bigr),\;
 \frac Rh\bigl(x^\circ_j-x^\circ_{j-1}\bigr)\Bigr),
\end{equation*}
for which $(Y_1,Y_0)=(Rx^\circ_j,Rx^\circ_{j-1})$, so $E_0=0$. Since $G_0\bar\delta=a/2$, Lemma~\ref{lem:metastability}(iii), applied with $\delta'=\bar\delta$ to the initial states with consecutive cycle indices $j$ and $j+1$, gives pairwise total variation separation at least $1-4ne^{-c_0R^2}$ for the recursion variables. For ABOBA and AOBOA this transfers to phase space under the linear bijection, and for BOAOB it transfers because the position is a measurable function of the state. Consequently,
\begin{equation}\label{eq:pairTV-three-splittings}
 \norm{\delta_{z_R^{(j)}}(Q_R^\dagger)^n-\delta_{z_R^{(j+1)}}(Q_R^\dagger)^n}_{\TV}
 \ge1-4ne^{-c_0R^2},
 \qquad n\ge1.
\end{equation}
The two initial states are distinct, so their total variation distance is one at time zero. In particular, the distance in \eqref{eq:pairTV-three-splittings} is at least $3/4$ for
\begin{equation*}
 0\le n\le\left\lfloor\frac1{16}e^{c_0R^2}\right\rfloor.
\end{equation*}
This completes the confinement and pairwise-separation argument for BOAOB, ABOBA, and AOBOA.

\emph{OABAO: initialization, confinement, and separation.} The OABAO cycling argument parallels the BAOAB argument after velocity elimination. The shifted force-evaluation variables satisfy the same noisy heavy-ball recursion, with noise having the same law as in BAOAB. Two additional points must be addressed: the shifted recursion has a random initial error, and its variables are not functions of the current phase-space state. We therefore first control the shifted initial error and then transfer confinement of the recursion variables to separation of the positions.

Write $y^\circ_t=x^\circ_t$, $t\in\mathbb Z$, for the cycle of Lemma~\ref{lem:metastability}, relabeled because it is now traversed by the recursion variables rather than by the positions. Define the deterministic cycle velocities
\begin{equation*}
 v^{(c)\circ}_t=\frac{2}{h(1+\beta)}\bigl(y^\circ_{t+1}-y^\circ_t\bigr),
 \qquad
 v^{(a)\circ}_t=v^{(c)\circ}_t+h\nabla U(y^\circ_t),
\end{equation*}
so that $v^{(a)\circ}_{t+1}=\beta v^{(c)\circ}_t$, as the noise-free recursion requires, and initialize at the state $z_R^{(j)}$ with $rV_0=Rv^{(a)\circ}_j$ and $X_0=Ry_j^\circ-\frac h2Rv^{(a)\circ}_j$.

Provided that $\delta_\dagger\le\bar\delta$, on the event
\begin{equation*}
 \max_{1\le\ell\le n+1}h\sigma\abs{\chi_\ell}\le\delta_\dagger R
\end{equation*}
we have
\begin{equation*}
 Y_1=Ry_j^\circ+\frac h2\sigma\chi_1,
 \qquad
 \abs{Y_1-Ry_j^\circ}\le\frac{\delta_\dagger R}2.
\end{equation*}
The local gradient identity \eqref{eq:transfer-local} then gives
\begin{equation*}
 Y_2-Ry_{j+1}^\circ
 =\frac{h\sigma}{2}\Bigl(\bigl[1+(1+\beta)(1-h^2/2)\bigr]\chi_1+\chi_2\Bigr).
\end{equation*}
Thus the initial error for the shifted recursion, started at cycle index $j+1$, is
\begin{equation*}
 E_0=\bigl(Y_2-Ry_{j+1}^\circ,\;Y_1-Ry_j^\circ\bigr).
\end{equation*}
Consequently,
\begin{equation*}
 \abs{E_0}\le C_2\delta_\dagger R
\end{equation*}
for a constant $C_2$ depending only on $(h,\gamma,\kappa)$, enlarged to bound both components of the shifted initial error.

Assume first that $n\ge2$, and, after choosing $a_\dagger$ and $\delta_\dagger$ as below, let $\mathcal G$ be the good event \eqref{eq:good-general} for the shifted recursion $\widehat Y_k=Y_{k+2}$, applied for $n-1$ updates with
\begin{equation*}
 \varepsilon_0=C_2\delta_\dagger,
 \qquad
 \delta'=\delta_\dagger.
\end{equation*}
The preceding bounds show that
\begin{equation*}
 \Bigl\{\max_{1\le\ell\le n+1}h\sigma\abs{\chi_\ell}\le\delta_\dagger R\Bigr\}
 \subseteq\mathcal G,
\end{equation*}
and it follows that $\Pp(\mathcal G^c)\le2(n+1)e^{-c_0R^2}$. On $\mathcal G$, the estimate established in the proof of Lemma~\ref{lem:metastability}(i) gives
\begin{equation*}
 \abs{Y_m-Ry^\circ_{j+m-1}}\le\bigl(CC_2+G_0\bigr)\delta_\dagger R\le a_\dagger R/2,
 \qquad 1\le m\le n+1.
\end{equation*}

From the proof of Proposition~\ref{prop:HB-palindromic}(iv),
\begin{equation*}
 X_m=\frac{\beta Y_m+Y_{m+1}}{1+\beta}
 -\frac{h\sigma}{2(1+\beta)}\,\chi_{m+1},
\end{equation*}
so the positions are recovered from the recursion variables with an error controlled by a constant $C_1$ depending only on $\beta$.

Next we identify the points that the positions track. Set
\begin{equation*}
 \bar x_t=y_t^\circ+\frac{1}{1+\beta}\bigl(y_{t+1}^\circ-y_t^\circ\bigr),
 \qquad t\in\mathbb Z.
\end{equation*}
For the roots-of-unity cycle, $\bar x_t=R_m^t\bar x_0$ with $\bar x_0=\frac{\beta\Id+R_m}{1+\beta}y_0^\circ\ne0$, since $-\beta$ is not an eigenvalue of the rotation $R_m$; the $\bar x_t$ are therefore $m$ distinct points.

Now choose the parameters. Fix $a_\dagger\in(0,a]$ small enough that the balls $B(R\bar x_t,2C_1a_\dagger R)$, $0\le t<m$, are pairwise disjoint, which is possible because the $\bar x_t$ are distinct and their separation scales with $R$. Then fix $\delta_\dagger\in(0,\bar\delta]$ small enough that $(CC_2+G_0)\delta_\dagger\le a_\dagger/2$ and $\delta_\dagger\le a_\dagger$, where $C$ and $G_0$ are the constants defined before Lemma~\ref{lem:metastability}; both choices depend only on $(h,\gamma,\kappa,m)$. With $c_0=\delta_\dagger^2/(2h^2(1-\beta^2))$ as above, the preceding tail bound applies.

On $\mathcal G$, the confinement above controls $Y_1,\dots,Y_{n+1}$, and hence $X_1,\dots,X_n$:
\begin{equation*}
 \abs{X_m-R\bar x_{j+m-1}}\le C_1(a_\dagger+\delta_\dagger)R\le2C_1a_\dagger R,
 \qquad 1\le m\le n,
\end{equation*}
so the two chains initialized at consecutive cycle indices lie, at the same time, in the disjoint position balls fixed above. For $n=1$, the same bound follows directly from the estimates on $Y_1$ and $Y_2$ and the recovery formula above. The pair separation follows with these position events in place of the tubes: the sum of the two failure probabilities is at most $4(n+1)e^{-c_0R^2}\le8ne^{-c_0R^2}$ for $n\ge1$. Consequently, for OABAO there are two deterministic initial states $z_R^{(j)}$ and $z_R^{(j+1)}$ corresponding to consecutive indices of the cycle such that
\begin{equation}\label{eq:pairTV-oabao}
 \norm{\delta_{z_R^{(j)}}(Q_R^\dagger)^n-\delta_{z_R^{(j+1)}}(Q_R^\dagger)^n}_{\TV}
 \ge1-8ne^{-c_0R^2},
 \qquad n\ge1.
\end{equation}
The distance is one at time zero. Hence the distance in \eqref{eq:pairTV-oabao} is at least $3/4$ for
\begin{equation*}
 0\le n\le\left\lfloor\frac1{32}e^{c_0R^2}\right\rfloor.
\end{equation*}
This completes the confinement and pairwise-separation argument for OABAO.

\emph{Invariant law and comparison with stationarity for all four splittings.} For BOAOB the one-step map $(\xi^{(1)},\xi^{(2)})\mapsto(x^+,v^+)$ is triangular, with $x^+$ affine in $\xi^{(1)}$ with coefficient $h\sigma\Id_d$ and, at fixed $\xi^{(1)}$, with $v^+$ affine in $\xi^{(2)}$ with coefficient $\sigma\Id_d$; the Jacobian determinant is therefore the constant $(h\sigma)^d\sigma^d$, and the kernel has an everywhere positive jointly continuous density. For OABAO, writing $y$ for the force-evaluation point, $x^+=2y-x-\frac{h^2}2\nabla U_R(y)$; the map $y\mapsto2y-x-\frac{h^2}2\nabla U_R(y)$ is the gradient of $y\mapsto\abs y^2-\langle x,y\rangle-\frac{h^2}2U_R(y)$, whose Hessian $2\bigl(\Id_d-\frac{h^2}4\nabla^2U_R(y)\bigr)$ is uniformly positive definite under the stability condition; gradients of strongly convex $C^2$ functions are global $C^1$ diffeomorphisms, and since $\xi^{(1)}\mapsto y$ is affine with coefficient $\frac h2\sigma\Id_d$ and $\xi^{(2)}$ enters $v^+$ affinely with coefficient $\sigma\Id_d$, the kernel again has an everywhere positive jointly continuous density. For ABOBA and AOBOA the same holds for the two-step kernel in the pair coordinates: at fixed initial pair, the map from the two consecutive Gaussian noise terms to the updated pair is triangular, since the first noise term enters the intermediate recursion variable additively and the second enters the final one additively at fixed intermediate value; the diagonal blocks are identities, the noise covariances $h^2(1-\beta^2)\Id_d$ are nonsingular, and the two-step kernel therefore has an everywhere positive jointly continuous density. Scheff\'e's lemma \citep[Theorem~16.12]{Billingsley1995} gives the strong Feller property, and positivity gives Lebesgue irreducibility. Hence the relevant one- or two-step kernel has at most one invariant probability law; for ABOBA and AOBOA an invariant law of the one-step kernel is invariant for its square, so uniqueness transfers. For each splitting, with $\nabla U_R(x)=\lambda_\ast x+Rw(x/R)$ as in \eqref{eq:dilatedremainder-general}, one step of the chain $Z_k=(X_k,V_k)$ on phase space has the form \eqref{eq:perturbed-linear}:
\begin{equation}\label{eq:splitting-decomposition}
 Z_{k+1}=\mathsf M_\dagger Z_k+\mathsf G_\dagger\xi_{k+1}
 +F_R(Z_k,\xi_{k+1}),
 \qquad
 \sup_{z,\xi}\abs{F_R(z,\xi)}\le C_FR,
\end{equation}
where $\mathsf M_\dagger$ is the one-step phase matrix of the splitting on the quadratic potential $\lambda_\ast\abs x^2/2$, $\mathsf G_\dagger$ is the constant input matrix, $\xi_{k+1}$ collects the step's Gaussian inputs, and $F_R$ collects the terms $Rw(\cdot/R)$ from the force evaluations, propagated through the remaining substeps. Each force-evaluation point is an affine function of $(Z_k,\xi_{k+1})$ plus terms already bounded by a constant multiple of $R$, and the substep maps are affine up to these remainders, with coefficients depending only on $(h,\gamma,\lambda_\ast)$, so $F_R$ is continuous and the bound in \eqref{eq:splitting-decomposition} holds with $C_F$ depending only on $(h,\gamma,\lambda_\ast,b_\ast)$. Moreover $\rho(\mathsf M_\dagger)<1$: each quadratic mode has the characteristic polynomial \eqref{eq:charpoly} with $s=s_\dagger$ and $\lambda=\lambda_\ast$, as verified above, and $0<s_\dagger\lambda_\ast<2(1+\beta)$ under the standing stability condition. Lemma~\ref{lem:drift} with $B=C_FR$ and the moment bound \eqref{eq:drift-moment} therefore yield an invariant law with $\int\abs z^2\,\pi_R(dz)\le C(1+R^2)$. Each of the two initial states has Wasserstein distance at most $C_0(1+R)$ from $\pi_R$, as before, the linear bijection for ABOBA and AOBOA changing $W_2$ by a factor depending only on $h$, absorbed into $C_0$.

At the endpoint of the applicable time window, the laws of the two initial states are at total variation distance at least $3/4$. The triangle inequality therefore shows that one of these states, denoted by $z_R$, is at total variation distance at least $3/8$ from $\pi_R$. Since this distance is nonincreasing in time, the same bound holds throughout the applicable window, proving case \emph{(C)}.
\end{proof}

\needspace{7\baselineskip}%
\begin{corollary}[No fixed-parameter acceleration for the remaining Strang splittings]\label{cor:main-palindromic}
Corollary~\ref{cor:sharp}(ii) holds after replacing OBABO by any of the four
splittings of Proposition~\ref{prop:HB-palindromic}, with the numerical
stability condition \eqref{eq:stabletuningintro} replaced by
$0<s_\dagger<2(1+\beta)/L$.
\end{corollary}

\begin{proof}
Repeat the proof of the lower-bound assertion of
Corollary~\ref{cor:sharp}, applying
Theorem~\ref{thm:palindromic} in place of
Theorem~\ref{thm:dichotomy}. The normalization argument of
Section~\ref{sec:notation} applies to each splitting, since the
substep maps transform under $(x,v)\mapsto(\sqrt Kx,v)$ exactly as
for OBABO. The smaller time window for OABAO affects only a fixed
constant and does not change the argument.
\end{proof}

\section{Consequences, scope and open problems}\label{sec:discussion}

The main results are mixing time lower bounds. Corollary~\ref{cor:sharp}(ii) shows that no ballistic mixing time bound holds for OBABO uniformly over the class \eqref{eq:classUkappa}: for every step size and friction chosen from the curvature bounds alone, there are targets in the class \eqref{eq:classUkappa} and cold starts whose total variation mixing time is not of order $\sqrt\kappa$, and indeed not of any order $o(\kappa)$. Theorem~\ref{thm:dichotomy} explains why: the chain either mixes no faster than at the nonaccelerated rate $O(\kappa^{-1})$ on some Gaussian target, giving a mixing time of at least order $\kappa$, or admits an attracting cycle of the associated heavy-ball recursion, which dilation converts into mixing times exponential in $R^2$ at fixed condition number (Corollary~\ref{cor:mixing}). By the results of Section~\ref{sec:BAOAB}, the same conclusions hold for the remaining standard Strang splittings, and no splitting admits, under a numerically stable fixed-parameter tuning, a cold-start exponential bound, uniform over the class \eqref{eq:classUkappa}, with convergence rate of larger order than $\kappa^{-1}$ (Corollaries~\ref{cor:main-BAOAB} and~\ref{cor:main-palindromic}). Theorem~\ref{thm:upper} provides the complementary upper bound at the diffusive scale: a fixed-parameter OBABO tuning achieves a total variation mixing time of order $\kappa$ up to a logarithmic factor. Together, the lower and upper bounds show that a cold-start mixing time of order $\kappa$ is optimal for fixed-parameter OBABO, up to a logarithmic factor (Corollary~\ref{cor:sharp}), and thereby settle Question~\ref{q:acceleration}.

These lower bounds may seem to contradict the convergence guarantees available for the exact diffusion and for Gaussian targets, but they do not. The exact process converges to the measure \eqref{eq:gibbs} in relative entropy at a rate of order $\sqrt K$ under the conditions of Lu \citep{Lu2026}, with a corresponding explicit $L^2$ estimate by Cao, Lu and Wang \citep{CaoLuWang2023}; this concerns the continuous dynamics, while the lower bounds concern its fixed-parameter discretizations, and the metastable cycles in case~(C) of Theorem~\ref{thm:dichotomy} are artifacts of a finite step size with no counterpart in the diffusion. On Gaussian targets, the optimally tuned OBABO chain has asymptotic total variation relaxation scale $\sqrt\kappa$ (Corollary~\ref{cor:optimalGaussian}). This is consistent with the lower bounds because the guarantee is uniform only over Gaussian targets: by Theorem~\ref{thm:dichotomy}, a tuning whose worst-case Gaussian rate is accelerated must admit an attracting cycle, and the dilated potentials on which this cycle produces metastability are not Gaussian. Appendix~\ref{app:diffusion} completes the comparison with the diffusion: on every Gaussian target with smallest curvature $K$, the critical friction $\gamma=2\sqrt K$ gives worst-case asymptotic total variation exponent $\sqrt K$ over deterministic initial states, and no other friction gives a larger worst-case exponent.

The deterministic dichotomy used above does not apply directly to the left-endpoint exponential integrators analyzed in \citep{ChengChatterjiBartlettJordan2018,DalalyanRiouDurand2020,MaChatterjiChengFlammarionBartlettJordan2021,KimGruffazParkDurmus2026}. For the Strang splittings treated here, eliminating the velocity yields a noisy heavy-ball recursion with a single current-gradient term. For a left-endpoint exponential integrator, the same elimination instead produces a two-step recurrence involving both $\nabla U(X_k)$ and $\nabla U(X_{k-1})$, with distinct coefficients. Nevertheless, a direct Gaussian argument suffices: Proposition~\ref{prop:exp-euler-gaussian} shows that every fixed choice of step size and friction is either unstable on a Gaussian target with curvature in $[1,\kappa]$, or admits such a target, an invariant law $\pi$, and a deterministic initial state $y$ such that, with $P$ the transition kernel of the left-endpoint exponential integrator on that target,
\begin{equation*}
 \lim_{n\to\infty}
 \norm{\delta_yP^n-\pi}_{\TV}^{1/n}
 \ge1-\frac{64}{\kappa}.
\end{equation*}
For large $\kappa$, this is incompatible with a cold-start bound of the form \eqref{eq:falseuniform} with $\tau(\kappa)=\sqrt\kappa$. Thus the standard left-endpoint exponential integrator likewise admits no cold-start ballistic bound, uniform over the class \eqref{eq:classUkappa}, although the mechanism is purely Gaussian rather than a deterministic cycle.

The results suggest the following informal conjecture. Fix a
consistent explicit, possibly randomized, discretization of
\eqref{eq:KLD} whose number of
pointwise evaluations of \(U\) and \(\nabla U\) per step is bounded
independently of \(\kappa\), and let \(\mathfrak t\) be any
fixed-parameter tuning rule for this discretization. In the normalized
setting \(K=1\), \(L=\kappa\), and \(d=2\), write \(P_\kappa^U\) for
the resulting transition kernel, and suppose that it has a unique
invariant law \(\pi_\kappa^U\) for every
\(U\in\cU_\kappa^2\). The conjecture is that, for every function
\(\tau:(1,\infty)\to(0,\infty)\) satisfying
\(\tau(\kappa)=o(\kappa)\), every \(C\ge1\), \(a,b\ge0\), \(c>0\),
and every \(\kappa_0>1\), there exist
\(\kappa\ge\kappa_0\), \(U\in\cU_\kappa^2\),
\(z\in\mathbb R^4\), and \(\varepsilon\in(0,1/2)\) such that
\[
 t_{\mathrm{mix}}
 \bigl(z,\varepsilon;P_\kappa^U,\pi_\kappa^U\bigr)
 >
 \left\lceil
 \frac{\tau(\kappa)}{c}
 \log\!\left(
   \frac{C\kappa^a
   \bigl(1+W_2(\delta_z,\pi_\kappa^U)\bigr)^b}
   {\varepsilon}
 \right)
 \right\rceil.
\]
A fully formal statement would require specifying the admissible
class of discretizations. Weak consistency, understood as convergence
of the one-step weak generator to the generator of
\eqref{eq:KLD}, constrains only the small-step limit and does not by
itself restrict the kernel at a tuned step size.
The equivalent conclusion at the level of cold-start exponential
bounds holding uniformly over \eqref{eq:classUkappa} is proved for the six standard Strang splittings
in Corollaries~\ref{cor:sharp}, \ref{cor:main-BAOAB}, and
\ref{cor:main-palindromic}. Proposition
\ref{prop:exp-euler-gaussian} gives the corresponding Gaussian rate
lower bound for the left-endpoint exponential integrator.

Randomized time integrators already provide a partial positive result.
Altschuler, Chewi and Zhang \citep{AltschulerChewiZhang2026} prove
that the randomized midpoint discretization admits a relative
entropy, and hence total variation, mixing guarantee whose
condition-number dependence is of order \(\kappa^{5/6}\) up to
logarithmic factors, from initial laws satisfying a controlled
\(\chi^2\)-divergence assumption. This dependence is sublinear in \(\kappa\), but does not
attain the ballistic scale \(\sqrt\kappa\). It also does not
contradict the conjectured mixing time lower bound above, which
concerns point-mass cold starts controlled only through their
Wasserstein distance to stationarity. The randomized
Runge-Kutta-Nystr\"om methods of Bou-Rabee and Kleppe
\citep{BouRabeeKleppe2025} provide another class of time-homogeneous
fixed-parameter randomized integrators. Their analysis establishes
finite-time \(L^2\)-accuracy and discusses the resulting bounds on
asymptotic bias, rather than mixing time dependence on \(\kappa\). The present mixing time lower
bounds do not extend directly to either class. For the randomized
Runge-Kutta-Nystr\"om schemes, in particular, the internal-stage
random variable retains its original distribution under the dilation
of Lemma~\ref{lem:metastability}, so the rescaled update does not
become a small random perturbation of a single deterministic
recurrence. Whether such internal randomization permits a
ballistic mixing bound, uniform over the class \eqref{eq:classUkappa}, from arbitrary point-mass initial states
remains open.

Metropolized OBABO is a particularly natural test of this conjecture.
The unadjusted chain generally converges to a
discretization-dependent invariant law, whereas the Metropolized
chain preserves the measure \eqref{eq:gibbs} exactly
\citep[Algorithm~3.2 and Proposition~3.4]{BouRabee2014};
consequently, the informal assertion that Metropolis adjustment can
only slow the chain does not provide a valid comparison of their
mixing times. Moreover, the acceptance decision and the momentum flip
following rejection change the deterministic dynamics. They may
break the attracting cycles used in the present mixing time lower
bounds and redirect subsequent motion, and could therefore contribute
to an accelerated mixing rate. They may instead cause the chain to
reverse direction and undo earlier progress, as discussed for
generalized HMC in
\citep[Section~2.2]{HoffmanSountsov2022}. Whether Metropolized OBABO
satisfies the conjectured mixing time lower bound remains open.

Since the lower bounds concern the transition laws themselves, no alternative coupling can repair the global theorem; in particular, pairing a uniform Wasserstein contraction at any per-step rate of larger order than $\kappa^{-1}$ with the finite-horizon regularization estimate of \citep{BouRabeeCoxSchieven2026} would produce exactly a cold-start bound excluded by Corollary~\ref{cor:sharp}. The result also leaves open warm-start acceleration, target-dependent parameters, time-dependent splittings, and analyses that compare a discretization to the accelerated exact semigroup without demanding global pointwise contractivity. The finite-horizon coupling of Bou-Rabee, Cox and Schieven \citep{BouRabeeCoxSchieven2026} and the shifted-composition framework of Altschuler, Chewi and Zhang \citep{AltschulerChewiZhang2026} remain natural tools for such restricted positive results.

The lower bounds also motivate modifying the kinetic Langevin dynamics \eqref{eq:KLD}. Gradient-adjusted underdamped Langevin dynamics \citep{ZuoOsherLi2025}, building on the Hessian-free high-resolution dynamics of Li, Zha and Tao \citep{LiZhaTao2022}, augments the kinetic Langevin drift with additional gradient terms and achieves total variation mixing on the ballistic scale $\sqrt\kappa$ on Gaussian targets; extending the proved acceleration beyond the Gaussian case remains open.

\begin{acks}[Acknowledgments]
The author thanks Andreas Eberle, Andre Wibisono, and Jason
Altschuler for helpful conversations.
\end{acks}

\appendix

\section{Gaussian non-acceleration of the left-endpoint exponential integrator}
\label{app:exponential-integrator}

The left-endpoint exponential integrator freezes the force at the
initial position during each step of \eqref{eq:KLD} and solves the
resulting linear stochastic differential equation exactly
\citep{ChengChatterjiBartlettJordan2018,DalalyanRiouDurand2020,
MaChatterjiChengFlammarionBartlettJordan2021,
KimGruffazParkDurmus2026}. The Gaussian non-acceleration result of
this appendix, Proposition~\ref{prop:exp-euler-gaussian} below, is
an application of Lemma~\ref{lem:gaussianTVrate}. For a step size $h>0$, write
$\beta=e^{-\gamma h}$ as in \eqref{eq:r_eta_sigma}. The integrator is
\begin{align}
 X_{k+1}
 &=X_k+\frac{1-\beta}{\gamma}V_k
   -\frac{\gamma h+\beta-1}{\gamma^2}\nabla U(X_k)
   +\zeta_{k+1}^X,\label{eq:exp-euler-x}\\
 V_{k+1}
 &=\beta V_k-\frac{1-\beta}{\gamma}\nabla U(X_k)
   +\zeta_{k+1}^V,\label{eq:exp-euler-v}
\end{align}
where the noise pairs are independent over $k$ and have the
representation
\begin{equation}\label{eq:exp-euler-noise}
 \begin{pmatrix}\zeta_{k+1}^X\\ \zeta_{k+1}^V\end{pmatrix}
 =\sqrt{2\gamma}\int_0^h
 \begin{pmatrix}
  \bigl(1-e^{-\gamma(h-t)}\bigr)/\gamma\\
  e^{-\gamma(h-t)}
 \end{pmatrix}dW_{k,t},
\end{equation}
with independent Brownian motions $W_{k,\cdot}$. The one-step
noise covariance is positive definite on $\R^{2d}$. Indeed, its
quadratic form at $(a,b)\in\R^d\times\R^d$ is
\begin{equation*}
 2\gamma\int_0^h
 \left|
 \frac{1-e^{-\gamma(h-t)}}{\gamma}\,a
 +e^{-\gamma(h-t)}\,b
 \right|^2dt,
\end{equation*}
and if this integral vanishes, continuity gives
$\frac{1-e^{-\gamma(h-t)}}{\gamma}\,a+e^{-\gamma(h-t)}\,b=0$ for
every $t\in[0,h]$; evaluating at $t=h$ gives $b=0$, and then
$a=0$.

For the one-dimensional quadratic potential
$U_\lambda(x)=\lambda x^2/2$, the deterministic part of
\eqref{eq:exp-euler-x}-\eqref{eq:exp-euler-v}, acting on $(X,V)$,
is
\begin{equation}\label{eq:exp-euler-matrix}
 M_\lambda=
 \begin{pmatrix}
  1-uA\lambda&D/\gamma\\
  -D\lambda/\gamma&\beta
 \end{pmatrix},
 \qquad
 u=\frac{1}{\gamma^2},\qquad
 z=\gamma h,\qquad D=1-\beta,\qquad
 A=z+\beta-1.
\end{equation}

\needspace{22\baselineskip}%
\begin{proposition}[Gaussian non-acceleration of the left-endpoint exponential integrator]
\label{prop:exp-euler-gaussian}
Let $\kappa>1$ and let $h,\gamma>0$. For
$U_\lambda(x)=\lambda x^2/2$ on $\R$, let $M_\lambda$ be the
matrix in \eqref{eq:exp-euler-matrix}. If $M_\kappa$ is not Schur
stable, then the left-endpoint exponential integrator for
$U_\kappa$ has no invariant probability law. If $M_\kappa$ is
Schur stable, then $M_\lambda$ is Schur stable for every
$\lambda\in[1,\kappa]$, and
\begin{equation}\label{eq:exp-euler-slow-rate}
 \rho(M_1)\ge1-\frac{64}{\kappa}.
\end{equation}
In this case, the chain for $U_1$ has a unique invariant law
$\pi_1$. Let $P$ denote the transition kernel of the chain for
$U_1$. For every initial state $y\in\R^2$,
\begin{equation}\label{eq:exp-euler-onepoint}
 \limsup_{n\to\infty}
 \norm{\delta_yP^n-\pi_1}_{\TV}^{1/n}
 \le\rho(M_1),
\end{equation}
and the limit exists and equals $\rho(M_1)$ for a suitable $y$.
Moreover, there are initial states $y,\widetilde y\in\R^2$ such
that
\begin{equation}\label{eq:exp-euler-TV-rate}
 \lim_{n\to\infty}
 \norm{\delta_yP^n-\delta_{\widetilde y}P^n}_{\TV}^{1/n}
 =\rho(M_1)\ge1-\frac{64}{\kappa}.
\end{equation}
\end{proposition}

\begin{proof}
Define
\begin{equation*}
 B=1-(1+z)\beta=D-z\beta,
 \qquad Q=u\kappa.
\end{equation*}
Both $A$ and $B$ are positive, and
\begin{equation}\label{eq:exp-euler-identities}
 A+B=zD,\qquad
 A-B=z(1+\beta)-2D>0.
\end{equation}
Indeed, $B=\int_0^z t e^{-t}\,dt>0$, the second expression in
\eqref{eq:exp-euler-identities} vanishes at $z=0$ and has derivative
$B>0$, and then $A=B+(A-B)>0$.

The trace and determinant of $M_\lambda$ are
\begin{equation*}
 \operatorname{tr}M_\lambda=1+\beta-uA\lambda,
 \qquad
 \det M_\lambda=\beta+uB\lambda.
\end{equation*}
The Jury conditions for a real quadratic give Schur stability exactly
when
\begin{align*}
 1-\det M_\lambda&=D-uB\lambda>0,\\
 1-\operatorname{tr}M_\lambda+\det M_\lambda
   &=u\lambda(A+B)>0,\\
 1+\operatorname{tr}M_\lambda+\det M_\lambda
   &=2(1+\beta)-u\lambda(A-B)>0.
\end{align*}
Thus stability of $M_\kappa$ implies
\begin{equation}\label{eq:exp-euler-stability-bounds}
 QB<D,\qquad Q(A-B)<2(1+\beta),
\end{equation}
and, by monotonicity in $u\lambda$, also implies stability of every
$M_\lambda$ with $\lambda\in[1,\kappa]$.

We first deduce the uniform estimate
\begin{equation}\label{eq:Qz-bound}
 Qz<16.
\end{equation}
If $0<z\le1$, then
\begin{equation*}
 B=\int_0^z t e^{-t}\,dt\ge\frac{e^{-1}z^2}{2},
 \qquad D\le z,
\end{equation*}
so the first inequality in \eqref{eq:exp-euler-stability-bounds}
gives $Qz<zD/B\le2e<6$. If $1\le z\le4$, then
$B\ge\int_0^1t e^{-t}\,dt=1-2/e$ and hence
$Qz<zD/B\le4/(1-2/e)<16$. Finally, if $z\ge4$, then
$A-B\ge z-2\ge z/2$, and the second inequality in
\eqref{eq:exp-euler-stability-bounds} gives
$Qz<4(1+\beta)\le8$. This proves \eqref{eq:Qz-bound}.

It remains to examine the low-curvature matrix $M_1$. Put $q=u=Q/\kappa$.
Writing each eigenvalue of $M_1$ as $1-\mu$, the characteristic equation becomes, using $A+B=zD$,
\begin{equation}\label{eq:exp-euler-s-equation}
 \mu^2-(D+qA)\mu+qzD=0.
\end{equation}
Suppose first that the discriminant in
\eqref{eq:exp-euler-s-equation} is nonnegative. The two roots of
\eqref{eq:exp-euler-s-equation} are then positive, since their sum
$D+qA$ and their product $qzD$ are positive. If $\mu_-$ is the
smaller root, then, since the larger root is at least $(D+qA)/2$
and the product of the roots is $qzD$,
\begin{equation*}
 \mu_-\le\frac{2qzD}{D+qA}\le2qz
 =\frac{2Qz}{\kappa}<\frac{32}{\kappa}.
\end{equation*}
For $\kappa\ge64$, the corresponding eigenvalue $1-\mu_-$ is positive,
so $\rho(M_1)\ge1-32/\kappa$; for $\kappa<64$, the claimed bound
\eqref{eq:exp-euler-slow-rate} is automatic.

If the discriminant is negative, then
\begin{equation*}
 D^2\le(D+qA)^2<4qzD,
 \qquad\text{and hence}\qquad
 D<4qz<\frac{64}{\kappa}.
\end{equation*}
The eigenvalues are a complex conjugate pair, so
\begin{equation*}
 \rho(M_1)=\sqrt{\det M_1}
 =\sqrt{\beta+qB}\ge\sqrt{1-D}\ge1-D
 >1-\frac{64}{\kappa}.
\end{equation*}
This proves \eqref{eq:exp-euler-slow-rate}.

Let $C\succ0$ denote the one-step noise covariance in
\eqref{eq:exp-euler-noise}, specialized to one dimension, and
choose $G\in\R^{2\times2}$ with $GG^\top=C$. The chain for
$U_\lambda$ can then be written
\begin{equation*}
 Y_{k+1}=M_\lambda Y_k+G\xi_{k+1},
 \qquad Y_k=(X_k,V_k),
\end{equation*}
where the $\xi_k$ are i.i.d.\ standard Gaussian vectors in $\R^2$.
When $M_\kappa$ is Schur stable, so is $M_1$, and therefore
$\rho(M_1)<1$. Lemma~\ref{lem:gaussianTVrate} applies with
$M=M_1$, $N=N'=2$, and $n_0=1$. It gives the unique invariant law
$\pi_1$, the one-point conclusions in
\eqref{eq:exp-euler-onepoint}, and the pairwise identity
\eqref{eq:exp-euler-TV-rate}.

Finally, suppose $M_\kappa$ is not Schur stable, so that
$\rho(M_\kappa)\ge1$, and let $\phi$ be the characteristic
function of a hypothetical invariant law. Iterating the invariance
identity gives
\begin{equation*}
 \abs{\phi(t)}
 =\abs{\phi((M_\kappa^\top)^nt)}
  \exp\left(-\frac12\sum_{j=0}^{n-1}
  \abs{C^{1/2}(M_\kappa^\top)^jt}^2\right).
\end{equation*}
Choose an eigenvalue of $M_\kappa^\top$ of modulus at least one. If
it is real, let $t$ be a corresponding nonzero real eigenvector; if
it is complex, let $t$ be any nonzero vector in the two-dimensional
real invariant subspace of the conjugate pair, on which the
restriction of $M_\kappa^\top$ is similar to a rotation scaled by
the eigenvalue modulus. In either case there is a constant
$c>0$ such that
$\abs{(M_\kappa^\top)^jt}\ge c\abs t$ for every $j\ge0$. Since
$C\succ0$, the sum in the exponent diverges as $n\to\infty$; as
$\abs\phi\le1$, the displayed identity forces $\phi(t)=0$. The same
argument applies to every nonzero scalar multiple of $t$, so $\phi$
vanishes along a punctured line through the origin, contradicting
the continuity of $\phi$ at zero and $\phi(0)=1$.
Thus the chain has no invariant probability law.
\end{proof}

\section{Sharp cold-start total variation rates for the kinetic Langevin diffusion on Gaussian targets}
\label{app:diffusion}

As an application of Lemma~\ref{lem:gaussianTVrate}, this appendix
determines the sharp worst-case asymptotic total variation decay
rate over deterministic initial states on every Gaussian target
whose smallest curvature is $K$.

Fix $K>0$, $d\ge1$, and $\gamma>0$, let $H\in\R^{d\times d}$ be a
symmetric matrix with smallest eigenvalue $K$, and consider
\eqref{eq:KLD} with the Gaussian potential
$U_H(x)=\tfrac12x^\top Hx$. The process
$Z_t=(X_t,V_t)$ is the linear diffusion $dZ_t=AZ_t\,dt+B\,dW_t$ with
\begin{equation*}
 A=\begin{pmatrix}0&\Id_d\\-H&-\gamma\,\Id_d\end{pmatrix},
 \qquad
 B=\begin{pmatrix}0\\\sqrt{2\gamma}\,\Id_d\end{pmatrix},
\end{equation*}
and transition kernels $(P_t)_{t\ge0}$. For the potential $U_H$,
the measure \eqref{eq:gibbs} is
$\mu_H=\cN(0,H^{-1})\otimes\cN(0,\Id_d)$. Let $\alpha(\gamma,H)$
denote the spectral abscissa of $A$, that is, the largest real part
of an eigenvalue of $A$.

For every friction parameter $\gamma>0$,
Proposition~\ref{prop:diffusion-coldstart} gives a deterministic
initial state whose asymptotic decay exponent is
$-\alpha(\gamma,H)$. Since $-\alpha(\gamma,H)\le\sqrt K$, with
equality if and only if $\gamma=2\sqrt K$
(Proposition~\ref{prop:diffusion-coldstart}(i)), no choice of
friction makes all deterministic initial states converge with an
exponent larger than $\sqrt K$. At the critical friction,
Lemma~\ref{lem:gaussianTVrate} also gives an upper exponent
$\sqrt K$ from every deterministic initial state, with equality
for a suitable one. Thus $\sqrt K$ is the optimal worst-case
asymptotic total variation exponent over deterministic initial
states. This statement does not assert that every initial state
has the same exact rate or that the multiplicative constant can be
chosen independently of the initial state.

The upper exponent can also be recovered from entropy estimates: the law of the diffusion started at a point is a nondegenerate Gaussian at every positive time, so it has finite relative entropy with respect to equilibrium, and at the critical friction the relative entropy bound of Arnold and Erb \citep[Theorem~4.9(ii)]{ArnoldErb2014}, whose arbitrarily small exponent loss disappears at the level of the asymptotic exponent, combined with Pinsker's inequality gives the upper exponent $\sqrt K$ from every deterministic initial state.

For nonquadratic strongly convex smooth potentials, the entropy bound of Lu \citep{Lu2026} gives a rate of order $\sqrt K$ when $\gamma\asymp\sqrt K$, but a Dirac initial condition has infinite relative entropy, so that result does not directly provide a cold-start bound. Guillin and Monmarch\'e \citep{GuillinMonmarche2016} and Eberle and L\"orler \citep{EberleLorler2026} study related questions for linear drifts and for non-reversible lifts, respectively; both concern entropy or relaxation-time notions rather than the deterministic-start total variation statement proved here. Determining the sharp worst-case cold-start total variation rate for nonquadratic strongly convex potentials remains open.

\begin{proposition}[Cold-start total variation rates at every friction]\label{prop:diffusion-coldstart}
Fix $K>0$, $d\ge1$, and $\gamma>0$, and let $H\in\R^{d\times d}$ be
symmetric with smallest eigenvalue $K$.
\begin{enumerate}
\item[\rm(i)] The eigenvalues of $A$ are the roots of
$\mu^2+\gamma\mu+\lambda=0$ as $\lambda$ ranges over the eigenvalues
of $H$, and
$\alpha(\gamma,H)\ge-\sqrt K$, with equality if and only if
$\gamma=2\sqrt K$.
\item[\rm(ii)] The measure $\mu_H$ is the unique invariant law of
$(P_t)_{t\ge0}$. For every initial state $z\in\R^{2d}$,
\begin{equation}\label{eq:diffusion-onepoint}
 \limsup_{t\to\infty}
 \norm{\delta_zP_t-\mu_H}_{\TV}^{1/t}
 \le e^{\alpha(\gamma,H)},
\end{equation}
and the limit exists and equals $e^{\alpha(\gamma,H)}$ for a
suitable $z$. Moreover, there are initial states
$y,\widetilde y\in\R^{2d}$ such that
\begin{equation}\label{eq:diffusion-TV-rate}
 \lim_{t\to\infty}
 \norm{\delta_yP_t-\delta_{\widetilde y}P_t}_{\TV}^{1/t}
 =e^{\alpha(\gamma,H)}
 \ \ge\ e^{-\sqrt K}.
\end{equation}
\item[\rm(iii)] There are an initial state $z\in\R^{2d}$ and a
constant $c>0$ such that, for every $\eps\in(0,c/4]$,
\begin{equation*}
 t_{\mathrm{mix}}\bigl(z,\eps;(P_t)_{t\ge0},\mu_H\bigr)
 \ \ge\ \frac{1}{\sqrt K}\,\log\frac{c}{4\eps},
\end{equation*}
where the mixing time is defined as in
Section~\ref{sec:notation}, with $t\in[0,\infty)$ in place of $n$.
\end{enumerate}
\end{proposition}

\begin{proof}
(i) In an orthonormal basis of eigenvectors of $H$, the matrix $A$
decomposes into $d$ two-by-two blocks with characteristic
polynomials $\mu^2+\gamma\mu+\lambda$, where $\lambda$ ranges over
the eigenvalues of $H$; this gives the list of eigenvalues. For
fixed $\gamma$, the largest real part of the roots of
$\mu^2+\gamma\mu+\lambda=0$ equals $-\gamma/2$ for
$\lambda\ge\gamma^2/4$ and
$\tfrac12\bigl(-\gamma+\sqrt{\gamma^2-4\lambda}\bigr)$ for
$\lambda<\gamma^2/4$, so it is nonincreasing in $\lambda$, and
$\alpha(\gamma,H)$ equals its value at the smallest eigenvalue
$\lambda=K$. If $\gamma\le2\sqrt K$, this value is
$-\gamma/2\ge-\sqrt K$, with equality only at $\gamma=2\sqrt K$. If
$\gamma>2\sqrt K$, the two roots for $\lambda=K$ are real and
negative with product $K$, so the larger root $\mu_+$ satisfies
$\abs{\mu_+}\le\sqrt K$; equality would force
$\abs{\mu_-}=\sqrt K$ and hence
$\gamma=\abs{\mu_+}+\abs{\mu_-}=2\sqrt K$, a contradiction.

(ii) Observed at integer times, the diffusion is the Gaussian
autoregressive process
\begin{equation*}
 Z_{n+1}=e^AZ_n+G\xi_{n+1},
 \qquad
 GG^\top=\int_0^1e^{As}BB^\top e^{A^\top s}\,ds.
\end{equation*}
We first verify directly that $GG^\top\succ0$. If
$y=(p,q)\in\R^d\times\R^d$ satisfies $y^\top GG^\top y=0$, then
\begin{equation*}
 \int_0^1\abs{B^\top e^{A^\top s}y}^2\,ds=0.
\end{equation*}
Continuity therefore gives $B^\top e^{A^\top s}y=0$ for every
$s\in[0,1]$. Evaluating at $s=0$ gives $\sqrt{2\gamma}\,q=0$,
while differentiating at $s=0$ gives
\begin{equation*}
 0=B^\top A^\top y=(AB)^\top y
 =\sqrt{2\gamma}\,(p-\gamma q).
\end{equation*}
Hence $p=q=0$, proving $GG^\top\succ0$.

By (i), the eigenvalues of $A$ are the roots of
$\mu^2+\gamma\mu+\lambda=0$ with $\lambda$ an eigenvalue of $H$;
each such pair of roots has product $\lambda>0$ and sum
$-\gamma<0$, so every eigenvalue of $A$ has negative real part.
The eigenvalues of $e^A$ are $e^\mu$, where $\mu$ ranges over the
eigenvalues of $A$, and hence
$\rho(e^A)=e^{\alpha(\gamma,H)}<1$.
Lemma~\ref{lem:gaussianTVrate}, applied with $M=e^A$ and $n_0=1$,
gives a unique invariant law for the integer-time chain, the
one-point upper bound in \eqref{eq:diffusion-onepoint} at integer
times with equality as a limit for a suitable initial state, and
the pairwise identity \eqref{eq:diffusion-TV-rate} at integer
times; the inequality $e^{\alpha(\gamma,H)}\ge e^{-\sqrt K}$ in
\eqref{eq:diffusion-TV-rate} follows from (i).

The measure $\mu_H$ is invariant for every $P_t$, in particular
for $P_1$, so the unique invariant law of the integer-time chain
is $\mu_H$. Any invariant law of $(P_t)_{t\ge0}$ is invariant for
$P_1$ and therefore equals $\mu_H$; hence $\mu_H$ is the unique
invariant law of the continuous-time process.

Both
$t\mapsto\norm{\delta_zP_t-\mu_H}_{\TV}$ and
$t\mapsto\norm{\delta_yP_t-\delta_{\widetilde y}P_t}_{\TV}$ are
nonincreasing, since
$\delta_zP_t=(\delta_zP_s)P_{t-s}$ and $\mu_H=\mu_HP_{t-s}$ for
$0\le s\le t$, and applying a common Markov kernel does not
increase total variation distance. Taking $n=\lfloor t\rfloor$ and
bounding each quantity at time $t$ between its values at times $n$
and $n+1$ extends the integer-time conclusions to every real
$t\ge0$.

(iii) Take as $z$ the initial state from the final part of the
proof of Lemma~\ref{lem:gaussianTVrate}, which attains equality in
\eqref{eq:diffusion-onepoint}. The lower bound established there
gives a constant $c_0>0$ with
\begin{equation*}
 \norm{\delta_zP_n-\mu_H}_{\TV}
 \ge c_0e^{\alpha(\gamma,H)n}
\end{equation*}
for all sufficiently large $n$. Each left-hand side is positive,
because $\delta_zP_n$ is Gaussian with the nonzero mean $e^{An}z$
while $\mu_H$ is centered; decreasing $c_0$ if necessary, the
bound therefore holds for every $n\in\mathbb N$. If
$n=\lfloor t\rfloor$, the monotonicity established in (ii) gives
\begin{equation*}
 \norm{\delta_zP_t-\mu_H}_{\TV}
 \ge\norm{\delta_zP_{n+1}-\mu_H}_{\TV}
 \ge c_0e^{\alpha(\gamma,H)(n+1)}
 \ge c_0e^{\alpha(\gamma,H)}e^{\alpha(\gamma,H)t}.
\end{equation*}
Thus, with $c=c_0e^{\alpha(\gamma,H)}$,
$\norm{\delta_zP_t-\mu_H}_{\TV}\ge c\,e^{\alpha(\gamma,H)t}$ for
every $t\ge0$. For $\eps\in(0,c/4]$, set
$t^*=\log(c/4\eps)/\abs{\alpha(\gamma,H)}$. Then
$\norm{\delta_zP_{t^*}-\mu_H}_{\TV}\ge c\,e^{\alpha(\gamma,H)t^*}
=4\eps>\eps$, so, by the monotonicity established in (ii),
$t_{\mathrm{mix}}(z,\eps;(P_t)_{t\ge0},\mu_H)\ge t^*$. Part (i)
gives $\abs{\alpha(\gamma,H)}\le\sqrt K$, so
$t^*\ge K^{-1/2}\log(c/4\eps)$, completing the proof.
\end{proof}

\bibliographystyle{imsart-number}
\bibliography{obabo_non_acceleration}

\begin{thebibliography}{41}

\bibitem{AltschulerChewiZhang2026}
\begin{binproceedings}[author]
\bauthor{\bsnm{Altschuler},~\bfnm{Jason~M.}\binits{J.~M.}},
  \bauthor{\bsnm{Chewi},~\bfnm{Sinho}\binits{S.}} \AND
  \bauthor{\bsnm{Zhang},~\bfnm{Matthew~S.}\binits{M.~S.}}
(\byear{2026}).
\btitle{Shifted Composition {IV}: Toward Ballistic Acceleration for Log-Concave
  Sampling}.
In \bbooktitle{Proceedings of the 58th Annual ACM Symposium on Theory of
  Computing}
\bpages{1739--1750}.
\bpublisher{Association for Computing Machinery}, \baddress{New York, NY, USA}.
\bdoi{10.1145/3798129.3800881}
\end{binproceedings}
\endbibitem

\bibitem{ArnoldErb2014}
\begin{barticle}[author]
\bauthor{\bsnm{Arnold},~\bfnm{Anton}\binits{A.}} \AND
  \bauthor{\bsnm{Erb},~\bfnm{Jan}\binits{J.}}
(\byear{2014}).
\btitle{Sharp Entropy Decay for Hypocoercive and Non-Symmetric
  {F}okker-{P}lanck Equations with Linear Drift}.
\bjournal{arXiv preprint arXiv:1409.5425v2}.
\bdoi{10.48550/arXiv.1409.5425}
\end{barticle}
\endbibitem

\bibitem{BakryGentilLedoux2014}
\begin{bbook}[author]
\bauthor{\bsnm{Bakry},~\bfnm{Dominique}\binits{D.}},
  \bauthor{\bsnm{Gentil},~\bfnm{Ivan}\binits{I.}} \AND
  \bauthor{\bsnm{Ledoux},~\bfnm{Michel}\binits{M.}}
(\byear{2014}).
\btitle{Analysis and Geometry of Markov Diffusion Operators}.
\bseries{Grundlehren der mathematischen Wissenschaften}
\bvolume{348}.
\bpublisher{Springer}, \baddress{Cham}.
\bdoi{10.1007/978-3-319-00227-9}
\end{bbook}
\endbibitem

\bibitem{Billingsley1995}
\begin{bbook}[author]
\bauthor{\bsnm{Billingsley},~\bfnm{Patrick}\binits{P.}}
(\byear{1995}).
\btitle{Probability and Measure},
\bedition{3} ed.
\bpublisher{John Wiley \& Sons}, \baddress{New York}.
\bmrnumber{1324786}
\end{bbook}
\endbibitem

\bibitem{BouRabee2014}
\begin{barticle}[author]
\bauthor{\bsnm{Bou-Rabee},~\bfnm{Nawaf}\binits{N.}}
(\byear{2014}).
\btitle{Time Integrators for Molecular Dynamics}.
\bjournal{Entropy}
\bvolume{16}
\bpages{138--162}.
\bdoi{10.3390/e16010138}
\end{barticle}
\endbibitem

\bibitem{BouRabeeCoxSchieven2026}
\begin{barticle}[author]
\bauthor{\bsnm{Bou-Rabee},~\bfnm{Nawaf}\binits{N.}},
  \bauthor{\bsnm{Cox},~\bfnm{Sonja}\binits{S.}} \AND
  \bauthor{\bsnm{Schieven},~\bfnm{Roy}\binits{R.}}
(\byear{2026}).
\btitle{On Couplings for Kinetic Langevin Diffusions}.
\bjournal{arXiv preprint arXiv:2605.31088v13}.
\bdoi{10.48550/arXiv.2605.31088}
\end{barticle}
\endbibitem

\bibitem{BouRabeeEberle2023}
\begin{barticle}[author]
\bauthor{\bsnm{Bou-Rabee},~\bfnm{Nawaf}\binits{N.}} \AND
  \bauthor{\bsnm{Eberle},~\bfnm{Andreas}\binits{A.}}
(\byear{2023}).
\btitle{Mixing Time Guarantees for Unadjusted {H}amiltonian {M}onte {C}arlo}.
\bjournal{Bernoulli}
\bvolume{29}
\bpages{75--104}.
\bdoi{10.3150/21-BEJ1450}
\end{barticle}
\endbibitem

\bibitem{BouRabeeKleppe2025}
\begin{barticle}[author]
\bauthor{\bsnm{Bou-Rabee},~\bfnm{Nawaf}\binits{N.}} \AND
  \bauthor{\bsnm{Kleppe},~\bfnm{Tore~Selland}\binits{T.~S.}}
(\byear{2025}).
\btitle{Randomized {Runge--Kutta--Nystr\"om} Methods for Unadjusted
  {Hamiltonian} and Kinetic {Langevin} {Monte Carlo}}.
\bjournal{Mathematics of Computation}
\bvolume{94}
\bpages{2839--2865}.
\bdoi{10.1090/mcom/4061}
\end{barticle}
\endbibitem

\bibitem{BouRabeeOberdorster2024}
\begin{barticle}[author]
\bauthor{\bsnm{Bou-Rabee},~\bfnm{Nawaf}\binits{N.}} \AND
  \bauthor{\bsnm{Oberd\"{o}rster},~\bfnm{Stefan}\binits{S.}}
(\byear{2024}).
\btitle{Mixing of {M}etropolis-Adjusted {M}arkov Chains via Couplings: The High
  Acceptance Regime}.
\bjournal{Electronic Journal of Probability}
\bvolume{29}
\bpages{1--27}.
\bdoi{10.1214/24-EJP1150}
\end{barticle}
\endbibitem

\bibitem{BouRabeeSchuh2023}
\begin{barticle}[author]
\bauthor{\bsnm{Bou-Rabee},~\bfnm{Nawaf}\binits{N.}} \AND
  \bauthor{\bsnm{Schuh},~\bfnm{Katharina}\binits{K.}}
(\byear{2023}).
\btitle{Convergence of Unadjusted {H}amiltonian {M}onte {C}arlo for Mean-Field
  Models}.
\bjournal{Electronic Journal of Probability}
\bvolume{28}
\bpages{1--40}.
\bdoi{10.1214/23-EJP970}
\end{barticle}
\endbibitem

\bibitem{BouRabeeVandenEijnden2010}
\begin{barticle}[author]
\bauthor{\bsnm{Bou-Rabee},~\bfnm{Nawaf}\binits{N.}} \AND
  \bauthor{\bsnm{Vanden-Eijnden},~\bfnm{Eric}\binits{E.}}
(\byear{2010}).
\btitle{Pathwise Accuracy and Ergodicity of Metropolized Integrators for
  {SDEs}}.
\bjournal{Communications on Pure and Applied Mathematics}
\bvolume{63}
\bpages{655--696}.
\bdoi{10.1002/cpa.20306}
\end{barticle}
\endbibitem

\bibitem{BussiParrinello2007}
\begin{barticle}[author]
\bauthor{\bsnm{Bussi},~\bfnm{Giovanni}\binits{G.}} \AND
  \bauthor{\bsnm{Parrinello},~\bfnm{Michele}\binits{M.}}
(\byear{2007}).
\btitle{Accurate Sampling Using Langevin Dynamics}.
\bjournal{Physical Review E}
\bvolume{75}
\bpages{056707}.
\bdoi{10.1103/PhysRevE.75.056707}
\end{barticle}
\endbibitem

\bibitem{CaoLuWang2023}
\begin{barticle}[author]
\bauthor{\bsnm{Cao},~\bfnm{Yu}\binits{Y.}},
  \bauthor{\bsnm{Lu},~\bfnm{Jianfeng}\binits{J.}} \AND
  \bauthor{\bsnm{Wang},~\bfnm{Lihan}\binits{L.}}
(\byear{2023}).
\btitle{On Explicit {$L^2$}-Convergence Rate Estimate for Underdamped
  {L}angevin Dynamics}.
\bjournal{Archive for Rational Mechanics and Analysis}
\bvolume{247}
\bpages{90}.
\bdoi{10.1007/s00205-023-01922-4}
\end{barticle}
\endbibitem

\bibitem{ChakMonmarche2026}
\begin{barticle}[author]
\bauthor{\bsnm{Chak},~\bfnm{Martin}\binits{M.}} \AND
  \bauthor{\bsnm{Monmarch\'e},~\bfnm{Pierre}\binits{P.}}
(\byear{2026}).
\btitle{Reflection Coupling for Unadjusted Generalized {H}amiltonian {M}onte
  {C}arlo in the Nonconvex Stochastic Gradient Case}.
\bjournal{IMA Journal of Numerical Analysis}
\bvolume{46}
\bpages{2216--2266}.
\bdoi{10.1093/imanum/draf045}
\end{barticle}
\endbibitem

\bibitem{ChengChatterjiBartlettJordan2018}
\begin{binproceedings}[author]
\bauthor{\bsnm{Cheng},~\bfnm{Xiang}\binits{X.}},
  \bauthor{\bsnm{Chatterji},~\bfnm{Niladri~S.}\binits{N.~S.}},
  \bauthor{\bsnm{Bartlett},~\bfnm{Peter~L.}\binits{P.~L.}} \AND
  \bauthor{\bsnm{Jordan},~\bfnm{Michael~I.}\binits{M.~I.}}
(\byear{2018}).
\btitle{Underdamped {L}angevin {MCMC}: A Non-Asymptotic Analysis}.
In \bbooktitle{Proceedings of the 31st Conference on Learning Theory}.
\bseries{Proceedings of Machine Learning Research}
\bvolume{75}
\bpages{300--323}.
\bpublisher{PMLR}.
\end{binproceedings}
\endbibitem

\bibitem{DalalyanRiouDurand2020}
\begin{barticle}[author]
\bauthor{\bsnm{Dalalyan},~\bfnm{Arnak~S.}\binits{A.~S.}} \AND
  \bauthor{\bsnm{Riou-Durand},~\bfnm{Lionel}\binits{L.}}
(\byear{2020}).
\btitle{On Sampling from a Log-Concave Density Using Kinetic {L}angevin
  Diffusions}.
\bjournal{Bernoulli}
\bvolume{26}
\bpages{1956--1988}.
\bdoi{10.3150/19-BEJ1178}
\end{barticle}
\endbibitem

\bibitem{DolbeaultMouhotSchmeiser2015}
\begin{barticle}[author]
\bauthor{\bsnm{Dolbeault},~\bfnm{Jean}\binits{J.}},
  \bauthor{\bsnm{Mouhot},~\bfnm{Cl\'{e}ment}\binits{C.}} \AND
  \bauthor{\bsnm{Schmeiser},~\bfnm{Christian}\binits{C.}}
(\byear{2015}).
\btitle{Hypocoercivity for Linear Kinetic Equations Conserving Mass}.
\bjournal{Transactions of the American Mathematical Society}
\bvolume{367}
\bpages{3807--3828}.
\bdoi{10.1090/S0002-9947-2015-06012-7}
\end{barticle}
\endbibitem

\bibitem{EberleGuillinZimmer2019}
\begin{barticle}[author]
\bauthor{\bsnm{Eberle},~\bfnm{Andreas}\binits{A.}},
  \bauthor{\bsnm{Guillin},~\bfnm{Arnaud}\binits{A.}} \AND
  \bauthor{\bsnm{Zimmer},~\bfnm{Raphael}\binits{R.}}
(\byear{2019}).
\btitle{Couplings and Quantitative Contraction Rates for {L}angevin Dynamics}.
\bjournal{The Annals of Probability}
\bvolume{47}
\bpages{1982--2010}.
\bdoi{10.1214/18-AOP1299}
\end{barticle}
\endbibitem

\bibitem{EberleLorler2026}
\begin{barticle}[author]
\bauthor{\bsnm{Eberle},~\bfnm{Andreas}\binits{A.}} \AND
  \bauthor{\bsnm{L\"{o}rler},~\bfnm{Francis}\binits{F.}}
(\byear{2026}).
\btitle{Non-Reversible Lifts of Reversible Diffusion Processes and Relaxation
  Times}.
\bjournal{Probability Theory and Related Fields}
\bvolume{194}
\bpages{173--203}.
\bdoi{10.1007/s00440-024-01308-x}
\end{barticle}
\endbibitem

\bibitem{FaureSchreiber2014}
\begin{barticle}[author]
\bauthor{\bsnm{Faure},~\bfnm{Mathieu}\binits{M.}} \AND
  \bauthor{\bsnm{Schreiber},~\bfnm{Sebastian~J.}\binits{S.~J.}}
(\byear{2014}).
\btitle{Quasi-Stationary Distributions for Randomly Perturbed Dynamical
  Systems}.
\bjournal{Annals of Applied Probability}
\bvolume{24}
\bpages{553--598}.
\bdoi{10.1214/13-AAP923}
\end{barticle}
\endbibitem

\bibitem{FreidlinWentzell2012}
\begin{bbook}[author]
\bauthor{\bsnm{Freidlin},~\bfnm{Mark~I.}\binits{M.~I.}} \AND
  \bauthor{\bsnm{Wentzell},~\bfnm{Alexander~D.}\binits{A.~D.}}
(\byear{2012}).
\btitle{Random Perturbations of Dynamical Systems},
\bedition{3} ed.
\bseries{Grundlehren der Mathematischen Wissenschaften}
\bvolume{260}.
\bpublisher{Springer}, \baddress{Heidelberg}.
\bdoi{10.1007/978-3-642-25847-3}
\end{bbook}
\endbibitem

\bibitem{GoujaudTaylorDieuleveut2025}
\begin{barticle}[author]
\bauthor{\bsnm{Goujaud},~\bfnm{Baptiste}\binits{B.}},
  \bauthor{\bsnm{Taylor},~\bfnm{Adrien~B.}\binits{A.~B.}} \AND
  \bauthor{\bsnm{Dieuleveut},~\bfnm{Aymeric}\binits{A.}}
(\byear{2025}).
\btitle{Provable Non-Accelerations of the Heavy-Ball Method}.
\bjournal{Mathematical Programming, Series B}.
\bnote{Published online October 27, 2025}.
\bdoi{10.1007/s10107-025-02269-2}
\end{barticle}
\endbibitem

\bibitem{GouraudLeBrisMajkaMonmarche2025}
\begin{barticle}[author]
\bauthor{\bsnm{Gouraud},~\bfnm{Nicola\"{i}}\binits{N.}},
  \bauthor{\bsnm{Le~Bris},~\bfnm{Pierre}\binits{P.}},
  \bauthor{\bsnm{Majka},~\bfnm{Adrien}\binits{A.}} \AND
  \bauthor{\bsnm{Monmarch\'{e}},~\bfnm{Pierre}\binits{P.}}
(\byear{2025}).
\btitle{{HMC} and Underdamped {L}angevin United in the Unadjusted Convex Smooth
  Case}.
\bjournal{SIAM/ASA Journal on Uncertainty Quantification}
\bvolume{13}
\bpages{278--303}.
\bdoi{10.1137/23M1608963}
\end{barticle}
\endbibitem

\bibitem{GuillinMonmarche2016}
\begin{barticle}[author]
\bauthor{\bsnm{Guillin},~\bfnm{Arnaud}\binits{A.}} \AND
  \bauthor{\bsnm{Monmarch\'e},~\bfnm{Pierre}\binits{P.}}
(\byear{2016}).
\btitle{Optimal Linear Drift for the Speed of Convergence of an Hypoelliptic
  Diffusion}.
\bjournal{Electronic Communications in Probability}
\bvolume{21}
\bpages{1--14}.
\bdoi{10.1214/16-ECP25}
\end{barticle}
\endbibitem

\bibitem{HestenesStiefel1952}
\begin{barticle}[author]
\bauthor{\bsnm{Hestenes},~\bfnm{Magnus~R.}\binits{M.~R.}} \AND
  \bauthor{\bsnm{Stiefel},~\bfnm{Eduard}\binits{E.}}
(\byear{1952}).
\btitle{Methods of Conjugate Gradients for Solving Linear Systems}.
\bjournal{Journal of Research of the National Bureau of Standards}
\bvolume{49}
\bpages{409--436}.
\bdoi{10.6028/jres.049.044}
\end{barticle}
\endbibitem

\bibitem{HoffmanSountsov2022}
\begin{binproceedings}[author]
\bauthor{\bsnm{Hoffman},~\bfnm{Matthew~D.}\binits{M.~D.}} \AND
  \bauthor{\bsnm{Sountsov},~\bfnm{Pavel}\binits{P.}}
(\byear{2022}).
\btitle{Tuning-Free Generalized {H}amiltonian {M}onte {C}arlo}.
In \bbooktitle{Proceedings of the 25th International Conference on Artificial
  Intelligence and Statistics}.
\bseries{Proceedings of Machine Learning Research}
\bvolume{151}
\bpages{7799--7813}.
\bpublisher{PMLR}.
\end{binproceedings}
\endbibitem

\bibitem{Kifer1990}
\begin{barticle}[author]
\bauthor{\bsnm{Kifer},~\bfnm{Yuri}\binits{Y.}}
(\byear{1990}).
\btitle{A Discrete-Time Version of the {W}entzell--{F}riedlin Theory}.
\bjournal{Annals of Probability}
\bvolume{18}
\bpages{1676--1692}.
\bdoi{10.1214/aop/1176990641}
\end{barticle}
\endbibitem

\bibitem{KimGruffazParkDurmus2026}
\begin{barticle}[author]
\bauthor{\bsnm{Kim},~\bfnm{Kyurae}\binits{K.}},
  \bauthor{\bsnm{Gruffaz},~\bfnm{Samuel}\binits{S.}},
  \bauthor{\bsnm{Park},~\bfnm{Ji~Won}\binits{J.~W.}} \AND
  \bauthor{\bsnm{Durmus},~\bfnm{Alain~Oliviero}\binits{A.~O.}}
(\byear{2026}).
\btitle{Analysis of Kinetic {L}angevin {M}onte {C}arlo under the Stochastic
  Exponential {E}uler Discretization from Underdamped All the Way to
  Overdamped}.
\bjournal{Electronic Journal of Statistics}
\bvolume{20}
\bpages{3106--3142}.
\bdoi{10.1214/26-EJS2556}
\end{barticle}
\endbibitem

\bibitem{LeimkuhlerMatthews2013}
\begin{barticle}[author]
\bauthor{\bsnm{Leimkuhler},~\bfnm{Benedict}\binits{B.}} \AND
  \bauthor{\bsnm{Matthews},~\bfnm{Charles}\binits{C.}}
(\byear{2013}).
\btitle{Rational Construction of Stochastic Numerical Methods for Molecular
  Sampling}.
\bjournal{Applied Mathematics Research eXpress}
\bvolume{2013}
\bpages{34--56}.
\bdoi{10.1093/amrx/abs010}
\end{barticle}
\endbibitem

\bibitem{LeimkuhlerPaulinWhalley2024}
\begin{barticle}[author]
\bauthor{\bsnm{Leimkuhler},~\bfnm{Benedict~J.}\binits{B.~J.}},
  \bauthor{\bsnm{Paulin},~\bfnm{Daniel}\binits{D.}} \AND
  \bauthor{\bsnm{Whalley},~\bfnm{Peter~A.}\binits{P.~A.}}
(\byear{2024}).
\btitle{Contraction and Convergence Rates for Discretized Kinetic Langevin
  Dynamics}.
\bjournal{SIAM Journal on Numerical Analysis}
\bvolume{62}
\bpages{1226--1258}.
\bdoi{10.1137/23M1556289}
\end{barticle}
\endbibitem

\bibitem{LessardRechtPackard2016}
\begin{barticle}[author]
\bauthor{\bsnm{Lessard},~\bfnm{Laurent}\binits{L.}},
  \bauthor{\bsnm{Recht},~\bfnm{Benjamin}\binits{B.}} \AND
  \bauthor{\bsnm{Packard},~\bfnm{Andrew}\binits{A.}}
(\byear{2016}).
\btitle{Analysis and Design of Optimization Algorithms via Integral Quadratic
  Constraints}.
\bjournal{SIAM Journal on Optimization}
\bvolume{26}
\bpages{57--95}.
\bdoi{10.1137/15M1009597}
\end{barticle}
\endbibitem

\bibitem{LiZhaTao2022}
\begin{binproceedings}[author]
\bauthor{\bsnm{Li},~\bfnm{Ruilin}\binits{R.}},
  \bauthor{\bsnm{Zha},~\bfnm{Hongyuan}\binits{H.}} \AND
  \bauthor{\bsnm{Tao},~\bfnm{Molei}\binits{M.}}
(\byear{2022}).
\btitle{Hessian-Free High-Resolution {N}esterov Acceleration for Sampling}.
In \bbooktitle{Proceedings of the 39th International Conference on Machine
  Learning}.
\bseries{Proceedings of Machine Learning Research}
\bvolume{162}
\bpages{13125--13162}.
\bpublisher{PMLR}.
\end{binproceedings}
\endbibitem

\bibitem{Lu2026}
\begin{barticle}[author]
\bauthor{\bsnm{Lu},~\bfnm{Jianfeng}\binits{J.}}
(\byear{2026}).
\btitle{A Sharp Hypocoercive Entropy Decay Estimate for Underdamped Langevin
  Dynamics}.
\bjournal{arXiv preprint arXiv:2605.01933v2}.
\bdoi{10.48550/arXiv.2605.01933}
\end{barticle}
\endbibitem

\bibitem{MaChatterjiChengFlammarionBartlettJordan2021}
\begin{barticle}[author]
\bauthor{\bsnm{Ma},~\bfnm{Yi-An}\binits{Y.-A.}},
  \bauthor{\bsnm{Chatterji},~\bfnm{Niladri~S.}\binits{N.~S.}},
  \bauthor{\bsnm{Cheng},~\bfnm{Xiang}\binits{X.}},
  \bauthor{\bsnm{Flammarion},~\bfnm{Nicolas}\binits{N.}},
  \bauthor{\bsnm{Bartlett},~\bfnm{Peter~L.}\binits{P.~L.}} \AND
  \bauthor{\bsnm{Jordan},~\bfnm{Michael~I.}\binits{M.~I.}}
(\byear{2021}).
\btitle{Is There an Analog of {N}esterov Acceleration for Gradient-Based
  {MCMC}?}
\bjournal{Bernoulli}
\bvolume{27}
\bpages{1942--1992}.
\bdoi{10.3150/20-BEJ1297}
\end{barticle}
\endbibitem

\bibitem{Monmarche2021}
\begin{barticle}[author]
\bauthor{\bsnm{Monmarch\'{e}},~\bfnm{Pierre}\binits{P.}}
(\byear{2021}).
\btitle{High-Dimensional {MCMC} with a Standard Splitting Scheme for the
  Underdamped Langevin Diffusion}.
\bjournal{Electronic Journal of Statistics}
\bvolume{15}
\bpages{4117--4166}.
\bdoi{10.1214/21-EJS1888}
\end{barticle}
\endbibitem

\bibitem{Nesterov2004}
\begin{bbook}[author]
\bauthor{\bsnm{Nesterov},~\bfnm{Yurii}\binits{Y.}}
(\byear{2004}).
\btitle{Introductory Lectures on Convex Optimization: A Basic Course}.
\bseries{Applied Optimization}
\bvolume{87}.
\bpublisher{Kluwer Academic Publishers}, \baddress{Boston}.
\bdoi{10.1007/978-1-4419-8853-9}
\end{bbook}
\endbibitem

\bibitem{Polyak1964}
\begin{barticle}[author]
\bauthor{\bsnm{Polyak},~\bfnm{Boris~T.}\binits{B.~T.}}
(\byear{1964}).
\btitle{Some Methods of Speeding Up the Convergence of Iteration Methods}.
\bjournal{USSR Computational Mathematics and Mathematical Physics}
\bvolume{4}
\bpages{1--17}.
\bdoi{10.1016/0041-5553(64)90137-5}
\end{barticle}
\endbibitem

\bibitem{SchuhWhalley2025}
\begin{barticle}[author]
\bauthor{\bsnm{Schuh},~\bfnm{Katharina}\binits{K.}} \AND
  \bauthor{\bsnm{Whalley},~\bfnm{Peter~A.}\binits{P.~A.}}
(\byear{2025}).
\btitle{Convergence of Kinetic {L}angevin Samplers for Non-Convex Potentials}.
\bjournal{arXiv preprint arXiv:2405.09992v1}.
\bdoi{10.48550/arXiv.2405.09992}
\end{barticle}
\endbibitem

\bibitem{TrefethenBau1997}
\begin{bbook}[author]
\bauthor{\bsnm{Trefethen},~\bfnm{Lloyd~N.}\binits{L.~N.}} \AND
  \bauthor{\bsnm{Bau},~\bfnm{David}\binits{D.} \bsuffix{III}}
(\byear{1997}).
\btitle{Numerical Linear Algebra}.
\bpublisher{Society for Industrial and Applied Mathematics},
  \baddress{Philadelphia}.
\bdoi{10.1137/1.9780898719574}
\end{bbook}
\endbibitem

\bibitem{Villani2009}
\begin{bbook}[author]
\bauthor{\bsnm{Villani},~\bfnm{C\'{e}dric}\binits{C.}}
(\byear{2009}).
\btitle{Hypocoercivity}.
\bseries{Memoirs of the American Mathematical Society}
\bvolume{202}.
\bpublisher{American Mathematical Society}.
\bdoi{10.1090/S0065-9266-09-00567-5}
\end{bbook}
\endbibitem

\bibitem{ZuoOsherLi2025}
\begin{barticle}[author]
\bauthor{\bsnm{Zuo},~\bfnm{Xinzhe}\binits{X.}},
  \bauthor{\bsnm{Osher},~\bfnm{Stanley}\binits{S.}} \AND
  \bauthor{\bsnm{Li},~\bfnm{Wuchen}\binits{W.}}
(\byear{2025}).
\btitle{Gradient-Adjusted Underdamped {L}angevin Dynamics for Sampling}.
\bjournal{SIAM/ASA Journal on Uncertainty Quantification}
\bvolume{13}
\bpages{1735--1765}.
\bdoi{10.1137/24M1702015}
\end{barticle}
\endbibitem

\end{thebibliography}

\end{document}